\pdfoutput=1
\documentclass[10pt]{amsart} 
\usepackage{amsmath,amscd,amsthm,amssymb,mathrsfs,mathtools,bm} 
\usepackage{microtype,lmodern,textcomp} 
\usepackage{enumerate,comment,braket,xspace,tikz-cd} 
\usepackage[utf8]{inputenc} 
\usepackage[T1]{fontenc} 
\usepackage[centering,vscale=0.7,hscale=0.7]{geometry}
\usepackage[hidelinks]{hyperref}
\usepackage[capitalize]{cleveref}

\numberwithin{equation}{section}
\theoremstyle{plain}
\newtheorem{theorem}[equation]{Theorem}
\newtheorem{lemma}[equation]{Lemma}
\newtheorem*{lemma*}{Lemma}
\newtheorem{claim}[equation]{Claim}
\newtheorem*{claim*}{Claim}
\newtheorem{conjecture}[equation]{Conjecture}
\newtheorem{proposition}[equation]{Proposition}
\newtheorem*{proposition*}{Proposition}
\newtheorem{corollary}[equation]{Corollary}
\theoremstyle{definition}
\newtheorem{definition}[equation]{Definition}
\newtheorem{definition-theorem}[equation]{Definition-Theorem}
\newtheorem{definition-lemma}[equation]{Definition-Lemma}
\newtheorem{construction}[equation]{Construction}
\newtheorem{assumption}[equation]{Assumption}
\newtheorem{notation}[equation]{Notation}
\newtheorem{example}[equation]{Example}
\newtheorem{remark}[equation]{Remark}

\newif\ifpersonal

\ifpersonal
\newcommand{\personal}[1]{\textcolor[rgb]{0,0,1}{(Personal: #1)}}

\newcommand{\todo}[1]{\textcolor{red}{(Todo: #1)}}
\else
\newcommand{\personal}[1]{\ignorespaces}
\newcommand{\discussion}[1]{\ignorespaces}
\newcommand{\todo}[1]{\ignorespaces}
\fi

\providecommand{\abs}[1]{\lvert#1\rvert}

\newcommand{\bbA}{\mathbb A}

\newcommand{\bbC}{\mathbb C}
\newcommand{\bbD}{\mathbb D}

\newcommand{\bbG}{\mathbb G}

\newcommand{\bbN}{\mathbb N}
\newcommand{\bbP}{\mathbb P}
\newcommand{\bbQ}{\mathbb Q}
\newcommand{\bbR}{\mathbb R}

\newcommand{\bbZ}{\mathbb Z}

\newcommand{\fC}{\mathfrak C}

\newcommand{\fL}{\mathfrak L}

\newcommand{\fX}{\mathfrak X}
\newcommand{\fY}{\mathfrak Y}
\newcommand{\fZ}{\mathfrak Z}

\newcommand{\cC}{\mathcal C}
\newcommand{\cD}{\mathcal D}
\newcommand{\cE}{\mathcal E}

\newcommand{\cH}{\mathcal H}

\newcommand{\cL}{\mathcal L}
\newcommand{\cM}{\mathcal M}
\newcommand{\cN}{\mathcal N}
\newcommand{\cO}{\mathcal O}
\newcommand{\cP}{\mathcal P}

\newcommand{\cT}{\mathcal T}

\newcommand{\cV}{\mathcal V}

\newcommand{\cX}{\mathcal X}
\newcommand{\cY}{\mathcal Y}

\makeatletter
\let\save@mathaccent\mathaccent
\newcommand*\if@single[3]{%
	\setbox0\hbox{${\mathaccent"0362{#1}}^H$}%
	\setbox2\hbox{${\mathaccent"0362{\kern0pt#1}}^H$}%
	\ifdim\ht0=\ht2 #3\else #2\fi
}
\newcommand*\rel@kern[1]{\kern#1\dimexpr\macc@kerna}
\newcommand*\widebar[1]{\@ifnextchar^{{\wide@bar{#1}{0}}}{\wide@bar{#1}{1}}}
\newcommand*\wide@bar[2]{\if@single{#1}{\wide@bar@{#1}{#2}{1}}{\wide@bar@{#1}{#2}{2}}}
\newcommand*\wide@bar@[3]{%
	\begingroup
	\def\mathaccent##1##2{%
		\let\mathaccent\save@mathaccent
		\if#32 \let\macc@nucleus\first@char \fi
		\setbox\z@\hbox{$\macc@style{\macc@nucleus}_{}$}%
		\setbox\tw@\hbox{$\macc@style{\macc@nucleus}{}_{}$}%
		\dimen@\wd\tw@
		\advance\dimen@-\wd\z@
		\divide\dimen@ 3
		\@tempdima\wd\tw@
		\advance\@tempdima-\scriptspace
		\divide\@tempdima 10
		\advance\dimen@-\@tempdima
		\ifdim\dimen@>\z@ \dimen@0pt\fi
		\rel@kern{0.6}\kern-\dimen@
		\if#31
		\overline{\rel@kern{-0.6}\kern\dimen@\macc@nucleus\rel@kern{0.4}\kern\dimen@}%
		\advance\dimen@0.4\dimexpr\macc@kerna
		\let\final@kern#2%
		\ifdim\dimen@<\z@ \let\final@kern1\fi
		\if\final@kern1 \kern-\dimen@\fi
		\else
		\overline{\rel@kern{-0.6}\kern\dimen@#1}%
		\fi
	}%
	\macc@depth\@ne
	\let\math@bgroup\@empty \let\math@egroup\macc@set@skewchar
	\mathsurround\z@ \frozen@everymath{\mathgroup\macc@group\relax}%
	\macc@set@skewchar\relax
	\let\mathaccentV\macc@nested@a
	\if#31
	\macc@nested@a\relax111{#1}%
	\else
	\def\gobble@till@marker##1\endmarker{}%
	\futurelet\first@char\gobble@till@marker#1\endmarker
	\ifcat\noexpand\first@char A\else
	\def\first@char{}%
	\fi
	\macc@nested@a\relax111{\first@char}%
	\fi
	\endgroup
}
\makeatother

\newcommand{\oK}{\widebar K}

\newcommand{\oN}{\widebar N}

\newcommand{\oQ}{\widebar Q}

\newcommand{\oY}{\widebar Y}

\newcommand{\oDelta}{\widebar\Delta}

\newcommand{\hR}{\widehat R}

\newcommand{\hX}{\widehat X}

\newcommand{\tA}{\widetilde A}

\newcommand{\tC}{\widetilde C}

\newcommand{\tE}{\widetilde E}

\newcommand{\tH}{\widetilde H}

\newcommand{\tN}{\widetilde N}

\newcommand{\tX}{\widetilde X}
\newcommand{\tY}{\widetilde Y}

\newcommand{\tGamma}{\widetilde\Gamma}
\newcommand{\tDelta}{\widetilde\Delta}

\newcommand{\cTV}{\mathcal{TV}}
\newcommand{\dV}{\mathrm{dV}}
\DeclareMathOperator{\link}{link}
\DeclareMathOperator{\Sec}{Sec}
\DeclareMathOperator{\MovFan}{MovFan}
\DeclareMathOperator{\MovSec}{MovSec}

\newcommand{\bdry}{\mathrm{bdry}}
\newcommand{\interior}{\mathrm{int}}
\newcommand{\NN}{\mathit{NN}}
\newcommand{\tNN}{\widetilde{\mathit{NN}}}

\DeclareMathOperator{\Mov}{Mov}
\DeclareMathOperator{\Hilb}{Hilb}

\DeclareMathOperator{\Sing}{Sing}
\DeclareMathOperator{\Tors}{Tors}

\newcommand{\ocP}{\widebar{\cP}}
\newcommand{\bbk}{\Bbbk}
\newcommand{\hcX}{\widehat{\cX}}

\newcommand{\longto}{\longrightarrow}

\newcommand{\hgamma}{\widehat\gamma}

\newcommand{\tSigma}{\widetilde\Sigma}
\newcommand{\tsigma}{\widetilde\sigma}

\newcommand{\uGamma}{\underline\Gamma}
\newcommand{\uLambda}{\underline\Lambda}

\newcommand{\ocT}{\widebar{\cT}}
\newcommand{\ocM}{\widebar{\cM}}

\newcommand{\ff}{\mathfrak f}
\newcommand{\fm}{\mathfrak m}

\newcommand{\sfL}{\mathsf L}
\newcommand{\sfN}{\mathsf N}
\newcommand{\sfM}{\mathsf M}

\newcommand{\SP}{\mathbf{SP}}

\newcommand{\an}{\mathrm{an}}

\newcommand{\ess}{\mathrm{ess}}

\newcommand{\trop}{\mathrm{trop}}

\newcommand{\Gm}{\bbG_\mathrm{m}}
\newcommand{\Gmk}{\bbG_{\mathrm m/k}}

\newcommand{\Zaffine}{$\bbZ$-affine\xspace}
\newcommand{\inv}{^{-1}}

\newcommand{\kc}{k^\circ}

\DeclareMathOperator{\Eff}{Eff}

\DeclareMathOperator{\MoriFan}{MoriFan}
\DeclareMathOperator{\NE}{NE}
\DeclareMathOperator{\Nef}{Nef}
\DeclareMathOperator{\Pic}{Pic}
\DeclareMathOperator{\Proj}{Proj}
\DeclareMathOperator{\Sk}{Sk}
\DeclareMathOperator{\oSk}{\overline{Sk}}

\DeclareMathOperator{\Spec}{Spec}
\DeclareMathOperator{\Spf}{Spf}
\DeclareMathOperator{\Sp}{Sp}

\DeclareMathOperator{\Trop}{Trop}
\DeclareMathOperator{\TV}{TV}

\DeclareMathOperator{\ex}{ex}

\DeclareMathOperator{\val}{val}
\DeclareMathOperator{\Tor}{Tor}

\renewenvironment{abstract}{%
	\quotation
	\small
	\textbf{\textit{\abstractname.}} 
}{\endquotation}

\makeatletter
\@namedef{subjclassname@2020}{%
	\textup{2020} Mathematics Subject Classification}
\makeatother

\begin{document}
\title{Secondary fan, theta functions and moduli of Calabi-Yau pairs}

\author{Paul Hacking}
\address{Paul Hacking, Department of Mathematics and Statistics, Lederle Graduate Research Tower, University of Massachusetts, Amherst, MA 01003-9305, USA}
\email{hacking@math.umass.edu}
\author{Sean Keel}
\address{Sean Keel, Department of Mathematics, 1 University Station C1200, Austin, TX 78712-0257, USA}
\email{keel@math.utexas.edu}
\author{Tony Yue YU}
\address{Tony Yue YU, Department of Mathematics M/C 253-37, California Institute of Technology, 1200 E California Blvd, Pasadena, CA 91125, USA}
\email{yuyuetony@gmail.com}
\date{January 10, 2022}

\maketitle

\begin{abstract}
	We conjecture that any connected component $Q$ of the moduli space of triples $(X,E=E_1+\dots+E_n,\Theta)$ where $X$ is a smooth projective variety, $E$ is a normal crossing anti-canonical divisor with a 0-stratum, every $E_i$ is smooth, and $\Theta$ is an ample divisor not containing any 0-stratum of $E$, is \emph{unirational}.
	More precisely:  note that $Q$ has a natural embedding into the Kollár-Shepherd-Barron-Alexeev moduli space of stable pairs, we conjecture that the induced compactification admits a finite cover by a complete toric variety.
	We construct the associated complete toric fan, generalizing the Gelfand-Kapranov-Zelevinski secondary fan for reflexive polytopes.
	Inspired by mirror symmetry, we speculate a synthetic construction of the universal family over this toric variety, as the Proj of a sheaf of graded algebras with a canonical basis, whose structure constants are given by counts of non-archimedean analytic disks.
	In the Fano case and under the assumption that the mirror contains a Zariski open torus, we construct the conjectural universal family, generalizing the families of Kapranov-Sturmfels-Zelevinski and Alexeev in the toric case.
	In the case of del Pezzo surfaces with an anti-canonical cycle of $(-1)$-curves, we prove the full conjecture.
\end{abstract}

\tableofcontents

\section{Introduction and statement of results}

We propose the following conjecture regarding the moduli space of smooth polarized Calabi-Yau pairs:

\begin{conjecture} \label{conj:smooth}
	Any connected component $Q$ of the coarse moduli space of triples $(X,E=E_1+\dots+E_n,\Theta)$ where
	\begin{enumerate}
		\item $X$ is a connected smooth projective complex variety,
		\item $E\in \abs{-K_X}$ is a normal crossing divisor with a 0-stratum, every $E_i$ is smooth,
		\item $\Theta\subset X$ is an ample divisor not containing any 0-stratum of $E$,
	\end{enumerate}
	is unirational.
	\end{conjecture}

We have a more precise form of the above conjecture.
Note that in view of conditions (1-2), (3) is equivalent to the condition that for sufficiently small $\epsilon > 0$, $(X,E + \epsilon \Theta)$ is a stable pair\footnote{It is sometimes called KSBA stable pair in honor of the works of Kollár-Shepherd-Barron \cite{Kollar_Threefolds_and_deformations} and Alexeev \cite{Alexeev_Moduli_spaces_MgnW}.
	}
(see \cite[\S 5]{Kollar_Singularities_of_the_minimal_model_program}), thus $Q$ immerses into $\SP$, the moduli space of stable pairs (which is a higher-dimensional generalization of the moduli space $\ocM_{g,n}$ of stable pointed curves, see \cite{Kollar_Moduli_of_varieties_of_general_type}).
Let $\oQ$ denote the closure of $Q$ in $\SP$.

\begin{conjecture} \label{conj:main}
	There is a complete toric variety $T$ with a finite surjective map $T \to \oQ$.
\end{conjecture}

We prove \cref{conj:main} in the following special case:

\begin{theorem} \label{thm:del_Pezzo_intro}
	\cref{conj:main} holds when $X$ is a del Pezzo surface, $E\in\abs{-K_X}$ is a cycle of $K_X$-degree-$(-1)$ curves, and $\Theta\in\abs{-K_X}$.
\end{theorem}

We introduce a generalization of the Gelfand-Kapranov-Zelevinski secondary fan \cite{Gelfand_Discriminants} and conjecture $T$ to be the associated toric variety. 
Moreover, we conjecture a synthetic construction of the pullback family over $T$, generalizing the mirror construction of \cite{Keel_Yu_The_Frobenius} (cf.\ \cite{Gross_Mirror_symmetry_for_log_Calabi-Yau_surfaces_I_v1,Gross_Intrinsic_mirror_symmetry_announcement,Gross_Intrinsic_mirror_symmetry}), as the Proj of a sheaf of graded algebras with a canonical basis, whose structure constants are given by counts of non-archimedean analytic disks --- we refer to this throughout the paper as the \emph{mirror family}.
By construction, every fiber is endowed with a canonical \emph{theta function} basis of sections of every power of the polarization.
We prove these conjectures under the assumptions of \cref{thm:del_Pezzo_intro}.

Our speculations are led by considerations from mirror symmetry.
From the viewpoint of birational geometry, even the conjecture that $Q$ is \emph{uniruled} does not strike us as at all obvious.
Furthermore, we are predicting that there is a (nearly uni) versal family of stable pairs parametrized by an algebraic torus.
In dimension two, although there is an elementary construction of versal families of pairs $(X,E)$ in \cite{Gross_Moduli_of_surfaces}, it is without the divisor $\Theta$, and does not apply to the stable pair compactifications.
We do not know any elementary constructions in higher dimensions.

Our generalization of the GKZ secondary fan is a general construction, of independent interest, mixing ideas from Mori theory and Berkovich geometry, see \cref{sec:secondary_fan}.
When $X$ is Fano and $\Theta\in\abs{-K_X}$, our \emph{secondary fan} $\Sec$ is a complete rational polyhedral fan with support $\Pic(Y)_\bbR$, which is a coarsening of the Mori fan for the mirror Calabi-Yau pair $(Y,D)$.
Under the further assumption that $U \coloneqq Y \setminus D$ contains a Zariski open torus\footnote{This always holds if $Y$ is 2-dimensional; there are also many important higher dimensional cluster variety examples, e.g.\ an open Richardson variety inside a flag manifold, see \cite{Serhiyenko_Cluster_structures}.},
we build the mirror family of triples $(\cX,\cE,\Theta)$ over (the complete toric variety) $\TV(\Sec)$, by gluing together $\Proj(A_\alpha)$ for graded rings $A_\alpha$ constructed from counts of non-archimedean analytic disks in $Y^\an$ as in \cite[\S 1.1]{Keel_Yu_The_Frobenius} (based on ideas from \cite{Yu_Enumeration_of_holomorphic_cylinders_I,Yu_Enumeration_of_holomorphic_cylinders_II}).
We show that all the fibers $(X,E)$ are semi-log-canonical, and the generic fiber $(X,E + \epsilon \Theta)$ is stable for sufficiently small $\epsilon>0$, see \cref{prop:mirror_is_Fano}.
In the case when $Y$ is del Pezzo, we can prove more:

\begin{theorem}[see \cref{thm:stable}] \label{thmintro:stable}
	In the two-dimensional case, the following hold:
	\begin{enumerate}
		\item \label{thmintro:stable:family} The mirror family $(\cX,\cE) \to \TV(\Sec)$ is a flat family of semi-log-canonical pairs $(X,E)$ with $K_X+E$ trivial, and $H^i(X,\cO_X) =0$ for $i > 0$.
		\item \label{thmintro:stable:E} The boundary $\cE \to \TV(\Sec)$ is a trivial family, with fiber a cycle of rational curves.
		\item \label{thmintro:stable:restriction}
		For every fiber $(X,E)$ over the structure torus $T_{\Pic(Y)}\subset\TV(\Sec)$, $X$ is a del Pezzo surface with at worst du Val singularities, $E\subset X$ is an anti-canonical cycle of $K_X$-degree-$(-1)$ rational curves, and the self-intersection number of $K_X$ is equal to the number of irreducible components of $D$.
				\item \label{thmintro:stable:stable} For $0<\epsilon\ll 1$, $(\cX,\cE + \epsilon \Theta)\to\TV(\Sec)$ is a family of stable pairs.
		\item \label{thminto:stable:finite} The induced map $\TV(\Sec)\to\SP$ to the moduli space of stable pairs is finite.
	\end{enumerate}
\end{theorem}

Finally when $Y$ is del Pezzo and $D \subset Y$ is an anti-canonical cycle of $(-1)$-curves, we prove that the image of the finite map $\TV(\Sec)\to\SP$ is the full deformation space $\oQ\subset\SP$, where $Q$ is the moduli space in \cref{conj:smooth} for the pair $(Y,D)$, see Theorem \ref{thm:del_Pezzo}.

\begin{remark}
	Mirror symmetry suggests that deformation types of $(X,E)$ as in \cref{conj:smooth} come in dual pairs, generalizing the Batyrev duality via reflexive polytopes \cite{Batyrev_Dual_polyhedra_and_mirror_symmetry}. 
	The mirror to a pair $(Y,D)$ with $Y$ a smooth del Pezzo consists of pairs $(X,E)$ with $E$ a cycle of $K_X$-degree-$(-1)$ curves on a del Pezzo $X$ with an $A_{k-1}$ singularity at the node of $E$ corresponding to an irreducible component of $D$ of $K_X$-degree $-k$.
	So the mirror can only be smooth when $D\subset Y$ is a cycle of $K_X$-degree-$(-1)$ curves, which is why we specialize to this case in \cref{thm:del_Pezzo_intro}.
	It is possible to have an analog of \cref{thm:del_Pezzo_intro} where we allow such $A_{k-1}$ singularities, see \cref{rem:duval}.
\end{remark}

\begin{remark}
	A version of mirror algebra has been constructed by Gross-Siebert \cite{Gross_Intrinsic_mirror_symmetry} in much greater generality using counts of punctured log curves instead of non-archimedean analytic curves. 
	However, in order to produce a family over the complete toric base associated to the secondary fan, one must show that the multiplication rule of the mirror algebra depends \emph{only} on the log Calabi-Yau $U\coloneqq Y\setminus D$, not on the compactification $U\subset Y$, (up to change of curve classes in the coefficients).
	This is easy from the non-archimedean approach, because the (punctured) analytic disks that contribute to the structure constants live inside the log Calabi-Yau $U$; however, this is currently unknown (and not at all clear) for the punctured log curve approach of Gross-Siebert, because the log curves can have components mapping completely into the boundary $D$).
	For the moment the non-archimedean construction of mirror algebra has only been carried out under the assumption of containing a Zariski open torus.
	As soon as the non-archimedean mirror construction generalizes to all affine log Calabi-Yau varieties with maximal boundary, the construction of the universal family over the toric variety associated to the secondary fan in this paper can also be extended to the general case of Fano varieties.
	\cref{sec:conjectural_construction} contains a speculative discussion about how one might approach \cref{conj:main} in full generality.
\end{remark}

Now let us give a more detailed overview of what we do in this paper:

\smallskip
As a guide let us first consider briefly the case where $(Y,D)$ is toric Fano.
Let $M$ denote the cocharacter lattice of the torus $U\coloneqq Y\setminus D$.
Since $Y$ is Fano, the first lattice points on the rays of the dual fan $\Sigma_{(Y,D)}$
are the vertices of a reflexive convex lattice polytope $\Lambda \subset M_\bbR$, in particular $0 \in \Lambda$ is the unique lattice point (see \cite{Batyrev_Dual_polyhedra_and_mirror_symmetry}).
The mirror to $(Y,D)$ is the polarized toric Fano $(X,E)$ given by $\Lambda$.
In this case, \cref{conj:main} can be deduced from Kapranov-Sturmfels-Zelevinsky \cite{Kapranov_Quotients_of_toric_vareities} and Alexeev \cite[1.2.15]{Alexeev_Complete_moduli}.
The finite map $T\to\oQ$ is a normalization, and $T$ is the toric variety associated to the GKZ secondary fan, which contains a cone for each tiling $\uLambda$ of $\Lambda$ given by a convex piecewise affine function.
The theta function basis for the homogeneous coordinate ring of the mirror family over $T$ consists of monomials corresponding to the lattice points in the cone $\Gamma$ over $\Lambda$.
The multiplication rule in the basis is simply the toric one: $\theta_P \cdot \theta_Q = \theta_{P + Q}$, addition in the ambient lattice $M\oplus\bbZ\supset\Gamma(\bbZ)$.

Now consider a general pair $(Y,D)$ with $Y$ smooth Fano and $D\subset Y$ a normal crossing anti-canonical divisor containing a 0-stratum.
Let $p\colon K \to Y$ denote the canonical bundle, and $K\to\oK$ the contraction of the 0-section.
The divisor $D \in \abs{-K_Y}$ gives a function $d\colon K \to \bbA^1$.
The central fiber $D_K \coloneqq d^{-1}(0)\subset K$ has normal crossings; we denote its dual complex by $\uLambda$.
Note that the complement $V \coloneqq K \setminus D_K$ is log Calabi-Yau, and $p\times d \colon V \to U \times \Gm$ is an isomorphism, where $U \coloneqq Y \setminus D$.
Let $\uGamma$ denote the cone complex over $\uLambda$.
The generalization of the ambient lattice $M \oplus \bbZ \subset M_\bbR \oplus \bbR$ is
\[
\Sk(V,\bbZ) \subset \Sk(V) \subset V^\an, \]
where $V^\an$ denotes the Berkovich analytification with respect to the trivial valuation on $\bbC$, $\Sk(V)$ the Kontsevich-Soibelman essential skeleton, and $\Sk(V,\bbZ)$ the integer points inside (see \cite[\S 2]{Keel_Yu_The_Frobenius} for a quick description of these objects).

Note $K/\oK$ is a relative Mori dream space, and Mori theory provides a complete fan $\MoriFan$, with support $\Pic(K/\oK)_\bbR \simeq \Pic(Y)_\bbR$.
It contains a maximal cone $\Nef(K') \subset \Pic(K/\oK)_\bbR$ for each flop $K \dasharrow K'$ over $\oK$,
as well as bogus cones (see \cref{sec:moving}).
Note $K' \dasharrow \bbA^1$ is again regular,
and the dual complex of the central fiber gives another triangulation $\uLambda'$ of the underlying topological space $\Lambda$ of $\uLambda$.
In the toric case, $\MoriFan$ is isomorphic to the GKZ secondary fan.
However, in general MoriFan gives the \emph{wrong} base for the mirror family:
Although the mirror family extends over the associated toric variety $\TV(\MoriFan)$, the induced map to the moduli space $\SP$ of stable pairs is not in general finite.
For example, when $Y$ is del Pezzo, the exceptional locus of $K\dasharrow K'$ is a disjoint union of $(-1)$-curves in $Y\subset K$.
Consider the case of a single exceptional $(-1)$-curve $C$, then $\Nef(K)$ and $\Nef(K')$ are adjacent maximal cones of $\MoriFan$ meeting along a codimension-one face, determining a boundary 1-stratum $S \subset \partial \TV(\MoriFan)$.
If $C$ is \emph{internal}, i.e.\ not an irreducible component of $D \subset Y$, then the restriction of the mirror family to $S$ will be trivial, and the map $\TV(\MoriFan)\to\SP$ contracts $S$ to a point.
In order to remedy this problem, we observe in this case that the triangulations $\uLambda$ and $\uLambda'$ are the same.
So the rough idea is to coarsen the Mori fan by gluing cones associated with same tilings of $\Lambda$.
We call the resulting fan the \emph{secondary fan}, denote it by $\Sec$, and prove that the mirror family over (the toric variety associated to) the Mori fan indeed descends to the secondary fan.
The precise construction of the secondary fan relies on Berkovich geometry and Mori theory, and is quite different from the Gelfand-Kapranov-Zelevinski toric construction (see \cref{sec:secondary_fan}, see also \cref{rem:Sec_on_PicY} for an alternative explicit description in the del Pezzo case).

The mirror family over the toric base $\TV(\Sec)$ is constructed in \cref{sec:family_over_secondary_fan}.
We start from the family over the nef cone $\Nef(K)$: as in the toric case, the associated mirror algebra $A$ has a basis parametrized by the integer points $\Gamma(\bbZ)$, but the multiplication rule is much more subtle, given by counts of non-archimedean analytic disks in $K^\an$.
The tropicalization $d^\trop\colon \Sk(V,\bbZ)\to\bbZ$ gives a grading on $A$, then $\Proj A$ gives the mirror family $\cX$ over the nef cone;
the boundary $\cE\subset\cX$ is given by the ideal generated by the integer points $\Gamma^\circ(\bbZ)$ in the interior of $\Gamma$, and the divisor $\Theta$ is the zero-locus of the sum of sections associated to the integer points $\Lambda(\bbZ)$ in $\Lambda$.
We carry out the same construction for every flop $K\dasharrow K'$ over $\oK$;
then by rephrasing the mirror algebras using universal torsor as in \cref{sec:rephrase},
the mirror families for various $K'$ glue together to a family over the moving part of the secondary fan, see \cref{sec:extension_moving}.
The extension to the full secondary fan is carried out in \cref{sec:extension_full}, which relies on the equivariant boundary torus action of \cref{sec:torus_action} and the toric fiber bundle construction of \cref{sec:toric_fiber_bundle}.

We study the singularities in \cref{sec:singularities}.
We show that the mirror family $(\cX,\cE)$ over the nef cone is a family of semi-log-canonical Fano varieties with log-canonical generic fiber, moreover for sufficiently small $\epsilon>0$, the generic fiber $(X,E+\epsilon\Theta)$ is stable, thus the family determines at least a rational map $\TV(\Sec) \dasharrow \SP$, see \cref{prop:mirror_is_Fano}.
We conjecture all the fibers are stable, and that the induced regular map to $\SP$ is finite.
We have much better control on singularities in the del Pezzo case, see \cref{thmintro:stable} (or \ref{thm:stable}).
Our proof uses the equivariant boundary torus action to push any fiber $(X,E)$ to the fiber over a 0-stratum of the toric base, which admits explicit geometric descriptions.
Note that over an orbit of the boundary torus action, $(X,E)$ is constant, but $\Theta$ varies, by arbitrary scaling of the non-zero coefficients: the weights for the theta function basis of $\Gamma(X,\cO(1))$ are exactly the coordinates on the boundary torus.
Thus the condition that $\Theta$ does not pass any log-canonical center of $(X,E)$ has a striking implication: all but exactly one of the (degree-one) theta functions must vanish at each log-canonical center, see \cref{cl:center_claim}, because otherwise by scaling we can move $\Theta$ so that it would contain a log-canonical center.

Finally we note one pleasant feature of our construction, that points up how desirable it is to find the {\it correct} fan:
In general, to prove modularity one would expect to need some deformation theory, or as in \cite{Gross_K3}, a Torelli theorem and computation of periods.
However, note that a map from a complete toric variety to a scheme is finite if and only if no boundary 1-stratum is contracted to a point, see \cref{lem:finite_map_from_toric_variety}.
While our mirror family built from non-archimedean enumerative geometry is quite complicated (it is after all, versal), the restrictions over 1-strata of the base toric variety are vastly simpler.
So we are able to show that they are non-trivial (i.e.\ no 1-stratum is contracted by the map to $\SP$), by exhibiting the smoothing of either a double curve or a 0-dimensional log-canonical center, see \cref{cl:restriction_to_1-stratum}.
Once we have the finite map $\TV(\Sec)\to\SP$, we show that its image is the full deformation space $\oQ$ (notation as in \cref{thm:del_Pezzo_intro}) by dimension count, see \cref{thm:del_Pezzo}.

\begin{remark}
	In the del Pezzo case, the mirror family $(\cX,\cE,\Theta) \to \TV(\Sec)$ is equivariant with respect to a natural (finite) Weyl group $W$ (see \cite{Gross_Moduli_of_surfaces}) and the map to $\SP$ factors through the quotient.
	We expect the induced finite map $\TV(\Sec)/W \to \SP$ is the normalization of its image; so in particular, in this case the normalization of $\oQ$ from \cref{thm:del_Pezzo_intro} is a quotient of a toric variety by a finite group.
	This is very special, as irreducible components of $\SP$ can be very complicated, even for the pair of $\bbP^2$ and a collection of general lines, see \cite[1.3, 3.13]{Keel_Geometry_of_Chow_quotients}.
	We are very interested in the question of how to generalize this Weyl group to higher dimensions.
\end{remark}

\bigskip \paragraph{\textbf{Acknowledgments}}
We enjoyed fruitful conversations with P.\ Achinger, V.\ Alexeev, M.\ Baker, V.\ Berkovich, F.\ Charles, A.\ Corti, A.\ Durcos, W.\ Gubler, J.\ Koll\'ar, M.\ Porta, J.\ Rabinoff, D.\ Ranganathan, M.\ Robalo and Y.\ Soibelman.
We were heavily inspired and influenced by our long-term collaborations with M.\ Gross, M.\ Kontsevich and B.\ Siebert.
Hacking was supported by NSF grants DMS-1601065 and DMS-1901970.
Keel was supported by NSF grant DMS-1561632.
T.Y.\ Yu was supported by the Clay Mathematics Institute as Clay Research Fellow.
Some of the research was conducted while Keel and Yu visited the Institute for Advanced Study in Princeton, and some while the three authors visited the Institut des Hautes Études Scientifiques in Bures-sur-Yvette.

\section{The secondary fan} \label{sec:secondary_fan}

In this section, we give a generalization of the Gelfand-Kapranov-Zelevinski secondary fan \cite{Gelfand_Discriminants}.
Here is the basic idea:
Consider a $\bbQ$-factorial compactification of a log Calabi-Yau variety $V\subset K$, with $\Pic(K)$ a lattice.
Let $G\to K$ be the universal torsor.
The restriction $\cT\coloneqq G|_V$ is again log Calabi-Yau, hence we obtain a map between essential skeletons $\Sk(\cT)\to\Sk(V)$, which admits a canonical section $\varphi\colon\Sk(V)\to\Sk(\cT)$, see \cref{sec:varphi}.
Then the secondary fan $\Sec(K)$ restricted to the cone $\Mov(K)\subset\Pic(K)_\bbR$ of moving divisors, is the coarsening of the Mori fan, whose maximal cones are unions of $\Nef(K')$ over all flops $K\dasharrow K'$ for which the sections $\varphi'$ coincide.
We prove that such unions are convex, see \cref{thm:convex_cone}.
A similar gluing for bogus cones is worked out in \cref{sec:bogus}.
An explicit description of the secondary fan in the del Pezzo case is given in \cref{sec:secondary_fan_del_Pezzo}.

\subsection{Construction of the section \texorpdfstring{$\varphi$}{φ}} \label{sec:varphi}

Let $k$ be a non-archimedean field (trivial valuation allowed), and $X$ a scheme locally of finite type over $k$.
We can analytify $X$ in the sense of Berkovich \cite{Berkovich_Spectral_theory} and obtain a $k$-analytic space $X^\an$.
The underlying topological space of $X^\an$ has a very simple description, which we recall for readers' convenience:

As a set, $X^\an$ consists of pairs $(\xi,\abs{\cdot})$, where $\xi\in X$ is a scheme point, and $\abs{\cdot}$ is an absolute value on the residue field $\kappa(\xi)$ extending the one on $k$.
The topology on $X^\an$ is the weakest one such that the forgetful map $\pi\colon X^\an\to X$ is continuous, and that for all Zariski open $U\subset X$, $f\in\cO_X(U)$, the function \[ \pi^{-1}(U)\longrightarrow\bbR_{\ge 0},\quad (\xi,\abs{\cdot})\longmapsto\abs{f(\xi)}\]
is continuous.

When $k$ has trivial valuation, we have a canonical analytic subdomain $X^\beth\subset X^\an$, consisting of $(\xi,\abs{\cdot})$ such that the center of the absolute value lies inside $X$, see \cite[Definition 1.3]{Thuillier_Geometrie_toroidale}.
The inclusion is an equality when $X$ is proper.

\begin{construction} \label{const:varphi}
	Let $\fX, \fY$ be formal schemes locally of finite presentation over the ring of integers $\kc$.
	Let $f\colon\fX\to\fY$ be a flat morphism with geometrically integral fibers.
	We can construct a canonical set-theoretic section $\varphi\colon\fY_\eta\longto\fX_\eta$ as follows:
	
	For every $y\in\fY_\eta$, let $\cH(y)$ denote its complete residue field.
	Let $\fZ$ be the pullback of $\fX$ along $\Spf\cH(y)^\circ\to\fY$.
	By assumption, it is an admissible formal scheme over $\cH(y)^\circ$ with integral special fiber $\fZ_s$, and generic fiber $\fZ_\eta\simeq(\fX_\eta)_y$.
	Let $\pi\colon\fZ_\eta\to\fZ_s$ be the reduction map (see \cite[\S 1]{Berkovich_Vanishing_cycles_for_formal_schemes}).
		By \cite[2.4.4(ii)]{Berkovich_Spectral_theory}, the preimage of the generic point of $\fZ_s$ by $\pi$ gives a unique point in $\fZ_\eta\simeq(\fX_\eta)_y$, which we take to be the image $\varphi(y)\in(\fX_\eta)_y\subset\fX_\eta$.
	
	In particular, we note the following special case when $k$ has trivial valuation:
	Let $f\colon X\to Y$ be a flat map of $k$-varieties with geometrically integral fibers.
	Then we have a canonical set-theoretic section \[\varphi\colon Y^\beth\longto X^\beth\subset X^\an.\]
\end{construction}

\begin{assumption} \label{ass:1}
	From now on we assume that $k$ has trivial valuation.
		Fix a smooth log Calabi-Yau variety $V$ with volume form $\omega_V$ (unique up to scaling).
	Fix a partial compactification $V\subset K$, with $K$ normal.
	In our application $K$ will be the total space of the canonical bundle of a smooth variety, whence our choice of letter.
\end{assumption}

\begin{construction} \label{const:universal_torsor}
	Given any lattice $N\subset\Pic(K)$, denote $M\coloneqq N^*$ and consider the universal torsor
	\[\pi\colon G\coloneqq\underline{\Spec}\bigg(\bigoplus_{L\in N} L\bigg)\longto K,\]
	which is a principal 
	\[T^N\coloneqq T_M\coloneqq M\otimes_\bbZ \Gmk = \Spec(k[N])\]
	bundle, and the action of $T^N$ is given by the $N$-grading.
	Denote $\cT\coloneqq G|_V$.
	The torus bundle has a canonical relative volume form, wedging with $\omega_V$ gives a volume form on $\cT$, making $\cT$ also log Calabi-Yau.
	Let $\Sk(V)\subset V^\an$ and $\Sk(\cT)\subset\cT^\an$ denote the essential skeletons (see \cite[\S 2]{Keel_Yu_The_Frobenius}), and let $\Gamma\coloneqq\Sk(V)\cap K^\beth$, $\tGamma\coloneqq\Sk(\cT)|_\Gamma$.
	Applying \cref{const:varphi}, we obtain $\varphi\colon K^\beth\to G^\beth$, a section of the projection.
	It follows from the construction that $\varphi$ restricts to	$\varphi\colon\Gamma\to\tGamma$, a section of $\tGamma\to\Gamma$.
\end{construction}

\begin{remark} \label{rem:varphi_resolution}
	Given any other compactification $V\subset K'$ with $K'\to K$ proper, consider the pullback $G'$ of $G$, and apply the above construction.
	We have $\cT'=\cT$, $\Sk(\cT')=\Sk(\cT)$, $\Gamma'=\Gamma$, $\tGamma'=\tGamma$ and $\varphi'=\varphi$.
	So for computing $\varphi$ one can always resolve $K$.
\end{remark}

\begin{remark}
	The $N_1(K,\bbR)$-bundle $\tGamma\to\Gamma$ and the section $\varphi\colon\Gamma\to\tGamma$ are basic objects in several previous works on mirror symmetry \cite{Gross_Mirror_symmetry_for_log_Calabi-Yau_surfaces_I_v1,Gross_K3,Gross_Mirror_symmetry_via_logarithmic_degeneration_data_I}, arising in a (it seems to us) rather ad hoc way. 
	As far as we know, the above canonical non-archimedean theoretic description is new.
\end{remark}

\begin{remark} \label{rem:varphiBT_explicit}
	Let us give a more explicit description of $\varphi\colon\Gamma\to\tGamma$.
	By \cref{rem:varphi_resolution} we may assume that $(K,D\coloneqq K\setminus V)$ is snc.
	Let $D^{\ess} \subset D$ be the union of essential divisors, i.e.\ irreducible components where $\omega_V$ has a pole, and $\Sigma \coloneqq \Sigma_{(K,D^{\ess})}$ the dual cone complex.
	Let $S_1\dots,S_m \in N \subset \Pic(K)$ be a line bundle basis so that $G \simeq L_1^\times \times \dots L_m^\times$, where $L_i$ denotes the dual of $S_i$, and $L_i^\times \subset L_i$ denotes the complement of the zero section.
	By \cite[Lemma 8.5]{Keel_Yu_The_Frobenius}, we have a canonical identification $\Gamma=\Sk(V)\cap K^\beth\simeq |\Sigma|$.
	
	Let $\pi\colon\ocT\coloneqq \bbP(\oplus_i L_i \oplus \cO)\to K$.
	Similarly we have
	$$
	\tGamma=\Sk(\cT)|_\Gamma=\Sk(\cT)\cap\ocT^\beth\simeq\Sigma_{(\ocT,(\ocT \setminus \cT)^{\ess})} \eqqcolon \tSigma.
	$$
	Note we have a natural inclusion of cone complexes $\Sigma \subset \tSigma$, the subcomplex generated by the rays corresponding to the irreducible components of $\pi^{-1}(D^{\ess})$;
	moreover, we have a natural projection $\tSigma \to\Sigma$.
	These coincide with $\varphi\colon\Gamma\to\tGamma$ and $\pi\colon\tGamma\to\Gamma$ respectively, via the identifications with the skeletons.
\end{remark}

\subsection{Moving cones} \label{sec:moving}

We recall some notions of birational geometry from \cite{Hu_Mori_dream_spaces}.
A \emph{small $\bbQ$-factorial modification} (SQM for short) is a birational map $f\colon X\dasharrow X'$ of normal $\bbQ$-factorial projective varieties that is an isomorphism in codimension 1.
A normal $\bbQ$-factorial projective variety $X$ is called a \emph{Mori dream space} if
\begin{enumerate}
	\item $\Pic(X)$ is finitely generated,
	\item $\Nef(X)$ is generated by finitely many semi-ample divisors,
	\item there is a finite collection of SQMs $f_i\colon X\dasharrow X_i$, such that each $X_i$ satisfies (2), and the cone of moving divisors $\Mov(X)$ is the union of $f_i^*(\Nef(X_i))$.
\end{enumerate}

A Mori dream space $X$ has a simple Mori chamber decomposition via Mori equivalence of line bundles, giving rise to a finite polyhedral fan $\MoriFan(X)$, called \emph{Mori fan}, which is supported on the cone of effective divisors $\Eff(X)$ (see \cite[Proposition 1.11]{Hu_Mori_dream_spaces}).
Each maximal cone of $\MoriFan(X)$ is of the form
\[f^*\Nef(Y)+\braket{\ex(f)},\]
for a birational contraction $f\colon X\dasharrow Y$ with $Y$ a Mori dream space, where $\braket{\ex(f)}$ denotes the subcone of $\Eff(X)$ spanned by $f$-exceptional effective divisors.

We denote by $\MovFan(X)$ the restriction of $\MoriFan(X)$ to the cone of moving divisors $\Mov(X)$, and call it the \emph{moving fan}.
Each of its maximal cones is of the form $f^*\Nef(X')$ for an SQM $f\colon X\dasharrow X'$.
The maximal cones of $\MoriFan(X)$ not contained in $\Mov(X)$ are called \emph{bogus cones}.

\begin{assumption} \label{ass:2}
In addition to \cref{ass:1}, we fix regular proper map $q\colon K \to \oK$, and assume $K/\oK$ is a relative Mori dream space.
Given an SQM $K \dasharrow K'$ over $\oK$, we have a canonical identification $\Pic(K)_\bbQ \simeq \Pic(K')_\bbQ$.
Whenever we consider various cones, e.g.\ $\NE$, $\Nef$, and fans, e.g.\ $\MoriFan$, for $K$ or $K'$, we will always mean relative to $\oK$.
Note that for sufficiently divisible $n$, $n\Pic(K') \subset \Pic(K')$ is free, and $n \Pic(K') \subset \Pic(K)_\bbQ$ is independent of the SQM.
We fix such an $\sfN\coloneqq n \Pic(K') \subset \Pic(K)_\bbQ$ throughout the paper, and let $\sfM$ denote its dual.
We have a canonical projection $N_1(K',\bbZ) \to \sfM$, and write $\NE(K')_\sfM\subset\sfM$ for the image of $\NE(K',\bbZ)$.
For our later application to \cref{thm:del_Pezzo_intro}, all SQM $K'$ will be smooth and we will simply take $n=1$ and $\sfN\coloneqq\Pic(K)$.
\end{assumption}

\begin{construction} \label{const:varphi_of_SQM}
	Applying \cref{const:universal_torsor} to $K$ and $\sfN\subset\Pic(K)$, we obtain $G\to K$, $\tGamma\to\Gamma$ and a section $\varphi\colon\Gamma\to\tGamma$.
	Given an SQM $f\colon K\dasharrow K'$ over $\oK$, applying \cref{const:universal_torsor} to $K'$ and $\sfN\subset\Pic(K')$, we obtain $G'\to K'$, $\tGamma'\to\Gamma'$ and a section $\varphi'\colon\Gamma'\to\tGamma'$.
	Since the volume forms agree, and $q\colon K\to\oK$, $q'\colon K'\to\oK$ are proper, we have $\Gamma=\Gamma'$ and $\tGamma=\tGamma'$.
	So $\varphi'\colon\Gamma'\to\tGamma'$ gives also a section of $\tGamma\to\Gamma$, which is in general different from $\varphi$.
\end{construction}

\begin{proposition} \label{prop:SQM}
	Let $f\colon K_1\dasharrow K_2$ be an SQM over $\oK$ corresponding to two maximal cones of $\MovFan(K)$.
	The following hold:
	\begin{enumerate}
	\item \label{prop:SQM:difference} The difference $\varphi_{K_2} - \varphi_{K_1}\colon \Gamma \to \sfM_\bbR$ has image in $\NE(K_1)$, (note the difference is between two sections of a principal $N_1(K,\bbR)$-bundle, hence an $N_1(K,\bbR)$-valued function.)
	\item \label{prop:SQM:Lperp} Let $L\in\Nef(K_1)\cap\Nef(K_2)$, then $\varphi_{K_2} - \varphi_{K_1}$ has image in $L^\perp$.
	\end{enumerate}
\end{proposition}

The proposition will follow from a more precise statement:

\begin{lemma} \label{lem:SQM}
	Choose a $\bbQ$-factorial projective variety $Z$ with birational morphisms $p_i\colon Z\to K_i$.
	Let $L \in \Pic(K)_\bbQ$, and write $p_1^*L = p_2^*L \otimes \cO(E)$ for a $\bbQ$-Cartier divisor $E$ supported on the $p_i$-exceptional loci (since $f$ is small, $p_1$ and $p_2$ have the same exceptional divisors).
	Then
	\[
	(\cdot L)\circ (\varphi_{K_1} - \varphi_{K_2}) = E^\trop \colon \Gamma \longto \bbR,
	\]
	where $E^\trop$ is given by taking valuation of local defining equations of $E$, see \cite[Construction 15.1]{Keel_Yu_The_Frobenius}.
\end{lemma}

\begin{proof}
	Put $\varphi_i \coloneqq (\cdot L) \circ \varphi_{K_i}$.
	Let $v\in\Gamma$, $\cH(v)$ its complete residue field, and $\cH(v)^\circ$ the valuation ring.
		By the properness of $Z$, we get a map $\Spec\cH(v)^\circ\to Z$.
	Now for computing $(\varphi_1 - \varphi_2)(v)$ we can replace $Z$ by $\Spec\cH(v)^\circ$.
	Furthermore, we can replace $\Pic(K)$ by the sublattice generated by $L$.
	We can also assume that $E$ is effective Cartier.
	Let $e\in\cH(v)$ be a defining equation for (the pullback of) $E$.
	Choose a trivialization, i.e.\ a nowhere vanishing section $s$, of $p_2^*L$ and write $p_2^*L\simeq\Spec\cH(v)^\circ[X_2]$.
	By $p_1^*L = p_2^*L \otimes \cO(E)$, $es$ gives a trivialization of $p_1^*L$, so we can write $p_1^*L\simeq\Spec\cH(v)^\circ[X_1]$ where $X_1=X_2/e$ over the generic fiber.
	By construction, $\varphi_{K_2}(v)$ is the Gauss point on the generic fiber with respect to the coordinate $X_2$, while $\varphi_{K_1}(v)$ is the Gauss point with respect to the coordinate $X_1=X_2/e$.
	Therefore,
	\begin{multline*}
	\varphi_1(v)-\varphi_2(v)=\val X_2(\varphi_{K_1}(v))-\val X_2(\varphi_{K_2}(v))=\val X_2(\varphi_{K_1}(v))\\
	=\val eX_1(\varphi_{K_1}(v))=\val X_1(\varphi_{K_1}(v))+\val e=\val e=E^\trop(v),
	\end{multline*}
	where $\val$ denotes the valuation on $\cH(v)^\circ$.
	This completes the proof.
\end{proof}

\begin{remark} \label{rem:internal_flop}
	Notation as in the proof of \cref{lem:SQM}, if $\Spec\cH(v)^\circ\to K_i$ has image in the open subset where $K_1 \dasharrow K_2$ is an isomorphism, then $\varphi_1(v) = \varphi_2(v)$.
	Note this holds for all $v\in\Gamma$ if the exceptional loci of $K_1 \dasharrow K_2$ contain no strata of the boundary.
\end{remark}

\begin{proof}[Proof of \cref{prop:SQM}]
(2) follows from (1).
For (1), we pick $L \in \Nef(K_1)$.
Notation as in \cref{lem:SQM}, we see that $E$ is $p_2$-nef.
So $-E$ is effective by the negativity lemma (see \cite[Lemma 3.39]{Kollar_Birational_geometry_of_algebraic_varieties}).
Thus $-E^\trop$ is non-negative.
So by \cref{lem:SQM}, $(\cdot L)\circ(\varphi_{K_2}-\varphi_{K_1})=-E^\trop$ is non-negative.
Now the result follows from the duality between $\NE(K_1)$ and $\Nef(K_1)$.
\end{proof}

\begin{theorem} \label{thm:convex_cone}
	Let $\alpha\subset\MovFan(K)$ be a maximal cone.
	Let $\sec(\alpha)$ be the union of maximal cones $\beta$ with $\varphi_\alpha=\varphi_\beta$.
	Then $\sec(\alpha)$ is a convex cone.
	Consequently, the collection of all such $\sec(\alpha)$ are the maximal cones of a rational polyhedral fan, denoted by $\MovSec(K)$, which is a coarsening of $\MovFan(K)$.
\end{theorem}

\begin{proof}
	It suffices to show that for any two maximal cones $\alpha,\beta\subset\MovFan(K)$ with $\varphi_\alpha=\varphi_\beta$, and any interval $f\colon[0,1]\to\MovFan(K_\gamma)$ from a general point of $\alpha$ to a general point of $\beta$, the whole image is contained in $\sec(\alpha)$.
	
	Since the interval is general, it intersects lower dimensional cones at finitely many points; so we obtain a partition $0=x_0<x_1<\dots<x_n=1$.
	Denote $I_i\coloneqq[x_i,x_{i+1}]$, and $\varphi_i\coloneqq\varphi_{K_i}$, where $K\to K_i$ is the SQM associated to the unique maximal cone containing $f(I_i)$.
	Fix any $b\in\Gamma$ and consider $d_{i,j}\coloneqq\varphi_i(b)-\varphi_j(b)$.
	For any $y\in\Pic(K)_\bbR$, write $d_{i,j,y}\coloneqq (\cdot y)\circ d_{i,j}$.
	
	Note $d_{i,j,f(x)}$ is linear on the interval $[0,1]$.
	Assume $i\le j$.
	By \cref{prop:SQM}(\ref{prop:SQM:difference}), $d_{i,j,f(x)}$ is non-positive for $x\in I_i$ and non-negative for $x\in I_j$.
	By linearity, this function is non-decreasing, and non-positive for all $x\in [0,x_{i+1}]$, and non-negative for all $x\in[x_j,1]$.
	
	Given any $i$ and $x\in I_i^\circ$, we claim that $d_{0,i,f(x)}=0$.
	Suppose to the contrary, by the above paragraph, we have $d_{0,i,f(1)}>0$ and $d_{i,n,f(1)}\ge 0$.
	So
	\[d_{0,n,f(1)}=d_{0,i,f(1)}+d_{i,n,f(1)}>0,\]
	a contraction to the assumption that $\varphi_\alpha=\varphi_\beta$.
	
	Since $f(1)$ is a general point of $\beta$, by perturbing $f(1)$, we deduce that $d_{0,i,y}=0$ for $y$ in a small neighborhood of $f(x)$; hence $d_{0,i}=0$.
	We conclude that $\varphi_i=\varphi_0$, completing the proof.
\end{proof}

\begin{lemma} \label{lem:recover_dual_complex}
	Let $K\dasharrow K'$ be an SQM over $\oK$, if $\varphi=\varphi'$, and both pairs $(K,V^c)$ and $(K',V^c)$ are dlt, then the dual complexes of $V\subset K$ and $V\subset K'$ coincide.
\end{lemma}
\begin{proof}
	Let us show that the dual complex of $V\subset K$ can be recovered from $\varphi$.
	Since we are only considering SQMs, the set of vertices of the dual complex is fixed; we only need to recover the information of which sets of boundary components have non-empty intersection.
	Let $\sfN=n\Pic(K)$ be as in \cref{ass:2}.
	Following Constructions \ref{const:varphi_of_SQM} and \ref{const:universal_torsor}, we have
	\[H^0(G,\cO) \simeq \bigoplus_{L \in \sfN} H^0(K,L).\]
	For each irreducible component $E\subset D$, the canonical section $1_{nE} \in H^0(K,\cO(nE))$ gives a regular function $f_{nE}$ on $G$ whose restriction to $\cT=G|_V$ is invertible.
		Let $I_D$ denote the set of irreducible components of $D$, and
	\begin{align*}
	W\colon\Sk(\cT)&\longto\bbR^{I_D}\\
	x&\longmapsto \set{f_{nE}^\trop(x)}_{E\in I_D}.
	\end{align*}
	Under the isomorphism $\Sigma_{(V\subset K)}\simeq\Gamma$, the composition $W\circ\varphi\colon\Gamma\to\bbR^{I_D}$ coincides with the canonical embedding $\Sigma_{(V\subset K)}\subset\bbR^{I_D}$ (up to scaling by $n$).
	Therefore, $\varphi$ determines the dual cone complex $\Sigma_{(V\subset K)}$, hence the dual complex of $V\subset K$.
\end{proof}

\begin{remark}
	We do not know whether the converse holds, i.e.\ whether equality of the dual complexes implies $\varphi= \varphi'$.
	This holds when $V$ is a torus, and also when $K$ is the total space of the canonical bundle of a del Pezzo (as in the context of \cref{thm:del_Pezzo_intro}).
\end{remark}

\subsection{Bogus cones} \label{sec:bogus}

\begin{assumption} \label{ass:3}
	We now assume, in addition to \cref{ass:2}, $q\colon K \to\oK$ is birational.
\end{assumption}

Here we define a canonical extension from the fan $\MovSec(K)$ to a coarsening of the full $\MoriFan(K)$.
Note since $q$ is birational, the support of $\MoriFan(K)$ is the full vector space $\Pic(K)_\bbR$.

\begin{lemma} \label{lem:pure_codimension_one}
	Let $f\colon Y \to Z$ be a regular birational contraction, with $Z$ $\bbQ$-factorial.
	Then the exceptional locus is pure codimension one.
\end{lemma}
\begin{proof}
	Suppose to the contrary there is an irreducible component of the exceptional locus which has higher codimension, and let $A \subset Y$
	be a general very ample divisor through a general point $ e \in E$.
	Then $f^*(f_*(A))$ is an effective $\bbQ$-Cartier divisor, identical to $A$ in a neighborhood of $e$.
	On the other hand, since set theoretically it is the inverse image of $f(A)$, it contains the (positive dimensional)
	fiber of $f$ through $e$, a contradiction.
\end{proof}

\begin{lemma} \label{lem:NefZ}
	Let $Y$ be a birational contraction of $K$.
	Let $f\colon Y \to Z$ be a regular birational contraction with $Z$ $\bbQ$-factorial, $E_1,\dots,E_n \subset Y$ the exceptional divisors, and $C_i \subset E_i$ a curve through a general point.
	Then $\bigcap_i C_i^{\perp} \cap \Nef(Y) = \Nef(Z)$.
\end{lemma}
\begin{proof}
	The inclusion $\supset$ is obvious.
	Let us prove the direction $\subset$.
	The given intersection defines a face of $\Nef(Y)$, let $A$ be a rational point in its interior.
	Since $Y$ is a birational contraction of $K$, $A$ is semi-ample;
		let $g\colon Y \to Z'$ be the corresponding contraction.
	Since the inclusion $\supset$ holds, the rational map $h\colon Z'\dasharrow Z$ is regular.
	Since $f$ is birational, so are $g$ and $h$.
	Since every $C_i$ is contracted by $g$, and $C_i$ passes through a general point of $E_i$, we see that every $E_i$ is $g$-exceptional.
	Therefore, $h\colon Z'\to Z$ is small, and thus an isomorphism by \cref{lem:pure_codimension_one}.
	We conclude that $A$ lies in the interior of $\Nef(Z)$, completing the proof.
\end{proof}

\begin{lemma} \label{lem:span}
	Let $p_i\colon K \dasharrow Z_i$ be birational contractions over $\oK$ such that $\Pic(Z_i)_\bbR \subset \Pic(K)_\bbR$ give the same linear subspaces, for $i=1,2$.
	Then both $p_i$ have the same exceptional divisors; in particular $Z_1 \dasharrow Z_2$ is small.
	
	Let $p\colon K \dasharrow Z$ be a birational contraction with $Z$ $\bbQ$-factorial.
	Then $\Pic(Z)_\bbR \subset \Pic(K)_\bbR$ is the linear span of a face of $\Mov(K)$.
\end{lemma}
\begin{proof}
	Observe that for a regular birational contraction $f\colon Y\to Z$ of $\bbQ$-factorial varieties, an irreducible effective divisor $E$ is in the base locus of $|f^*L\otimes \cO(E)|$ for all $L\in\Pic(Z)$ if and only if $E$ is $f$-exceptional (using the projection formula).
		Thus the vector subspace $\Pic(Z)_\bbR \subset \Pic(Y)_\bbR$ determines the exceptional divisors of $f$.
	For each $p_i\colon K\dasharrow Z_i$, up to replacing $K$ by an SQM, we can assume that $p_i$ is regular.
	So the first statement follows.
	
	For the second:
	Let $\Gamma$ be the minimal face of $\Mov(K)$ containing $\Nef(Z)$, and suppose to the contrary that $\Nef(Z) \subset \Gamma$ is lower dimensional.
	Choose a maximal cone $\Nef(Z_1) \subset \Gamma$ containing $\Nef(Z)$.
	Apply \cref{lem:NefZ} to the regular birational contraction $Z_1\to Z$, and let $E_i, C_i, i=1,\dots,n$ be as in that lemma, with $\bigcap_i C_i^{\perp} \cap \Nef(Z_1) = \Nef(Z)$.
	If all the $C_i$ are non-negative on $\Gamma$, then they cut out a proper face of $\Gamma$ containing $\Nef(Z)$, contradicting the minimality of $\Gamma$.
	Thus there exists one $C_i$, another maximal cone $\Nef(Z_2) \subset \Gamma$ containing $\Nef(Z)$, and a rational point $A$ in the interior of $\Nef(Z_2)$ with $C_i\cdot A < 0$.
	Since $C_i$ passes through a general point of $E_i$ as in \cref{lem:NefZ}, $C_i\cdot A < 0$ implies that $E_i$ lies in the base locus of $A$ (viewed as a divisor on $Z_1$).
	Note that $Z_1 \dasharrow Z_2$ is small by the first paragraph.
	This contradicts the assumption that $A$ is ample on $Z_2$, completing the proof.
\end{proof}

\begin{definition}
	For each maximal bogus cone $b=f^*\Nef(Z)+\braket{\ex(f)}$ corresponding to a birational contraction $f\colon K\to Z$, let $\gamma_b \in\MovSec(K)$ be the minimal cone containing $f^*\Nef(Z)$.
	Two maximal bogus cones are said to be \emph{equivalent} if they have the same $\gamma_b$.
\end{definition}

\begin{proposition} \label{prop:add_bogus_cones}
	Equivalent maximal bogus cones have the same (associated) exceptional locus, and same $\Pic(Z)_\bbR\subset\Pic(K)_\bbR$.
	For each maximal bogus cone $b=f^*\Nef(Z)+\braket{\ex(f)}$, define
	\[
	\beta_b \coloneqq \gamma_b + \braket{\ex(f)}.
	\]
	The cone $\beta_b$ and the above sum decomposition depend only on the equivalence class of $b$.
	Adding the cones $\beta_b$, over all (equivalence classes of) maximal bogus cones $b$, to $\MovSec(K)$ gives a finite rational polyhedral fan $\Sec(K)$, called the \emph{secondary fan}, which is a coarsening of $\MoriFan(K)$.
\end{proposition}
\begin{proof}
	This follows from \cref{lem:span}.
\end{proof}

\subsection{The secondary fan in the del Pezzo case} \label{sec:secondary_fan_del_Pezzo}

Here we apply the general construction of secondary fan to the canoincal bundle over a del Pezzo surface, see \cref{prop:secondary_fan_del_Pezzo} and \cref{rem:Sec_on_PicY} for an explicit description.
We begin with some general notations.

\begin{notation} \label{nota:K}
	Let $k$ be the field of complex numbers $\bbC$.
		Let $Y$ be a smooth Fano $k$-variety, and $D\subset Y$ a normal crossing anti-canonical divisor containing a 0-stratum.
	Let $U\coloneqq Y\setminus D$, $K\to Y$ the canonical bundle, $K^\times\coloneqq K\setminus Y$, $V\coloneqq K^\times|_U$, and $\cP\coloneqq\bbP(K\oplus\cO)$.
	Let $K\to\oK$ and $\cP\to\ocP$ be the contractions of the 0-sections.
		By \cite[Cor.\ 1.3.1]{Birkar_Existence_of_minimal_models_for_varieties_of_log_general_type}, $Y$, $K/\oK$ and $\cP/\ocP$ are Mori dream spaces.
\end{notation}

\begin{lemma} \label{lem:Mori_fan}
	The following hold:
	\begin{enumerate}
		\item We have isomorphisms
		\[\Pic(Y)\xrightarrow[\ \sim\ ]{p^*}\Pic(K)\xrightarrow{\ \sim\ }\Pic(K/\oK),\]
		\[\Pic(Y)\xrightarrow[\ \sim\ ]{p^*}\Pic(\cP/\ocP),\]
		where $p^*$ denotes pullback by projection.
		\item Under the isomorphism $\Pic(\cP/\ocP) \xrightarrow{\sim} \Pic(K/\oK)$, $\MoriFan(\cP/\ocP)$ is identified with $\MoriFan\allowbreak (K/\oK)$.
		\item Under these isomorphisms, the cones $\Mov(K/\oK)$, $\Mov(\cP/\ocP)$ and $\Eff(Y)$ are identified.
		\item As fans on $\Mov(K/\oK)\simeq\Eff(Y)$, $\MovFan(K/\oK)$ refines $\MoriFan(Y)$.
		The two fans are the same when $Y$ has dimension two, (which we expect to hold in all dimensions).
	\end{enumerate}
\end{lemma}

\begin{proof}
	By \cite[Theorem 3.3(a)]{Fulton_Intersection_theory}, we have $\Pic(Y)\xrightarrow[\sim]{p^*}\Pic(K)$, with inverse given by $s^*$, with $s\colon Y\to K$ the 0-section.
	The map $\Pic(K)\to\Pic(K/\oK)$ is surjective by definition.
	To see that it is injective, take $L\in\Pic(K)$ which is equal to the pullback of a line bundle on $\oK$.
	Then $s^*L$ is a trivial line bundle on $Y$, i.e.\ $s^*L=0\in\Pic(Y)$, so $L=0\in\Pic(K)$.
	Next, \cite[Theorem 3.3(b)]{Fulton_Intersection_theory} shows that $\Pic(Y)\xrightarrow{p^*}\Pic(\cP)$ is injective, with retraction given by $s^*$, with $s\colon Y\to\cP$ the 0-section.
	Then the same argument above shows that $\Pic(Y)\xrightarrow{p^*}\Pic(\cP/\ocP)$ is an isomorphism.
	This shows (1).
	
	Observe that $\cP\to\ocP$ and $K\to\oK$ have the same exceptional locus $Y$, so (2) follows, as well as the identification between $\Mov(K/\oK)$ and $\Mov(\cP/\ocP)$.
	Moreover, since $\oK$ is affine, $Y$ is the only possible base divisor for any $L\in\Pic(K)$; this implies the identification between $\Mov(K/\oK)$ and $\Eff(Y)$.
	So we obtain (3).	
	
	Next we show that Mori equivalence for $K/\oK$ implies Mori equivalence for $Y$.
	Recall that two effective divisors are Mori equivalent if they give rise to the same contraction and they have the same stable base divisors (with reduced structure).
	Since the restriction $H^0(K,p^*L)\to H^0(Y,L)$ is surjective (viewed as $\cO(\oK)$-module) for any $L\in\Eff(Y)$, the contraction of $K/\oK$ given by $p^*L$ restricts to the contraction of $Y$ given by $L$, and the base locus of $p^*L$ on $K$ restricts to the base locus of $L$ on $Y$.
	Therefore, given $L_1,L_2\in\Eff(Y)$, if $p^*L_1$ and $p^*L_2$ are Mori equivalent on $K/\oK$, then $L_1$ and $L_2$ are Mori equivalent on $Y$.
	So $\MovFan(K/\oK)$ refines $\MoriFan(Y)$.
	Finally, when $Y$ is 2-dimensional, the refinement is an equality by the explicit description in \cref{lem:SQM_of_K}.
\end{proof}

\begin{remark} \label{rem:Gamma}
	Let $\Gamma=\Sk(V)\cap K^\beth$ be as in \cref{const:universal_torsor}.
	It can also be described by tropicalization of divisors:
	Let $Y_0$ and $Y_\infty$ denote respectively the 0-section and the $\infty$-section of $\pi\colon\cP=\bbP(K\oplus\cO)\to Y$.
	Consider the divisor $\delta\coloneqq Y_0-Y_\infty$.
	Then $\Gamma=\set{\delta^\trop\ge 0}\subset\Sk(V)\subset V^\an$.
	By \cite[Lemma 15.2]{Keel_Yu_The_Frobenius},
		$\Gamma$ is the support of the sub cone complex $\Sigma_{(V\subset K)}\subset\Sigma_{(V\subset\cP)}$, generated by the components of $K\setminus V$ (see \cref{def:dual_complex} for dual (cone) complexes in the normal crossing case).
	We denote by $\uGamma$ the induced cone complex structure on $\Gamma$.
	
	The divisor $D\subset Y$, viewed as a section of the dual $K^*\to Y$, gives a regular function $d$ on $K$.
	Note the principle divisor on $\cP$ associated to $d$ is equal to $p^*D+\delta$.
	We have an isomorphism $p \times d\colon V \xrightarrow{\sim} U \times \Gm$ which induces $\Sk(V) \xrightarrow{\sim} \Sk(U) \times \bbR$ (see \cite[Proposition 8.8]{Keel_Yu_The_Frobenius}), and
	\begin{equation} \label{eq:Gamma}
	\Gamma\xrightarrow{\ \sim\ }\Set{(x,n) | D^\trop(x)\le n}\subset\Sk(U)\times\bbR_{\ge 0},
	\end{equation}
	together with the analogous bijections for the integer points.
	
	Let
	\begin{equation} \label{eq:Lambda}
	\Lambda\coloneqq\Set{x \in \Gamma| d^{\trop}(x) =1}.
	\end{equation}
	It projects to $\{D^{\trop}(x) \leq 1\}\subset\Sk(U)$.
	We have $[Y] \in \Lambda(\bbZ)$, the divisorial valuation given by $Y\subset K$.
	When $(Y,D)$ is toric Fano, $\Lambda$ is isomorphic to the dual reflexive polytope associated to $Y$.
	
	By construction, we have $\Gamma\simeq C(\Lambda)$, where $C(\cdot)$ denotes the cone.
	Then the cone complex $\uGamma$ induces a cell complex structure $\uLambda$ on $\Lambda$ (more precisely, a $\Delta$-complex structure, see \cref{def:dual_complex}).
	Taking boundary, we have $\partial\uGamma\simeq\Sigma_{(V\subset K^\times)}\simeq C(\partial\uLambda)$.

Note $K^\times\to\oK$ is an open embedding, its complement is the point to which $Y$ contracts under $K\to\oK$.
It follows that any birational contraction $K\dasharrow K'$ over $\oK$ has exceptional locus contained in $Y$, so it compactifies uniquely to a birational contraction $\cP\dasharrow\cP'$.
Moreover, we always have $K^\times\subset K'$, so the dual cone complex $\Sigma_{(V\subset K^\times)}$ is a sub-complex of $\Sigma_{(V\subset K')}$, which is supported on $\partial\Gamma$.

Let $\cT \to V$ denote the restriction of the universal torsor $\pi\colon G \to K$ corresponding to the fixed $\sfN \subset \Pic(K)$, as in \cref{const:varphi_of_SQM}.
For any SQM $K\dasharrow K'$ over $\oK$, the restriction to $V$ of the corresponding torsor $G'\to K'$ is canonical identified with $\cT$, and we have a section $\varphi_{K'}\colon\Gamma\to\tGamma$, where $\tGamma=\Sk(\cT)|_\Gamma$.
\end{remark}

We will be omitting $/\oK$ from the notations as in \cref{ass:2}.

\begin{lemma} \label{lem:bogus_K}
	The bogus cones of $\Sec(K)$ are exactly $\gamma+\bbR_{\ge 0}[Y]\subset\Eff(K)\simeq\Pic(K)_\bbR$ where $\gamma$ is a cone of $\MovSec(K)$ lying in the boundary $\Mov(K)$.
\end{lemma}
\begin{proof}
	It suffices to observe that any divisorial contraction $K\dasharrow K'$ over $\oK$ cannot contract any divisor except $Y\subset K$.
\end{proof}

\begin{lemma} \label{lem:varphieq}
	For any SQM $K\dasharrow K'$ over $\oK$, the section $\varphi_{K'}\colon\Gamma\to\tGamma$ restricted to $\partial\Gamma\cup\{[Y]\}\subset\Gamma$ is independent of the SQM.
\end{lemma}
\begin{proof}
	Using \cref{rem:internal_flop}, it follows from the fact that $K \dasharrow K'$ is an isomorphism on $K^\times$, as well as on the generic point of $Y\subset K$.
\end{proof}

Assume for the remainder of this section that $Y$ is 2-dimensional.

\begin{lemma} \label{lem:SQM_of_K}
	For any SQM $b\colon K \dasharrow K'$ over $\oK$,
	the following hold:
	\begin{enumerate}
		\item The exceptional locus, $\ex(b)$, is a disjoint union of $(-1)$-curves in the zero section $Y \subset K$, each of which has normal bundle $O(-1) \oplus O(-1)$ in $K$.
		\item The map $b$ is the composition of $(-1,-1)$-flops of disjoint curves in $\ex(b)$, meaning that we blowup each curve $C$ in $\ex(b)$, obtain an exceptional divisor $E\simeq\bbP^1\times\bbP^1 \xrightarrow{p_1}\bbP^1\simeq C$, and then blowdown by the second projection $p_2$.
		Moreover, $K'$ is smooth.
		\item The rational map $Y \subset K \overset{b}{\dasharrow} K'$ is regular, the blowdown of the disjoint union of $(-1)$-curves in $\ex(b)$.
		We denote $Y'\coloneqq b(Y)\subset K'$ and $b_Y\colon Y\to Y'$.
		\item \label{lem:SQM_of_K:J} Let $J \subset K'$ be the strict transform of $p^{-1}(\ex(b)) \subset K$ (where $p\colon K \to Y$ is the projection).
		Then the restriction of the rational map $p\colon K' \dasharrow Y'$ to $J^c$ is regular, and is canonically identified with the canonical bundle $K_{Y'} \to Y'$.
	\end{enumerate}
\end{lemma}

\begin{proof}
		Since $K \to \oK$ is an isomorphism outside the zero section $Y\subset K$, the exceptional locus $\ex(b)$ lies in $Y\subset K$.
	Let $L\coloneqq b^*A$ for an ample line bundle $A$ on $K'$.
	By \cref{lem:Mori_fan}(1), $L\simeq p^*L_Y$ for a line bundle $L_Y$ on $Y$.
	Since the restriction $H^0(K,p^*L)\to H^0(Y,L)$ is surjective (viewed as $\cO(\oK)$-module), the contraction of $K/\oK$ given by $L$ restricts to the contraction of $Y$ given by $L_Y$, denoted by $b_Y\colon Y\dasharrow Y'$.
	Since $b$ is small, $b_Y$ is birational.
	Since $Y$ is 2-dimensional, $b_Y$ is regular.
	By the genus formula, the exceptional locus of $b_Y$ is a union of $(-1)$-curves.
		It is a disjoint union because the intersection matrix of the irreducible components must be negative definite.
		Then by the adjunction formula, the normal bundle $\cN_{C\subset K}$ is isomorphic to $\cO(-1)\oplus\cO(-1)$.
	It follows from the uniqueness of flop of a small contraction that $b$ is the $(-1,-1)$-flop of these curves.
	This shows statements (1-3).
	
	Statement (4) follows from a simple explicit computation relating the flop to the elementary transformation of the line bundle $K \to Y$ along $E\coloneqq\ex(b)\subset Y\subset K$, which transforms it into $K \otimes \cO(-E) \simeq b_Y^*(K_{Y'})$.
	\end{proof}

\begin{proposition} \label{prop:secondary_fan_del_Pezzo}
	We call a $(-1)$-curve in $Y$ either \emph{boundary} or \emph{internal} depending on whether it is a component of $D\subset Y$.
	Each maximal cone of $\MovSec(K)$ is the union of maximal cones of $\MovFan(K)$ corresponding to SQMs $K\dasharrow K'$ with the same set of boundary exceptional $(-1)$-curves.
	The bogus cones of $\Sec(K)$ are described in \cref{lem:bogus_K}.
\end{proposition}
\begin{proof}
	By \cref{rem:internal_flop}, the flop of an internal $(-1)$-curve does not change $\varphi$.
	On the other hand by \cref{lem:SQM}, if we flop a boundary $(-1)$-curve  $C\subset Y\subset K$, then $\varphi([E])$ changes, where $[E]\in\Gamma(\bbZ)$ corresponds to the exceptional divisor for the blowup of $C\subset K$.
\end{proof}

\begin{remark} \label{rem:Sec_on_PicY}
	We can then equivalently describe $\Sec(K)$ as a fan structure on $\Pic(Y)_\bbR$.
	By \cref{lem:Mori_fan}, the isomorphism $\Pic(K/\oK)_\bbR\simeq\Pic(Y)_\bbR$ identifies $\Mov(K)$ with $\Eff(Y)$.
	Then each maximal cone of $\MovSec(K)$ corresponds to the union of cones in $\Pic(Y)_\bbR$ of form $f^*\Nef(Y')+\braket{\ex(f)}$ over all divisorial contractions $f\colon Y\to Y'$ with the same set of boundary exceptional $(-1)$-curves; while each bogus cone of $\Sec(K)$ corresponds to a cone in $\Pic(Y)_\bbR$ of form $\gamma+\bbR_{\ge 0}[K]$ where $\gamma$ lies in the boundary $\partial\Eff(Y)$ and is a cone of the fan we just made in $\Eff(Y)$.
\end{remark}

\begin{example}
	Consider the case when $Y$ is a del Pezzo of degree one.
	Then $D \in |-K_Y|$ is irreducible, so there are no boundary $(-1)$-curve in $Y$.
	Thus by \cref{prop:secondary_fan_del_Pezzo}, $\MovSec(K)$ has just a single maximal cone, $\Mov(K) = \Eff(Y)$, and $\Sec(K)$ is obtained by adding the cone $\gamma + \bbR_{\geq 0}[Y]\subset\Eff(K)$ for each face $\gamma$ of $\Mov(K)$.
\end{example}

\section{Rephrasing the mirror algebra using the universal torsor} \label{sec:rephrase}

Let $k$ be a field of characteristic 0, $U$ a smooth affine log Calabi-Yau $k$-variety containing an open split algebraic torus, and $U\subset Y$ a normal crossing compactification.
Let $A_Y$ denote the associated mirror algebra of \cite{Keel_Yu_The_Frobenius}.
It is a free module over the monoid ring $R_Y\coloneqq\bbZ[\NE(Y,\bbZ)]$ with basis $\Sk(U,\bbZ)$, the integer points in the essential skeleton.
Given $P_1,\dots,P_n\in\Sk(U,\bbZ)$, $n\ge 2$, write the product in $A_Y$ as
\begin{equation} \label{eq:multiplication}
\theta_{P_1}\cdots\theta_{P_n}=\sum_{Q\in\Sk(U,\bbZ)}\ \sum_{\gamma\in\NE(Y)} \chi(P_1,\dots,P_n,Q,\gamma) z^\gamma\theta_Q.
\end{equation}
The structure constants $\chi(P_1,\dots,P_n,Q,\gamma)$ are given by counts of non-archimedean analytic disks in the analytification $U^\an$ with respect to the trivial valuation on $k$ (see \cite[Definition 1.5]{Keel_Yu_The_Frobenius}).

Note that the mirror algebra involves two kinds of monomial-like objects: the theta function basis $\theta_P$ for $P\in\Sk(U,\bbZ)$, and the coefficients $z^\gamma$ for $\gamma\in\NE(Y,\bbZ)$.
In this section, we will rephrase the mirror algebra using the universal torsor over $Y$, see \cref{thm:reformulation_universal_torsor}.
In this way, we can incorporate the second sort of monomials into the first sort.
Such a reformulation will be necessary for comparing and gluing mirror algebras over different birational models of $Y$.

As in \cite[Remark 1.3]{Keel_Yu_The_Frobenius}, we can remove the independence of the mirror algebra $A_Y$ on the compactification $Y$ by setting all curve classes to 0: we put $A_U\coloneqq A_Y\otimes_{R_Y}\bbZ$, where $R_Y\to\bbZ$ sends every $z^\gamma$ to 1.

Since $Y$ contains an open algebraic torus, $\Pic(Y)$ is a lattice.
Given any sublattice $N\subset\Pic(Y)$, let $M\coloneqq N^*$, $T^N=T_M$, $\pi\colon G\to Y$, $\cT\coloneqq G|_U$, $\Sk(\cT)\to\Sk(U)$ and $\varphi\colon\Sk(U)\to\Sk(\cT)$ be as in \cref{const:universal_torsor}.
We have a natural projection $\pi_M\colon N_1(Y,\bbZ)\to M$.
Let $S_1,\dots,S_m$ be a basis of $N$, then we have $G\simeq L_1^\times\times\dots\times L_m^\times$, where $L_i$ denotes the dual of $S_i$.
Let \[f\colon [\bbD,(p_1,\dots,p_n,s)]\longrightarrow Y^\an\]
be a structure disk as in \cite[Definition 14.5]{Keel_Yu_The_Frobenius} defined over a discrete valuation field.
Let $\eta\coloneqq\partial\bbD$ be the Berkovich boundary point, $\bbD^\circ \coloneqq \bbD \setminus \{p_1,\dots,p_n\}$, $\Gamma$ the convex hull of $\eta,p_1,\dots,p_n$ in $\bbD$, and $\Gamma^\circ\coloneqq\Gamma\setminus\{p_1,\dots,p_n\}$.
By \cite[Theorem 3.7.2]{Fresnel_Rigid_analytic_geometry_and_its_applications} and \cite[Tag 0BCH]{Stacks_project}, all line bundles on $\bbD$ are trivial; in particular, $f^*(L_i)$ is trivial for every $i=1,\dots,m$.
Let $\sigma_i$ be a non-vanishing section of $f^*(L_i)$.
They give rise to a section $\sigma\colon\bbD\to f^*(G)$.
Composing with fiberwise retraction of $T_M^\an$ onto its skeleton, we obtain a section $\sigma^t\colon\Gamma^\circ\to\Sk(\cT)|_{\Gamma^\circ}$.
Furthermore, the restriction of $\varphi\colon\Sk(U)\to\Sk(\cT)$ gives a section $\varphi\colon\Gamma^\circ\to\Sk(\cT)|_{\Gamma^\circ}$.
Thus we obtain a continuous function $g\colon\sigma^t-\varphi\colon\Gamma^\circ\to M_\bbR$, as it is the difference of two sections of an $M_\bbR$-principal bundle.

\begin{lemma} \label{lem:class_via_varphi}
	We have
	\[d_\eta(g) = \pi_M ([f\colon \bbD \to Y^\an]) \in M\]
	where $d_\eta$ denotes the derivative at $\eta$ in the direction of the unique incident edge.
	Moreover $d(g)$ is zero near every $p_j$.
\end{lemma}

\begin{remark}
	Note that $\sigma^t$ depends on the choice of the non-vanishing sections $\sigma_i$.
	Nevertheless, any other choice of $\sigma_i$ differs by an invertible function on $\bbD$, which has constant tropicalization by the maximum modulus principle.
	Hence the derivative of $\sigma^t$ is independent of choices of $\sigma_i$.
	This independence is also a consequence of \cref{lem:class_via_varphi}.
\end{remark}

\begin{proof}[Proof of \cref{lem:class_via_varphi}]
	It is enough to prove the case when $N$ has rank one, as both sides commute with projection.
	So we may assume $N$ is generated by a single line bundle $S$.
	Let $L\to Y$ be the dual line bundle.
	Note that the projection $\pi_M\colon N_1(Y,\bbZ)\to M$ is given by coupling with $c_1(S)=-c_1(L)$.
	So the statement (in this rank one case) is equivalent to
	\[-d_\eta(g) = c_1(L) \cdot [f\colon \bbD \to Y^\an].\]
	Let $\ff\colon\fC\to\widehat Y_{\kc}$ be a strictly semistable formal model of the structure disk up to passing to a base field extension.
	We view $\sigma$ as a rational section of $\ff^*(L)$ on $\fC$, and let $F$ be the associated Cartier divisor, which by assumption is supported on the central fiber.
	Taking valuation, we obtain $F^\trop\colon\Gamma^\circ\to\bbR$ (see \cite[Construction 15.1]{Keel_Yu_The_Frobenius}).
	Taking local trivialization of $\ff^*(L)$ and tracing through the definitions, we have $\sigma^t - \varphi = F^\trop$ on $\Gamma^\circ$.
		Since $F$ is supported on the central fiber, $F^\trop$ is constant on $\Gamma^\circ$ near every $p_j$ by \cite[Lemma 15.2]{Keel_Yu_The_Frobenius}, and thus its derivative is zero near every $p_j$.
	Now the result follows from \cite[Lemma 15.3]{Keel_Yu_The_Frobenius}.
\end{proof}

\begin{definition} \label{def:above}
	We say a point $P\in\Sk(\cT,\bbZ)$ is $M'$-\emph{above} $\varphi$, for some monoid $M'\subset M$, if $P-\varphi(\pi(P))\in M'$.
\end{definition}

We consider the mirror algebra $A_\cT$ for $\cT$ with all curve classes set to 0, this is independent of any compactification $\cT\subset\ocT$.
Comparing with the base extension $A_Y\otimes_{R_Y} \bbZ[M]$, we note there is a canonical identification of their $\bbZ$-bases under $\Sk(U)\times M \xrightarrow{\sim} \Sk(\cT)$, sending $(P,\gamma)\mapsto\varphi(P)+\gamma$.

Let $\NE(Y)_M \subset M$ denote the image of $\NE(Y,\bbZ)$ under the projection $\pi_M\colon\allowbreak N_1(Y,\bbZ)\allowbreak \to M$. 

\begin{theorem} \label{thm:reformulation_universal_torsor}
	The above identification of free $\bbZ$-modules gives an isomorphism of rings, $A_Y\otimes_{R_Y} \bbZ[M]\xrightarrow{\sim}A_\cT$.
	The image of $A_Y \otimes_{R_Y} \bbZ[\NE(Y)_M]$ is the sub-$\bbZ$-module of $A_\cT$ spanned by the basis elements $\theta_P$ over all $P\in\Sk(\cT,\bbZ)$ that are $\NE(Y)_M$-above $\varphi$.
\end{theorem}

\begin{proof}
	The ring isomorphism in the theorem follows from the two equalities below:
	\begin{enumerate}
	\item For any $P\in\Sk(U,\bbZ)$, $\alpha,\beta \in M$,
	\[\theta_{\varphi(P)+\alpha} \cdot \theta_{\varphi(0)+\beta} = \theta_{\varphi(P)+\alpha+\beta},\]
	\item For any $P_1,P_2,Q\in\Sk(U,\bbZ)$, $\gamma\in M$,
	\[\sum_{\substack{\gamma'\in\NE(Y,\bbZ)\\ \pi_M(\gamma')=\gamma}}\chi(P_1,P_2,Q,\gamma')=\sum_{\delta\in\NE(\ocT,\bbZ)}\chi(\varphi(P_1),\varphi(P_2),\varphi(Q)+\gamma,\delta).\]
\end{enumerate}

Fix an open algebraic torus $T\subset U$.
Choose an snc compactification $\cT\subset\ocT$ so that $\pi\colon\cT\to U$ extends to $\pi\colon\ocT\to Y$.
For statement (1):
Let $f\colon[\bbD,(p_1,p_2,s)]\to\ocT^\an$ be a general structure disk responsible for any structure constant for the multiplication $\theta_{\varphi(P)+\alpha}\cdot\theta_{\varphi(0)+\beta}$ in $A_\cT$, see \cite[Definition 14.5]{Keel_Yu_The_Frobenius}.
Let $g$ be the composition of $f$ with $\pi\colon\ocT\to Y$.
Then $g$ is a structure disk responsible for a structure constant for the multiplication $\theta_P\cdot\theta_0$ in $A_Y$.
So the spine in $\Sk(U)$ associated to $g$ is disjoint from any walls (with respect to the fixed open algebraic torus $T \subset U$).
In particular, the spine does not have any bending points, and there are no twigs attached.
It follows that the image $g(\bbD\setminus p_1)$ lies in $T^\an\subset U^\an$, hence the image $f(\bbD\setminus p_1)$ lies in $\pi^{-1}(T)^\an$ where $\pi\colon\ocT\to Y$.
Now the equality (1) follows from the multiplication rule in the toric case, see \cite[Lemma 9.4]{Keel_Yu_The_Frobenius}.

Next we prove statement (2):
Applying \cite[(14.2)]{Keel_Yu_The_Frobenius} (and its notation) to $U\subset Y$, $P_1,P_2,Q\in\Sk(U,\bbZ)$ and $\gamma'\in\NE(Y,\bbZ)$, we obtain
\[\cN(U\subset Y,P_1,P_2,Q,\gamma')\xrightarrow{\ \Phi^\an\ }V_\cM\times V_Q\subset(\cM_{0,4}\times U)^\an,\]
whose degree gives $\chi(P_1,P_2,Q,\gamma')$.
Similarly, applying to $\cT\subset\ocT$, $\varphi(P_1),\varphi(P_2),\allowbreak \varphi(Q)\allowbreak +\gamma\in\Sk(\cT,\bbZ)$ and $\delta\in\NE(\ocT,\bbZ)$, we obtain
\[\cN\big(\cT\subset\ocT,\varphi(P_1),\varphi(P_2),\varphi(Q)+\gamma,\delta\big)\xrightarrow{\ \Psi^\an\ }V_\cM\times V_{\varphi(Q)+\gamma}\subset(\cM_{0,4}\times\cT)^\an,\]
whose degree gives $\chi(\varphi(P_1),\varphi(P_2),\varphi(Q)+\gamma,\delta)$.

Pick any $\mu\in V_\cM$ and let
\[b\coloneqq(\mu,\varphi(Q)+\gamma)\in V_\cM\times V_{\varphi(Q)+\gamma}\subset(\cM_{0,4}\times\cT)^\an.\]
Make a base field extension so that $k$ becomes algebraic closed and $b\in(\cM_{0,4}\times\cT)^\an(k)$.
Then by \cite[Lemma 9.11]{Keel_Yu_The_Frobenius}, the fiber
\[\cN\big(\cT\subset\ocT,\varphi(P_1),\varphi(P_2),\varphi(Q)+\gamma,\delta\big)_b\]
is just a finite set, whose cardinality gives $\chi(\varphi(P_1),\varphi(P_2),\varphi(Q)+\gamma,\delta)$.
The point $b$ projects to a point $b'\in(\cM_{0,4}\times U)^\an(k)$, and the fiber
\[\cN(U\subset Y,P_1,P_2,Q,\gamma')_{b'}\]
is also a finite set, whose cardinality gives $\chi(P_1,P_2,Q,\gamma')$.

By \cref{lem:class_via_varphi}, the projection $\pi\colon\cT\to U$ induces a map
\begin{multline*}
\Pi\colon\coprod_{\delta\in\NE(\ocT,\bbZ)}\cN\big(\cT\subset\ocT,\varphi(P_1),\varphi(P_2),\varphi(Q)+\gamma,\delta\big)_b\\
\longto\coprod_{\substack{\gamma'\in\NE(Y,\bbZ)\\ \pi_M(\gamma')=\gamma}}\cN(U\subset Y,P_1,P_2,Q,\gamma')_{b'}.
\end{multline*}
Now it remains to show that $\Pi$ is a bijection.
We construct the inverse of $\Pi$ as follows.
Let $S_1,\dots,S_m$ be a basis of $\Pic(Y)$, and let $L_i$ denote the dual of $S_i$.
Let
\[[C,(p_1,p_2,z,s),f\colon C\to Y^\an]\in\cN(U\subset Y,P_1,P_2,Q,\gamma')_{b'},\]
with $\gamma'\in\NE(Y,\bbZ)$, $\pi_M(\gamma')=\gamma$.
For every $i=1,\dots,m$, we choose a rational section $\sigma_i$ of $f^*L_i$ whose associated divisor is supported at $z$;
note the choice is unique up to multiplication by a scalar in $k^\times$.
The rational sections induce a lift $\sigma(f)\colon C\to\ocT^\an$, and we choose the scalars uniquely so that $\sigma(f)(s)=b$.
Note that $\sigma(f)(C\setminus z)\in G^\an$ (where $G\to Y$ is the torsor associated to $N \subset \Pic(Y)$).
Then \cref{lem:class_via_varphi} implies
\[\sigma(f)\in\cN\big(\cT\subset\ocT,\varphi(P_1),\varphi(P_2),\varphi(Q)+\gamma,\delta\big)_b\]
for a unique $\delta\in\NE(\ocT)$.
The assignment $f\mapsto\sigma(f)$ gives the inverse of $\Pi$, completing the proof.
\end{proof}

\begin{remark}  \label{rem:wave}
	We note that in \cite[Remark 17.7]{Keel_Yu_The_Frobenius}, the mirror algebra is defined for any normal projective compactification $U \subset Y$, not necessarily snc.
	\cref{thm:reformulation_universal_torsor} holds as stated in that generality.
	Indeed, the mirror algebra in general is constructed by choosing an snc resolution $q\colon \tY \to Y$ which is an isomorphism over $U$, and then taking $A_Y \coloneqq A_{\tY} \otimes_{R_Y} R_{\tY}$ (which turns out to be independent of the resolution).
	We pullback the universal torsor $G \to Y$ to $\tY$.
	Note that $\cT$ (and in particular, $A_{\cT}$) does not change.
	So the result for $\tY$ implies the result for $Y$.
\end{remark} 

\section{Central fiber as an umbrella} \label{sec:umbrella}

In this section we describe the central fiber of the affine mirror family as the spectrum of a (generalized) Stanley-Reisner ring, which we call an \emph{umbrella}.
This explicit description will be useful in the proof of the main theorem.

We continue the setting of \cref{sec:rephrase}, where $U$ is a smooth affine log Calabi-Yau variety containing an open split algebraic torus, $U\subset Y$ a normal crossing compactification, and $D\coloneqq Y\setminus U$.
We have the mirror algebra $A_Y$ from \cite{Keel_Yu_The_Frobenius} over the monoid ring $R_Y\coloneqq\bbZ[\NE(Y,\bbZ)]$.
Let $\cV\coloneqq\Spec A_Y$ and $\cV_0$ the fiber over the unique 0-stratum of the base toric variety $\TV(\Nef(Y))\simeq\Spec R_Y$, which is given by the maximal monomial ideal $\fm_Y$ in $R_Y$.
Our goal here is to compute the fiber $\cV_0$.

First let us define the \emph{dual complex} $\Delta(E)$ of any normal crossing divisor $E\subset Y$, which generalizes the usual dual complex in the simple normal crossing case.
We use the terminology of $\Delta$-complex from \cite[\S 2.1]{Hatcher_Algebraic_topology}, which is a generalization of simplicial complex.

\begin{definition} \label{def:dual_complex}
	A \emph{stratum} of $E$ is an irreducible component of an iterated singular locus $\Sing(\dots\allowbreak\Sing(E))$ of $E$.
	We take an $n$-simplex for each codimension-$n$ stratum of $E$, and glue them according to the way the strata of $E$ fit together.
	The resulting $\Delta$-complex $\Delta(E)$ is call the \emph{dual complex} of $E$.
	Let $\Delta'(E)$ denote the collection of simplexes before the gluing.
	Let $\Sigma_{(Y,E)}$ denote the cone over $\Delta(E)$,  and similarly for $\Sigma'_{(Y,E)}$.
	We call $\Sigma_{(Y,E)}$ the \emph{dual cone complex}.
\end{definition}

\begin{definition} \label{def:umbrella}
	Let $D^\ess\subset D$ be the union of essential divisors, i.e.\ irreducible components where the volume form has a pole.
	Let $S'$ be the Stanley-Reisner ring for $\Delta'(D^\ess)$ (see \cite[Definition 1.6]{Miller_Combinatorial_commutative_algebra}).
	As an abelian group, it is free with basis the integer points $\Sigma'_{(Y,D^\ess)}(\bbZ)$.
	Let $q\colon\Sigma'_{(Y,D^\ess)}\to\Sigma_{(Y,D^\ess)}$ denote the quotient map.
	Let $S\subset S'$ be the subgroup with bases
	\[\theta_P\coloneqq\sum_{P'\in q\inv(P)}\theta_{P'}\]
	over all $P\in\Sigma_{(Y,D^\ess)}(\bbZ)$.
	We call $\Spec S$ an \emph{umbrella}.
\end{definition}

\begin{proposition} \label{prop:umbrella}
	We have $A_Y\otimes_{R_Y} R_Y/\fm_Y\simeq S$, identifying the basis elements via the canonical isomorphism $\Sk(U,\bbZ)\simeq\Sigma_{(Y,D^\ess)}(\bbZ)$.
	\end{proposition}

In order to prove \cref{prop:umbrella}, we need to compute the multiplication rule on $A_Y\otimes_{R_Y} R_Y/\fm_Y$.
The only structure disks which contribute (modulo $\fm_Y$) have zero curve classes.
From \cite[Definition 7.1]{Keel_Yu_The_Frobenius}, the spine associated to a general structure disk of class 0 must map to the interior of a maximal cell of $\Sigma_{(Y,D)}$ and must be balanced.
So \cref{prop:umbrella} follows from the following:

\begin{proposition} \label{prop:count_balanced_spine}
	Let $h\colon\Gamma\to\oSk(U)$ be a spine in the essential skeleton $\Sk(U)$ (see \cite[Definition 9.1]{Keel_Yu_The_Frobenius}).
	Assume that $h(\Gamma)\cap\Sk(U)$ lies in the interior of a maximal cell of $\Sigma_{(Y,D^\ess)}$, and that $h$ is balanced (at the interior vertices of $\Gamma$).
	Then the count $N(h,0)=1$, and $N(h,\gamma)=0$ for all $\gamma\neq 0$.
\end{proposition}

The two propositions above hold under \cite[Assumption 2.4]{Keel_Yu_The_Frobenius} by Proposition 15.12 (see also Lemmas 9.4 and 14.7) in loc.\ cit..
It is possible to remove that Assumption 2.3, but for the simplicity of exposition and for the purpose of this paper, let us explain only the 2-dimensional case, which is the content of \cref{prop:count_balanced_spine_dim2}.

\subsection{Counts of balanced spines in dimension two} \label{sec:counts_balanced}

In this subsection, we assume moreover that $U$ is 2-dimensional.
Recall that any minimal snc compactification $U\subset Y$ gives a Berkovich retraction $\tau\colon U^\an\to\Sk(U)$.
The retraction does not change if we blowup 0-strata of $D$, and any two minimal snc compactifications of $U$ are related by such blowups.
Therefore the retraction $\tau$ is canonical.
Moreover, the retraction is an affinoid torus fibration outside $0\in\Sk(U)$, which induces a canonical \Zaffine structure on $\Sk(U)\setminus 0$ (see \cite[Proposition 3.6]{Yu_Enumeration_of_holomorphic_cylinders_I}, see also \cite{Nicaise_Xu_Yu_The_non-archimedean_SYZ_fibration}).

\begin{lemma} \label{lem:torus_atlas}
	Let $E \in \Sk(U,\bbZ)$ be a non-zero primitive integer point.
	Then there is an snc compactification $U \subset Y$ where $E$ has divisorial center, and a toric model (see \cite[Definition 1.18]{Gross_Mirror_symmetry_for_log_Calabi-Yau_surfaces_I_v1}) $\pi\colon Y\to\oY$ whose exceptional locus is disjoint from (the divisor corresponding to) $E$.
\end{lemma}
\begin{proof}
	Let $T_M \subset U$ be an open algebraic torus with cocharacter lattice $M$.
	The choice identifies $\Sk(U,\bbZ) \simeq \Sk(T_M,\bbZ) \simeq M$.
	Let $T_M \subset \oY$ be a toric compactification with a ruling $\oY\to\bbP^1$, on which (the boundary divisors corresponding to) $E,-E \in M$ are disjoint sections.
	Then after replacing $\oY$ by a toric blowup (i.e.\ refining the fan), there is a minimal snc compactification $U \subset Y$, such that $E$ has divisorial center (which we denote also by $E$), $\pi\colon Y\dasharrow\oY$ is regular, and no component of $D\coloneqq Y\setminus U$ is $\pi$-exceptional.
	Suppose there is a $\pi$-exceptional divisor meeting $E$ at $p$.
	Let $\oY'$ be the elementary transformation of $\oY\to\bbP^1$ at $\pi(p)$ (i.e.\ blowup $\pi(p)$ and blowdown the strict transform of the fiber through $\pi(p)$), which is again toric (but for a different copy of $T_M \subset U$).
	Then $\pi'\colon Y \dasharrow \oY'$ is again regular, but there is no $\pi'$-exceptional divisor meeting $E$ at $p$.
	We replace $\oY$ by $\oY'$ and repeat the process until the $\pi$-exceptional locus is disjoint from $E$.
\end{proof}

Our main interest in \cref{lem:torus_atlas} is:

\begin{lemma} \label{lem:enough_tori}
	There is a finite covering of $\Sk(U)\setminus 0$ by open cones, such that $\tau^{-1}$ of each cone is contained in $T^\an\subset U^\an$ for some open algebraic torus $T\subset U$, where $\tau\colon U^\an\to\Sk(U)$ is the canonical retraction (in dimension two).
\end{lemma}
\begin{proof}
	For any non-zero primitive integer point $E\in\Sk(U,\bbZ)$, let $\pi\colon Y\to\oY$ be as in \cref{lem:torus_atlas}, and let $\operatorname{star}(E)$ be the union of open cones of $\Sigma_{(Y,D)}$ whose closure contains $E$.
	Then \cref{lem:torus_atlas} implies that $\tau\inv(\operatorname{star}(E))\subset T^\an$, for $T$ the structure torus of $\oY$.
	All such $\operatorname{star}(E)$ cover $\Sk(U)\setminus 0$, and we extract a finite covering by compactness.
\end{proof}

\begin{proposition} \label{prop:count_balanced_spine_dim2}
	Fix an snc compactification $U\subset Y$ with $D\coloneqq Y\setminus U$.
	Let $h\colon\Gamma\to\oSk(U)$ be a spine in $\Sk(U)$ (see \cite[Definition 9.1]{Keel_Yu_The_Frobenius}).
	Assume it is transverse to the dual cone complex $\Sigma_{(Y,D)}$.
	Let $f\colon C\to Y^\an$ be a map from a compact quasi-smooth curve such that the associated spine is $h$.
	Let $Z(f)\in Z_1(Y)$ be the algebraic 1-cycle associated to $f$ as in \cite[Definition 7.1]{Keel_Yu_The_Frobenius}.
	Then $h$ is balanced (at the interior vertices) with respect to the \Zaffine structure on $\Sk(U)\setminus \{0\}$ 
	if and only if $Z(f)$ is supported on boundary 1-strata of $Y$.
	
	Next assume $h$ is balanced (at the interior vertices).
	Define
	\[Z(h)\coloneqq\sum_{x\in h\inv(\Sigma^1_{(Y,D)})} \abs{d_x h \wedge e_{h(x)}} Z_{h(x)} \in Z_1(Y),\]
	where $\Sigma^1_{(Y,D)}\subset\Sigma_{(Y,D)}$ denotes the union of 1-dimensional cones, $d_x h$ denotes the derivative at $x$, $e_{h(x)}$ denotes the primitive integral vector on the ray containing $h(x)$, $\abs{\cdot}$ denotes the lattice length of the wedge product, and $Z_{h(x)}$ denotes the boundary 1-stratum of $Y$ corresponding to the ray containing $h(x)$.
	Then the count
	\[N(h,\gamma)=
	\begin{cases}
	1 &\text{ if }\gamma=Z(h)\in\NE(Y,\bbZ),\\
	0 &\text{ otherwise}.
	\end{cases}.\]
\end{proposition}
\begin{proof}
	Since $ 0 \in \Sk(U)$ is the only singularity of the \Zaffine structure, the spine $h$ is balanced (at the interior vertices) if and only if the tropical curve associated to $f$ has no twigs (i.e.\ branches attached to the spine).
	Since only twigs can contribute to components of $Z(f)$ not supported on boundary 1-strata of $Y$, we deduce the first assertion of the proposition.
	
	For the second assertion, using the gluing formula (\cite[Theorem 12.4]{Keel_Yu_The_Frobenius}), we can cut $h$ into small pieces, and then by \cref{lem:enough_tori}, we can assume $\tau\inv(h(\Gamma))\subset T^\an$ for an open algebraic torus $T\subset U$, where $\tau \colon U^\an\to\Sk(U)$ is the canonical retraction map.
	Then the spine $h$ is also balanced with respect to the \Zaffine structure given by $T$, so we can conclude by \cite[Lemma 9.4]{Keel_Yu_The_Frobenius} using this torus $T$.
\end{proof}

Let us deduce a corollary from \cref{prop:count_balanced_spine_dim2} which will be useful in \cref{sec:proof_of_stability}.

\begin{definition} \label{def:rigid}
	We say an effective algebraic 1-cycle $\alpha$ on a projective variety is rigid if it is the only effective cycle in its numerical equivalence class.
\end{definition}

\begin{corollary} \label{cor:count_rigid}
	Notation as in \cref{prop:count_balanced_spine_dim2}.
	Let $\gamma$ be a rigid algebraic 1-cycle supported on boundary 1-strata of $D$.
	If the count $N(h,\gamma)\neq 0$, then $h$ is balanced and $\gamma=Z(h)$.
	In this case, $N(h,\gamma)=1$.
\end{corollary}
\begin{proof}
	Since $\gamma$ is rigid, if $f$ contributes to $N(h,\gamma)$ then the algebraic 1-cycle $Z(f)$ associated to $f$ is necessarily $\gamma$.
	Now the corollary follows from \cref{prop:count_balanced_spine_dim2}.
\end{proof}

We remark that in the 2-dimensional case, \cref{prop:umbrella} can also be deduced from \cite{Gross_Mirror_symmetry_for_log_Calabi-Yau_surfaces_I_v1} via the following comparison result:

\begin{proposition} \label{prop:alg_equal}
	For $(Y,D)$ 2-dimensional, our mirror algebra $A_Y$ coincides with the mirror algebra of \cite{Gross_Mirror_symmetry_for_log_Calabi-Yau_surfaces_I_v1}.
\end{proposition}
\begin{proof}
	This follows from the non-degeneracy of the Frobenius pairing (see \cite[Theorem 1.2(1)]{Keel_Yu_The_Frobenius}), and the comparison of Frobenius pairings (see \cite[Prop.\ 6.1]{Mandel_Theta_bases}, the issue is the equality between certain virtual and naive relative Gromov-Witten invariants).
\end{proof}

\section{Toric fiber bundles} \label{sec:toric_fiber_bundle}

In this section we describe a general construction of toric stacks that will be used in \cref{sec:extension_full}.
Along the way, we correct a mistake in \cite{Fulton_Introduction_to_toric_varieties} concerning toric fiber bundles.

\begin{notation} \label{nota:TV_Z}
	We denote by $\TV(\Delta,N)$ the toric scheme (over $\bbZ$) given by a fan $\Delta$ in a lattice $N$, writing $\TV(\Delta)$ if there is no ambiguity about $N$.
	We denote by $T_N$ the algebraic torus (over $\bbZ$) with cocharacter lattice $N$.
	We denote by $\TV(\Delta,N)_\bbC$ and $T_{N,\bbC}$ the base changes to $\bbC$.
	We will build the mirror family over $\bbZ$ in \cref{sec:family_over_secondary_fan}, but will need to base change everything to $\bbC$ in Sections \ref{sec:singularities} and \ref{sec:modularity}, where the subscript $\bbC$ denoting the base change will be dropped (cf.\ \cref{nota:base_change}).
\end{notation}

\begin{construction} \label{const:tfb}
	Let $\Delta$ be a rational polyhedral fan in a lattice $N$ and $\Delta' \subset \Delta$ a subfan.
	Let $L \subset N$ be a subgroup (the quotient $N/L$ is not necessarily torsion-free).
	Assume that for each $\sigma \in \Delta$, we have $\sigma=\sigma_1+\sigma_2$ where $\sigma_1 \subset L_\bbR$, $\sigma_2 \in \Delta'$, and $L_\bbR \cap \braket{\sigma_2}_\bbR=0$ (so that in particular $\sigma_1 \times \sigma_2 \rightarrow \sigma$ is a bijection).
	
	Note that $\sigma_1,\sigma_2$ can be recovered from $\sigma$, i.e.\ $\sigma_1 = \sigma \cap L_\bbR$, and $\sigma_2 \subset \sigma$ is the unique face containing every face whose span has zero intersection with $L_\bbR$.
	The collection of cones $\sigma_1 \subset L_\bbR$ over all $\sigma \in \Delta$ forms a fan $\Delta_L$ in $L$; and the analogous collection of $\sigma_2 \subset N_\bbR$
	forms a subfan of $\Delta'$.
	Let $\tN \coloneqq L \oplus N$, and let $\tDelta$ be the collection of cones $\sigma_1 + \sigma_2 \subset \tN_\bbR$, together with their faces.
	This is a subfan of the product fan $\Delta_L \times \Delta'$.
	
	We have an exact sequence
	$$0 \rightarrow L \xrightarrow{a} L \oplus N \xrightarrow{b} N \rightarrow 0$$
	where $a(l)=(-l,l)$ and $b(l,n)=l+n$.
	Define $\cTV(\Delta,N)\coloneqq[\TV(\tDelta,\tN)/T_L]$ as a stack where we use the inclusion
	$$T_L \subset T_L \times T_N, \quad t \mapsto (t^{-1},t)$$
	induced by $a$.
	The map $b$ gives a map of fans $\tDelta \to \Delta$, and so a map $\TV(\tDelta,\tN) \to \TV(\Delta,N)$.
	This is $T_L$-invariant by the exact sequence and so induces
	\[b\colon \cTV(\Delta,N)=[\TV(\tDelta,\tN)/T_L] \longto \TV(\Delta,N).\]
	
	The projection $\tN=L \oplus N \rightarrow N$ gives a map of fans $\tDelta \to \Delta'$, and so a map $\pi_N\colon\TV(\tDelta,\tN) \to \TV(\Delta',N)$.
	This is $T_L$-equivariant (where $L\subset\tN$ via $a$) and so induces a canonical representable map
	$$
	\widebar\pi_N\colon \cTV(\Delta,N) = [\TV(\tDelta,\tN)/T_L] \longrightarrow [\TV(\Delta',N)/T_L].
	$$
\end{construction}

\begin{proposition}\label{prop:tfb}
	The following hold:
	\begin{enumerate}
		\item $\cTV(\Delta,N)$ is a Deligne-Mumford stack, and $b\colon \cTV(\Delta,N) \to \TV(\Delta,N)$ is its coarse moduli space.
		\item The map $b$ is an isomorphism over $\TV(\Delta',N)\subset\TV(\Delta,N)$.
		\item The composition
		$$
		\TV(\Delta',N) \subset \cTV(\Delta,N) \xrightarrow{\ \widebar\pi_N\ } [\TV(\Delta',N)/T_L]
		$$
		is the canonical quotient map.
	\end{enumerate}
\end{proposition}

\begin{proof}
	Let $\tsigma\subset \tN_\bbR$ be a cone and $\sigma \subset N_\bbR$ its image under $b_\bbR$.
	Assume $\tsigma \cap a(L)_\bbR = 0$ so that $\tsigma \rightarrow \sigma$ is a bijection.
	Then $T_L$ acts with finite stabilizers on $\TV(\tsigma,\tN)$, and the coarse moduli space is $\TV(\sigma,N)$.
	This shows (1).
	
	For (2), observe that
	\[b^{-1}(\TV(\Delta',N))\simeq\TV(\{0\}\times\Delta',\tN)\simeq T_L\times\TV(\Delta',N),\]
	so $T_L$ acts freely on $b^{-1}(\TV(\Delta',N))$, hence (2) holds.
	(3) is obvious.
\end{proof}

\begin{remark} \label{rem:L}
	If we replace $L$ by a finite index subgroup $L'\subset L$, then the quotient stack $\cTV(\Delta,N)=[\TV(\tDelta,\tN)/T_L]$ will change, while its coarse moduli space $\TV(\Delta,N)$ stays the same.
\end{remark} 

\begin{remark} \label{rem:stabilizer}
	The stabilizer $A$ along the toric stratum of $\TV(\tDelta,\tN)$ corresponding to a cone $\tsigma=\sigma_1 \times \sigma_2 \in \tDelta$ is the kernel of the composition
	$$T_L \rightarrow T_L \times T_N \rightarrow T_{L/N_1} \times T_{N/N_2}$$
	where $N_1 \subset L$ and $N_2 \subset N$ are the subgroups generated by $\sigma_1 \cap L$ and $\sigma_2 \cap N$ respectively.
	Thus $A \simeq T_{N_1} \cap T_{N_2} \subset T_N$.
	
	If we base change to $\bbC$, the group $A_\bbC$ is isomorphic to the torsion group of the quotient $N /(N_1 \oplus N_2)$ (recall that by assumption $N_1 \cap N_2 = \{0\}$).
	Indeed, applying $(\cdot )\otimes_{\bbZ} \bbC^\times$ to the exact sequence
	$$0 \rightarrow N_1 \oplus N_2 \rightarrow N \rightarrow N/(N_1 \oplus N_2) \rightarrow 0$$
	gives
	\[A_\bbC\simeq \Tor^1_{\bbZ}(N/(N_1 \oplus N_2),\bbC^\times)\simeq\Tors N/(N_1\oplus N_2).\]
	\end{remark}

The following corollary corrects a mistake in the toric fiber bundle construction of \cite[p.\ 41]{Fulton_Introduction_to_toric_varieties}.
The fiber bundle claim in loc.\ cit.\ is wrong when the stabilizer group in \cref{rem:stabilizer} is non-trivial.

\begin{corollary}
	Let $0\to\ L\to N\to\oN\to 0$ be a short exact sequence of lattices, and let $\Delta_L$, $\Delta$ and $\oDelta$ be fans in $L$, $N$ and $\oN$ respectively that are compatible with the maps of lattices.
	Suppose there is a fan $\Delta'$ in $N$ that lifts $\oDelta$ such that the cones $\sigma\in\Delta$ are exactly $\sigma_1+\sigma_2$ with $\sigma_1\in\Delta_L$ and $\sigma_2\in\Delta'$.
	Then we may apply \cref{const:tfb}.
	In this case, the stack $[\TV(\Delta',N)/T_L]$ is a Deligne-Mumford stack with coarse moduli space $\TV(\oDelta,\oN)$, and the representable map $\widebar\pi_N\colon \cTV(\Delta,N) \rightarrow [\TV(\Delta',N)/T_L]$ is a $\TV(\Delta_L,L)$-bundle (for the étale topology).
\end{corollary}
\begin{proof}
	Since $L\cap\Delta'=0$, $T_L$ acts with finite stabilizers on $\TV(\Delta',N)$ and the quotient stack $[\TV(\Delta',N)/T_L]$ is a Deligne-Mumford stack with coarse moduli space $\TV(\oDelta,\oN)$.
	Next, note that the $T_L$-equivariant morphism $\pi_N\colon\TV(\tDelta, \tN) \rightarrow \TV(\Delta',N)$ is a $\TV(\Delta_L,L)$-bundle (for the Zariski topology), which is trivial over $\TV(\sigma,N)$ for each cone $\sigma \in \Delta'$.
	Taking quotient by $T_L$ on both sides, we deduce the second statement of the corollary, completing the proof.
\end{proof}

\begin{corollary} \label{cor:extend}
	Notation as in \cref{prop:tfb}, consider the $T_L$-action on $\cTV(\Delta,N)\allowbreak\to\TV(\Delta,N)$ induced by $L\hookrightarrow L\oplus N$, $l\mapsto(0,l)$.
	Then any $T_L$-equivariant family of varieties $\cX' \rightarrow \TV(\Delta',N)$ extends to a $T_L$-equivariant representable morphism $\cX \rightarrow \cTV(\Delta,N)$ of Deligne-Mumford stacks, where
	\[\cX \coloneqq \cTV(\Delta,N) \times_{[\TV(\Delta',N)/T_L]} [\cX'/T_L].\]
	
	Furthermore, assume there is a $T_L$-equivariant line bundle $\cL' \to \cX'$ with a section $\sigma'$ which is a $T_L$-eigensection of weight $\chi \in L^*$.
	Then there is a line bundle $\cL \to \cX$ extending $\cL'$, and $\sigma'$ extends to a section of $\cL$ if $\chi \in H^0(T_L,\cO)$ extends to $\TV(\Delta_L,L)$.
\end{corollary}

\begin{proof}
	By the $T_L$-equivariance, $\cX' \to \TV(\Delta',N)$ is pulled back from $[\cX'/T_L] \to [\TV(\Delta',N)/T_L]$.
	Therefore, the formula of $\cX$ in the corollary gives the required extension.	
	
	The same construction works for the total space of $\cL' \to \cX'$.
	Recall the projection $\pi_N\colon\TV(\tDelta,\tN)\allowbreak\to\TV(\Delta',N)$.
	If $\chi\in H^0(T_L,\cO)$ extends to $\TV(\Delta_L,L)$, we define $\tsigma \coloneqq \pi_N^*(\sigma') \cdot \chi \in H^0\big(\TV(\tDelta,\tN),\allowbreak \pi_N^*\cL'\big)$.
	This is $T_L$-invariant (where $L\subset\tN$ via $a$), so it descends to a section $\sigma\in H^0(\cTV(\Delta,N),\cL)$, which by construction extends the given section $\sigma'$.
\end{proof}

\section{The mirror family over the toric variety for the secondary fan } \label{sec:family_over_secondary_fan}

In this section we construct the mirror family over the toric variety for the secondary fan $\Sec(K)$.
First we construct the mirror algebra $A_{K_\alpha}$ for every SQM $K\dasharrow K_\alpha$ (see \cref{prop:A_K}), and we relate $A_K$ with $A_Y$ in \cref{prop:relation_AKAY}.
Next we glue the mirror algebras $A_{K_\alpha}$ to get a family over the toric variety for $\MovSec(K)$, the moving part of the secondary fan.
This is achieved by rephrasing the mirror algebras using universal torsor as in \cref{sec:rephrase}.
Finally we extend the mirror family over the toric variety for the full secondary fan, using the equivariant boundary torus action, as well as the toric fiber bundle construction of \cref{sec:toric_fiber_bundle}.

\subsection{The mirror algebra for \texorpdfstring{$K$}{K}} \label{sec:A_K}

We follow \cref{nota:K} and \cref{rem:Gamma}.
We have natural identifications
\[
\NE(Y,\bbZ) \simeq \NE(K,\bbZ) \simeq \NE(\cP/\ocP,\bbZ) \subset \NE(\cP,\bbZ)
\]
generated by curves on $Y\simeq Y_0\subset\cP$, and so natural subrings
\[R_Y\coloneqq\bbZ[\NE(Y,\bbZ)]\simeq R_K\coloneqq\bbZ[\NE(K,\bbZ)]\subset \bbZ[\NE(\cP,\bbZ)].\]
For any SQM $K\dasharrow K_\alpha$ over $\oK$, we denote $R_{K_\alpha}\coloneqq\bbZ[\NE(K_\alpha,\bbZ)]$.

\begin{proposition} \label{prop:A_K}
	For any SQM $K\dasharrow K_\alpha$ over $\oK$, consider the multiplication rule for the mirror algebra $A_{\cP_\alpha}$ associated to $V\subset \cP_\alpha$.
	For any $P_1,\dots,P_n\in \Gamma(\bbZ)$, $Q\in\Sk(V,\bbZ)$ and $\gamma\in\NE(\cP_\alpha,\bbZ)$ such that the structure constant $\chi(P_1,\dots,P_n,Q,\gamma) \neq 0$, we have $Q \in \Gamma(\bbZ)\subset\Sk(V,\bbZ)$ and $\gamma\in \NE(K_\alpha,\bbZ) \subset \NE(\cP_\alpha,\bbZ)$.
	Consequently, the free $R_{K_\alpha}$-submodule of $A_{\cP_\alpha}$ with basis $\theta_P$, $P\in\Gamma(\bbZ)$ is an $R_{K_\alpha}$-subalgebra, which we denote by $A_{K_\alpha}$.
\end{proposition}

For the proof, we introduce Lemmas \ref{lem:balancing}-\ref{lem:structure_disk_in_K}.

\begin{lemma} \label{lem:balancing}
	Let $C$ be a connected quasi-smooth $k$-analytic curve and $f_1,\dots,f_n$ regular functions on $C$.
	Let $F\coloneqq \min_i\{-\log\abs{f_i}\}\colon C\to(-\infty,+\infty]$.
	Let $S\subset C$ be the convex hull of $\partial C$.
	If $F|_{S}$ is not constant, then there exists a point $v\in\partial C$ and an edge of $S$ connected to $v$ such that $F$ achieves its minimum at $v$ and that the derivative $d_v F|_e > 0$.
\end{lemma}
\begin{proof}
	Denote $F_i\coloneqq-\log\abs{f_i}$.
	By \cite[(4.5.21)]{Ducros_La_structure_des_courbes_analytiques}, there exists a subgraph $\Gamma\subset C$ containing $S$ such that each $F_i$ is balanced (aka harmonic) at every finite vertex $v$ of $\Gamma\setminus\partial C$, i.e.\
	\[\sum_{e\ni v} d_v F_i|_e =0\]
	summing over all edges of $\Gamma$ containing $v$.
	Hence
	\begin{equation} \label{eq:balancing}
	\sum_{e\ni v} d_v F|_e \le 0.
	\end{equation}
	By the maximum modulus principle, $F$ achieves its minimum $F_\mathrm{\min}$ at $\partial C$.
	Let $T\subset S$ be the locus where $F|_S=F_\mathrm{min}$.
	Since $F|_S$ is not constant, there exists a domain of affineness $e\subset S$ of $F|_S$ with one endpoint $v\in T$ and another endpoint $w\notin T$.
	Then by \eqref{eq:balancing}, $v$ belongs to $\partial C$, completing the proof.
\end{proof}

\begin{lemma} \label{lem:non-decreasing}
	Notation as in \cref{lem:balancing}.
	Assume $C$ is rational.
	Fix $r\in\partial C$ and assume that for any $v\in\partial C\setminus r$ and any edge $e$ of $S$ containing $v$, we have $d_v F|_e\le 0$.
	Then $F$ is non-decreasing along any simple path from $r$ to any $x\in S$.
	In particular $F$ attains its minimum at $r$.
\end{lemma}
\begin{proof}
	Since $F$ is piecewise affine on $S$, it suffices to prove that for any domain of affineness $e$ of $F|_S$, the derivative of $F|_e$ pointing away from $r$ is non-negative.
	Choose any point $v$ in the interior of $e$ and cut $S$ at $e$.
	Since $C$ is rational, $S$ is a tree, so we obtain two parts $S_1$ and $S_2$, say $S_1\ni r$.
	Let $e_2$ be the edge of $S_2$ containing $v$.
	Then it suffices to show that the derivative $d_v F|_{e_2}\ge 0$.
	Let $C_2$ be the preimage of $S_2$ by the retraction map $C\to S$.
	We achieve the proof by applying \cref{lem:balancing} to $C_2$.
\end{proof}

\begin{lemma} \label{lem:non-decreasing_semiample}
	Assume $k$ has non-trivial valuation, let $\fC$ be a rational semistable formal curve over the ring of integers $\kc$ and $C\coloneqq\fC_\eta$.
	Let $\fL$ be a line bundle on $\fC$ and $t$ a nonzero rational section of $\fL$ that is regular on $C$.
	Note that $F\coloneqq-\log\abs{t}\colon C\to (-\infty,+\infty]$ is well-defined, since different local trivializations of $\fL$ do not change the norm $\abs{t}$.
	Let $S\subset C$ be the convex hull of $\partial C$ and fix $r\in\partial C$.
	Assume that $\fL^*$ is semi-ample, and that for any $v\in\partial C\setminus r$ and any edge $e$ of $S$ containing $v$, we have $d_v F|_e\le 0$.
	Then $F$ is non-decreasing along any simple path from $r$ to any $x\in S$.
	In particular $F$ attains its minimum at $r$.
\end{lemma}
\begin{proof}
	Since $\fL^*$ is semi-ample, some positive tensor power of $L$ is a subbundle of a trivial vector bundle.
	Since we can replace $t$ by a positive tensor power without changing the maximum locus of $\abs{t}$, we may assume that $\fL\subset\cO_\fC^m$.
	Then $t\colon\cO_\fC\dasharrow L\subset\cO_\fC^m$ is given by $m$ rational functions $t_1,\dots,t_n$ on $\fC$.
	Note that $\abs{t}=\max\abs{t_i}$ on $C$, so the result follows from \cref{lem:non-decreasing}.
\end{proof}

\begin{lemma} \label{lem:structure_disk_in_K}
	Given $P_1,\dots,P_n,Q\in\Gamma(\bbZ)$ and $\gamma\in\NE(\cP_\alpha,\bbZ)$.
	Let \[f\colon[\bbD,(p_1,\dots,p_n,s)]\longto\cP_\alpha^\an\]
	be a structure disk contributing to $\chi(P_1,\dots,P_n,Q,\gamma)$ defined over a field with discrete valuation such that $f(s)$ is a general rational point in $\Gamma$ near $Q$.
	Let $\cC$ be $\bbD$ minus $n$ small open disks centered at $p_1,\dots,p_n$ respectively.
	Up to passing to a finite base field extension, let $\fC$ be a strictly semistable formal model of $\cC$ such that $f|_\cC$ extends to $\ff\colon\fC\to\widehat{(\cP_\alpha)_{\kc}}$.
	Then $\ff_s(\fC_s) \subset K_\alpha\subset \cP_\alpha$.
\end{lemma}
\begin{proof}
	Since the exceptional locus of $\cP\dasharrow\cP_\alpha$ is contained in $Y\subset K$, it suffices to prove the lemma for $K_\alpha=K$.
	We apply \cref{lem:non-decreasing_semiample} to $\fC$ and the pullback of the Cartier divisor $\delta$ (notation as in \cref{rem:Gamma}).
	The conditions of the lemma on $d_v F$ (notation from the lemma) are satisfied because $P_1,\dots,P_n\in\Gamma(\bbZ)$.
	Let $r$ be the Gauss point of $\bbD$.
	By \cite[Lemma 8.22]{Keel_Yu_The_Frobenius}, we have $f(r)=f(s)=Q\in\Gamma(\bbZ)$.
	Therefore, \cref{lem:non-decreasing_semiample} implies that $f|_\cC\colon\cC\to V^\an$ has image in $\{\abs{\delta}\leq 1\}\subset V^\an$.
	Hence the pullback of $\delta$ to $\fC$ is effective.
	As the zeros and poles of $\delta = Y_0 - Y_{\infty}$ on $\cP$ are disjoint, we deduce that $\ff_s(\fC_s)\subset K=\cP\setminus Y_\infty$.
\end{proof}

\begin{proof}[Proof of \cref{prop:A_K}]
	By \cite[Lemma 8.21(1)]{Keel_Yu_The_Frobenius}, the structure disks miss any codimension two subset of $K_\alpha$, so it suffices to prove $Q\in\Gamma(\bbZ)$ for $K_\alpha=K$.
	Since $-K$ is ample on $Y$, the divisors $-Y_0$ and $Y_\infty$ are nef on $\cP$, in particular $-\delta$ is nef on $\cP$.
	Then it follows from \cite[Theorem 15.7(2)]{Keel_Yu_The_Frobenius} that $Q \in \Gamma(\bbZ)$.
	Next we apply \cref{lem:structure_disk_in_K}, by choosing $\cC$ sufficiently close to $\bbD$.
	By \cite[Lemma 14.6]{Keel_Yu_The_Frobenius}, $\gamma$ is equal to the pushforward $\ff_{s,*}$ of the proper part of $\fC_s$.
We conclude that $\gamma\in\NE(K_\alpha,\bbZ)$.
\end{proof}

Next let us relate the two mirror algebras $A_K$ and $A_Y$.
Notation as in \cref{rem:Gamma}, $d^\trop\colon\Sk(V,\bbZ)\allowbreak\to\bbZ$ gives an $\bbN$-grading on the $R_K$-algebra $A_K$.
We denote $\cX\coloneqq\Proj(A_K)$.
The divisor $D\subset Y$ gives a filtration on $A_Y$ with $A_{\le n}$ having basis $\theta_P$, $D^\trop(P)\le n$.
By \cite[Theorem 15.7(2)]{Keel_Yu_The_Frobenius}, we have $A_{\le m}\cdot A_{\le n}\subset A_{\le m+n}$.
So we obtain a graded $R_Y$-algebra $\tA_Y\subset A_Y[T]$, having basis $\theta_P\cdot T^n$ with $D^\trop(P)\le n$.

\begin{proposition} \label{prop:relation_AKAY}
	The isomorphism of graded $R_Y$-modules $A_K \simeq \tA_Y$ induced by \eqref{eq:Gamma} is an isomorphism of graded $R_Y$-algebras.
\end{proposition}
\begin{proof}
	The same argument in the proof of \cref{thm:reformulation_universal_torsor} shows that the isomorphism of graded $R_Y$-modules is compatible with the multiplication rule.
\end{proof}

\begin{remark}
	Given the isomorphism $A_K \simeq \tA_Y$, one might well wonder, why bother with $K \to Y$, instead of just working with $Y$? The advantage of $K$ comes when we want to extend the family $\Proj\tA_Y\to\Spec R_Y$ to (the correct) compact base, which turns out to be the toric variety associated to the secondary fan for $K/\oK$ rather than for $Y$.
	One can already see this in the toric case: There are maximal cones in the secondary fan of a polytope (reflexive in the Fano case) for each coherent triangulation, which are dual complexes of a natural boundary, not on $Y$, but on the total space of a line bundle ($K$ in the Fano case).
	In the toric Fano case, we have $\Sec(K) \simeq \MoriFan(K)$ (which is not the Mori fan for $Y$).
\end{remark}

\subsection{Extension to the moving part of the secondary fan} \label{sec:extension_moving}

We fix a sublattice $\sfN\subset\Pic(K)\simeq\Pic(K/\oK)\simeq\Pic(Y)$ as in \cref{ass:2}, with $\sfM$ the dual lattice.
When applying to the proof of \cref{thm:del_Pezzo_intro} where $Y$ is del Pezzo, we will simply take $\sfN=\Pic(K)\simeq\Pic(Y)$, then $\sfM\simeq N_1(Y)$.
We follow the notations in \cref{rem:Gamma}.

\begin{proposition} \label{prop:AK_in_A}
	Let $A \subset A_\cT$ be the sub-$\bbZ$-module with basis $\tGamma(\bbZ) \subset \Sk(\cT,\bbZ)$, it is a $\bbZ[\sfM]$-subalgebra.
	For any SQM $K\dasharrow K_\alpha$ over $\oK$, we have an isomorphism of rings
	\[A_{K_\alpha}\otimes_{R_{K_\alpha}}\bbZ[\sfM]\xrightarrow{\ \sim\ } A,\]
	where the image of $A_{K_\alpha}$ is the sub-$\bbZ$-module of $A$ spanned by the basis elements $\theta_P$ over all $P\in\tGamma(\bbZ)$ that are $\NE(K_\alpha)_\sfM$-above $\varphi_{K_\alpha}\colon\Gamma\to\tGamma$ (cf.\ \cref{thm:reformulation_universal_torsor}).
\end{proposition}
\begin{proof}
	This follows from \cref{thm:reformulation_universal_torsor} and \cref{prop:A_K}, using \cref{rem:wave}.
\end{proof}

\begin{notation} \label{nota:dual_cone}
	For any cone $\sigma \subset \MovSec(K)$, we denote by $P_\sigma\subset \sfM$ the monoid of integer points in the dual cone.
\end{notation}

\begin{proposition} \label{prop:disk_class_A_K}
	Let $\alpha\subset\MovFan(K)$ be a maximal cone, and $\gamma\coloneqq\sec(\alpha)$ the union of maximal cones $\beta$ with $\varphi_\alpha=\varphi_\beta$ (as in \cref{thm:convex_cone}).
	Then the curve class of any structure disk for $A_{K_\alpha}$ lies in $P_\gamma$.
\end{proposition}
\begin{proof}
	Let $\beta\subset\MovFan(K)$ be a maximal cone with $\varphi_\alpha=\varphi_\beta$.
	Since $K_\alpha\dasharrow K_\beta$ is an isomorphism along $V$, the moduli spaces of structure disks for $A_{K_\alpha}$ are naturally isomorphic to those for $A_{K_\beta}$.
	The only question is whether the classes of structure disks, computed in $K_\alpha$ or $K_\beta$, are the same.
	This follows from \cref{lem:class_via_varphi}, which implies that the class of a structure disk is determined by $\varphi_\alpha = \varphi_\beta$ and the (punctured) structure disk in $V$.
	Consequently, the class of any structure disk for $A_{K_\alpha}$ lies in $P_\gamma$.
\end{proof}

\begin{corollary} \label{cor:Agamma}
	The algebra $A_{K_\alpha}$ is naturally a base extension of an $R_{\gamma} \coloneqq \bbZ[P_{\gamma}]$-algebra, which we denote by $A_{\gamma}$.
	We have a canonical inclusion $A_{\gamma} \subset A$, realizing it as the sub-$\bbZ$-module with basis the points of $\tGamma(\bbZ)$ that are $P_\gamma$-above $\varphi_{\gamma}$.
\end{corollary}
\begin{proof}
	The first statement follows from \cref{prop:disk_class_A_K}, and the second from \cref{prop:AK_in_A}.
\end{proof}

\begin{lemma} \label{lem:glue_Aalpha}
	For any two maximal cones $\alpha,\beta$ of $\MovSec(K)$, we have
	\[
	A_\alpha \otimes_{\bbZ[P_\alpha]} \bbZ[P_{\alpha \cap \beta}] = A_\beta \otimes_{\bbZ[P_\beta]} \bbZ[P_{\alpha \cap \beta}] \subset A,
	\]
	with basis the points of $\tGamma(\bbZ)$ that are $P_{\alpha \cap \beta}$-above $\varphi_\alpha$ (or equivalently $\varphi_\beta$).
\end{lemma}
\begin{proof}
	By \cref{prop:SQM}(\ref{prop:SQM:difference}), the difference $\varphi_\beta-\varphi_\alpha\colon\Sk(V)\to\sfM_\bbR$ has image in $P_\alpha$.
	So the set of points of $\tGamma(\bbZ)$ that are $P_{\alpha\cap\beta}$-above $\varphi_\alpha$ is equal to the set of points that are $P_{\alpha\cap\beta}$-above $\varphi_\beta$.
	Now we conclude from \cref{cor:Agamma}.
\end{proof}

Denote $\MovSec\coloneqq\MovSec(K)$ for simplicity.
We have
\[\Spec\bbZ[\NE(K)_\sfM]=\TV(\Nef(K),\sfN)\subset\TV(\MovSec,\sfN).\]
Now we can extend the family $\Spec A_K\to\Spec R_K$ (resp.\ $\Proj A_K\to\Spec R_K$) over the bigger base $\TV(\MovSec,\sfN)$, by patching together all $\Spec A_\gamma$ (resp.\ $\Proj A_\gamma$) in the same way the open affine subsets $\Spec\bbZ[P_{\gamma}]$ of $\TV(\MovSec,\sfN)$ are glued:

\begin{proposition} \label{prop:glue_over_MovSec}
	Via \cref{lem:glue_Aalpha}, $\Spec A_\gamma$ (resp $\Proj A_\gamma$), for all maximal cones $\gamma$ of $\MovSec$, glue to give a family $\hcX \to \TV(\MovSec,\sfN)$ (resp.\ $(\cX,\cO(1)) \to \TV(\MovSec,\sfN)$).
\end{proposition}

Next we describe the canonical theta functions on the family $(\cX,\cO(1)) \to \TV(\MovSec,\sfN)$.
We cannot expect each $P\in\Gamma(\bbZ)$ to give a global section of some $\cO(n)$, otherwise the structure constants would be global functions on the base and the family would be trivial.
Instead, we need to twist $\cO(n)$ by the pullback of a toric line bundle $L_P$ on $\TV(\MovSec,\sfN)$ which we now describe.

\begin{definition-lemma} \label{lem:L_P}
Let $\alpha,\beta \in \MovSec$ be two maximal cones, and $P\in\Sk(V,\bbZ)$.
Let
$$
C_{\alpha,\beta}^P \coloneqq \varphi_\alpha(P) - \varphi_\beta(P) \in \sfM.
$$
Then $C^P_{\alpha,\beta} \in P_{\alpha \cap \beta}^{\times}$, see \cref{nota:dual_cone}, where the superscript $^\times$ denotes the group of invertible elements.
Let $U_\alpha, U_\beta\subset\TV(\MovSec,\sfN)$ denote the associated toric affine subvarieties.
For fixed $P$, the collection of invertible functions
$$
z^{C^P_{\alpha,\beta}} \in H^0(V_\alpha \cap V_\beta,\cO^{\times})
$$
gives a Čech $1$-cocycle.
We write $L_P$ for the associated toric line bundle on $\TV(\MovSec,\sfN)$.
\end{definition-lemma}

\begin{proof}
By \cref{prop:SQM}(\ref{prop:SQM:Lperp}),
$$
(\varphi_\alpha(P) - \varphi_\beta(P) )\cdot L = 0
$$
for all $L \in \alpha \cap \beta$, so we have $C^P_{\alpha,\beta} \in P_{\alpha \cap \beta}^{\times}.$
The cocycle condition for $z^{C^P_{\alpha,\beta}}$ holds because $C^P_{\alpha,\beta}$ is defined by differences of sections.
\end{proof}

\begin{proposition} \label{prop:glue_theta_function}
	Given $P \in \Gamma(\bbZ)$, the theta functions $\theta_{\varphi_\alpha(P)} \in A_\alpha$ over all maximal cones $\alpha\in\MovSec$ glue to a canonical section $\theta_P$ of the pullback of $L_P$ to $\hcX$, inducing a canonical section of $\cO(n) \otimes\pi^*(L_P)$ on $\cX$, with $n\coloneqq d^\trop(P) \in \bbN$ and $\pi\colon\cX\to\TV(\MovSec,\sfN)$.
\end{proposition}
\begin{proof}
	This follows from \cref{prop:glue_over_MovSec} and \cref{lem:L_P}.
\end{proof}

\begin{definition} \label{def:Theta}
	By \cref{lem:varphieq}, for $P \in \partial \Gamma(\bbZ)$ or $P = [Y]$, the line bundle $L_P$ of \cref{lem:L_P} is trivial, thus the canonical section $\theta_P$ of \cref{prop:glue_theta_function} gives a section $\theta_P\in H^0(\cX,\cO(1))$.
	In particular we have a canonical section for each $P \in \Lambda(\bbZ) \subset \Gamma(\bbZ)$, notation as in \cref{rem:Gamma}.
	We define $\Theta\subset\cX$ to be the zero scheme of the section  $\sum_{P\in \Lambda(\bbZ)} \theta_P\in H^0(\cX,\cO(1))$.
	\end{definition} 

The family $\cX\to\TV(\MovSec,\sfN)$ contains a natural boundary divisor $\cE$ as follows:

For each maximal cone $\alpha\in\MovSec$, we denote by $I_\alpha\subset A_\alpha$ the free $R_\alpha\coloneqq\bbZ[P_\alpha]$-submodule with basis the integer points $\Gamma^\circ(\bbZ)$ in the interior of $\Gamma$.

\begin{lemma} \label{lem:boundary_E}
	The $R_\alpha$-submodule $I_\alpha \subset A_\alpha$ is an ideal.
	The quotient $A_\alpha/I_\alpha$ is a free $R_\alpha$-module with basis $\theta_P$, $P \in \partial\Gamma(\bbZ)$.
	The multiplication rule on $A_\alpha/I_\alpha$ is {\it constant}, more precisely, $\theta_P \cdot \theta_Q$ is an integer combination of various $\theta_R$.

	If $(Y,D)$ satisfies \cite[Assumption 2.4]{Keel_Yu_The_Frobenius}, then for $P,Q\in\partial\Gamma(\bbZ)$, $\theta_P\cdot\theta_Q=\theta_{P+Q}$ if $P,Q$ lie in a common cone of $\partial\uGamma$ (notation as in \cref{rem:Gamma}), and $\theta_P\cdot\theta_Q=0$ otherwise.
\end{lemma}
\begin{proof}
	To show that $I_\alpha$ is an ideal, it suffices to check that if $\theta_R$ appears with non-zero coefficient in the product $\theta_P \cdot \theta_Q$, with $\delta^\trop(P) > 0$ and $\delta^\trop(Q) \geq 0$, then we have $\delta^\trop(R) > 0$.
	This is independent of $\alpha$, so it is enough to check in $A_K$.
	Then it follows from the nefness of $-\delta$ using \cite[Theorem 15.7(2)]{Keel_Yu_The_Frobenius}.
	The second statement follows from the first.
	
	To show that the multiplication rule on $A_\alpha/I_\alpha$ is constant, by \cref{prop:AK_in_A} and \cref{lem:varphieq}, it suffices to prove it for $K$.
	It is equivalent to the statement that if $\theta_R$ appears with non-zero coefficient in the product $\theta_P \cdot \theta_Q$ with $\delta^\trop(P)=\delta^\trop(Q)=\delta^\trop(R)=0$, then every contributing structure disk has trivial curve class.
	This follows from \cref{prop:relation_AKAY} and the ampleness of $-D$, using \cite[Theorem 15.7(1)]{Keel_Yu_The_Frobenius}.
	
	Finally, under \cite[Assumption 2.4]{Keel_Yu_The_Frobenius}, the explicit multiplication rule follows from \cite[Theorem 15.7(3)]{Keel_Yu_The_Frobenius}.  
\end{proof}

\begin{proposition} \label{prop:boundary_E}
	The $\Spec A_\alpha/I_\alpha$ over all maximal cone $\alpha\in\MovSec$ glue to a subscheme $\cE \subset \cX$.
	It is a trivial family over $\TV(\MovSec,\sfN)$, i.e.\ we have
	\[(\cE,\cO(1)) \simeq (X, \cO(1)) \times \TV(\MovSec,\sfN),\]
	for a polarized projective scheme $X$, moreover every theta function is pulled back from $X$.
	
	If $(Y,D)$ satisfies \cite[Assumption 2.4]{Keel_Yu_The_Frobenius}, then $(X,\cO(1))$ is isomorphic to the polarized broken toric variety given by the simplicial complex $\partial\uLambda$, i.e.\ the Proj of the associated graded Stanley-Reisner ring (see \cite[Definition 1.6]{Miller_Combinatorial_commutative_algebra}).
\end{proposition}
\begin{proof}
	This is immediate from \cref{lem:boundary_E}.
\end{proof}

\begin{lemma} \label{lem:boundary_E_dim2}
	If $Y$ has dimension two, then each fiber $(X,\cO(1))$ of the polarized constant boundary family is a cycle of rational curves, with polarization of degree one on each irreducible component.
	Moreover, the divisor $\Theta$ does not contain any nodes of $X$.
\end{lemma}
\begin{proof}
	By \cref{prop:boundary_E}, it suffices to compute the fiber over the unique 0-stratum of the base $\TV(\Nef(Y))$.
	Then the result follows from Propositions \ref{prop:relation_AKAY} and \ref{prop:umbrella}.
\end{proof} 

\subsection{The boundary torus action} \label{sec:torus_action}

Here we describe a natural torus action on the mirror family $\cX\to\TV(\MovSec,\sfN)$ constructed in \cref{prop:glue_over_MovSec}.

Recall $\sfN=n\Pic(K)\subset\Pic(K)$.
Let $\bbZ_{\Lambda(\bbZ)}$ be the free abelian group with basis the integer points $\Lambda(\bbZ)$, notation as in \eqref{eq:Lambda}, and $\bbZ^{\Lambda(\bbZ)}$ it dual.
Let $\sfL\coloneqq n\bbZ_{\Lambda(\bbZ)}\subset\sfN$, and $\sfL^*$ its dual.
Let $T^\Lambda$ (resp.\ $T_\Lambda$) be the split torus with character (resp.\ cocharacter) lattice $\sfL$.

Let $w$ denote the composition
\begin{equation} \label{eq:w}
	\Gamma(\bbZ)\simeq\Sigma_{(V\subset K)}(\bbZ)\subset\bbZ^{\Lambda(\bbZ)}\to\sfL^*.
\end{equation}
We denote by the same letter $w\colon\sfM\to\sfL^*$, the dual of $\sfL\to\sfN$.
For any SQM $K\dasharrow K_\alpha$ over $\oK$, by \cite[Theorem 16.2]{Keel_Yu_The_Frobenius}, $T_\Lambda$ acts diagonally on the mirror algebra $A_{K_\alpha}$ with weight $w(P)+w(\gamma)$ on the basis vector $z^\gamma\theta_P$.
Now we check that the actions are compatible with respect to the gluing in \cref{prop:glue_over_MovSec}.

The points in $\Lambda(\bbZ)$ correspond to the irreducible components of the boundary divisor $K\setminus V$.
As in the proof of \cref{lem:recover_dual_complex}, they give rise to a function $\cT\to T^\Lambda$, which tropicalizes to $W\colon\Sk(\cT)\to\sfL^*_\bbR$; moreover, for each SQM $K \dasharrow K_\alpha$ over $\oK$, the composition $W\circ\varphi_\alpha\colon\Gamma(\bbZ)\to\sfL^*$ is equal to the weight function $w\colon\Gamma(\bbZ)\to\sfL^*$ in \eqref{eq:w}.
Consequently, we obtain the following:

\begin{proposition}
	The $T_\Lambda$-action on $A$ induced by the isomorphism $A \simeq A_{K_\alpha}\otimes_{R_{K_\alpha}}\bbZ[\sfM]$ in \cref{prop:AK_in_A} is independent of the choice of SQM $K\dasharrow K_\alpha$ over $\oK$;
	so we obtain an equivariant action of $T_\Lambda$ on the family $\cX\to\TV(\MovSec,\sfN)$.
\end{proposition}

\subsection{Extension to the full secondary fan} \label{sec:extension_full}

In \cref{sec:extension_moving}, we constructed the mirror family $(\cX,\cE,\Theta)$ over the toric variety associated to the moving part of the secondary fan $\MovSec(K)$.
Next we further extend this family to the toric variety associated to the full secondary fan $\Sec(K)$.

Recall from \cref{lem:bogus_K} that each bogus cone of $\Sec(K)$ is of the form $\hgamma\coloneqq\gamma+\bbR_{\ge 0}[Y]$ for a cone $\gamma$ of $\MovSec(K)$ lying in the boundary of the moving cone $\Mov(K)$.

Here is the basic idea for the extension over a bogus cone $\hgamma$.
Assume for simplicity that $\sfN=\Pic(K)$.
Let $\braket{\gamma}_\bbR\subset\sfN_\bbR\simeq\Pic(K)_\bbR$ denote the span of $\gamma$, and let $\braket{\gamma}(\bbZ)\coloneqq\braket{\gamma}_\bbR\cap\sfN$.
If $\braket{\gamma}(\bbZ)$ and $[Y]$ generate $\sfN$,
then the affine toric open subset of $\TV(\Sec(K))$ corresponding to the bogus cone is isomorphic to
\[
\TV(\gamma) \times \TV(\bbR_{\ge 0} [Y]) \simeq \TV(\gamma) \times \bbA^1.
\]
The structure torus of the $\bbA^1$-factor is $T_{\braket{[Y]}} \subset T_\Lambda$ (notation as in \cref{sec:torus_action}). 
By the $T_\Lambda$-equivariance, the mirror family restricted to $\TV(\gamma) \times \Gm$ is isomorphic to a product $(\cX_\gamma,\cE_\gamma)\times\Gm$.
Under the $T_{\braket{[Y]}}$-action, $\theta_{[Y]}$ scales by $\lambda\in\Gm\simeq T_{\braket{[Y]}}$, while the other terms of $\Theta$ of \cref{def:Theta} remain constant.
Hence we can extend the family over $\TV(\gamma)\times\{0\}$ by the product $(\cX_\gamma,\cE_\gamma)\times\{0\}$, with the $\theta_{[Y]}$ summand of $\Theta$ set to zero.

However, in general, $\braket{\gamma}(\bbZ)$ and $[Y]$ generate only a finite-index sublattice of $\Pic(K)$, so the toric open subset of $\TV(\Sec(K))$ corresponding to the bogus cone is not literally a product,
and moreover the base toric varieties are defined with respect to the sublattice $\sfN \subset \Pic(K)$ instead of the whole Picard group.
To overcome these issues, we apply the general toric stack construction of \cref{sec:toric_fiber_bundle}, specifically \cref{cor:extend}, with $\Delta = \Sec(K)$, $\Delta' = \MovSec(K)$, $N = \sfN$ and $L=\braket{n[Y]}\subset\sfN$ (recall $\sfN = n \Pic(K) \subset \Pic(K)$, and see \cref{rem:L} regarding the choice of $L$).
Hence we obtain the extended mirror family $(\cX,\cE,\Theta)$ over the toric stack $\cTV(\Sec(K),\sfN)$.
For the purpose of this paper, i.e.\ the study of \cref{conj:main}, the stacky structure of $\cTV(\Sec(K),\sfN)$ does not play any roles.
So for simplicity, in the absence of ambiguities, we will not distinguish $\cTV(\Sec(K),\sfN)$ from its coarse moduli space $\TV(\Sec(K),\sfN)$.

For use in the proof of Theorem \ref{thm:stable}, we describe more explicitly the extension above.
For any fan $\Delta$ in a lattice $N$ and any cone $\gamma\in\Delta$, we denote by $S^\circ(\gamma,N)\subset\TV(\Delta,N)$ the open stratum corresponding to $\gamma$.
Denote $\Sec\coloneqq\Sec(K)$ for simplicity.
The open stratum of $\TV(\Sec,\sfN)$ corresponding to the bogus cone $\hgamma=\gamma+\bbR_{\ge 0}[Y]$ above is $S^\circ(\hgamma,\sfN)$.
Pick a rational hyperplane $C_\bbR \subset \Pic(K)_\bbR$ containing $\gamma$, complementary to $\bbR [Y]$, and denote $C \coloneqq C_\bbR \cap \sfN$.
Then
$$
g\colon S^\circ(\gamma, C)\longto S^\circ(\gamma,\sfN) \longto \TV(\MovSec,\sfN)
$$
is an immersion, and $S^\circ(\gamma, C) \to S^\circ(\hgamma,\sfN)$ is a finite surjection.
The composition $S^\circ(\gamma,C) \to S^\circ(\hgamma,\sfN) \subset \TV(\Sec,\sfN)$ factors through $\cTV(\Sec,\sfN) \to \TV(\Sec,\sfN)$ giving quasi-finite $h\colon S^\circ(\gamma,C)\allowbreak \to \cTV(\Sec,\sfN)$, with image the stacky open stratum for $\hgamma$.

By the construction above, the $h$ pullback of the extension $(\cX,\cE,\Theta)$ to $\cTV(\Sec,\sfN)$ is equal to the $g$ pullback of
$$
(\cX,\cE,\Theta')|_{S^\circ(\gamma \in \Delta',\sfN)}
$$
where $\Theta' \coloneqq Z(\sum_{P\in \Lambda(\bbZ)\setminus[Y]} \theta_P)$.

For example: when $\gamma=0\in\sfN$, i.e.\ the corresponding divisorial contraction is just $K \to \oK$, we have $S^\circ(\gamma,\sfN) \simeq T_\sfN$, $S^\circ(\gamma,C) \simeq T_C \subset T_\sfN$, and the finite surjection $S^\circ(\gamma,C)\to S^\circ(\hgamma,\sfN)$ is an isomorphism.
The extended family over the corresponding new open stratum $S^\circ(\hgamma,\sfN)$ is equal to $(\cX,\cE,\Theta')|_{T_C}$, with $\Theta'$ obtained from $\Theta$ by setting the $\theta_{[Y]}$-coefficient to zero.

\subsection{An inductive structure on the mirror family} \label{sec:induction}

In this subsection, we assume $Y$ is a smooth del Pezzo surface.
In the proof of \cref{thm:stable}, we will study the mirror family restricted to different strata of $\TV(\Sec(K))$.
We will show in \cref{prop:restriction_SQM} that over certain strata, the restriction is isomorphic to the mirror family for some blowdown $Y'$ of $Y$ plus some simple constant pieces.
This gives an inductive structure on the mirror family, i.e.\ we will be able to reason by induction on the Picard number.

We follow the notations of \cref{lem:SQM_of_K}.

\begin{lemma} \label{lem:contract_Y1}
	There is a regular divisorial contraction $c\colon K' \to \oK'$ (over $\oK$), with exceptional locus $Y' \subset K'$, contracting $Y'$ to a point.
	Under the identification of \cref{lem:Mori_fan}(4), the cone $\Nef(\oK')$ in $\MovFan(K)$ corresponds to the face $\braket{\ex(b_Y)}$ of $\Eff(Y)$.
\end{lemma}
\begin{proof}
	Consider the cone $p^*\braket{\ex(b_Y)}$ in $\MovFan(K)$ under the identification of \cref{lem:Mori_fan}(4), and the corresponding contraction $K\dasharrow\oK'$, which is necessarily birational with exceptional locus contained in $Y \subset K$, as we are working relatively over $\oK$.
	Since $\braket{\ex(b_Y)}$ is a face of the bogus cone corresponding to $c_Y\colon Y\to Y'$, its pullback $p^*\braket{\ex(b_Y)}$ is a face of $b^*(\Nef(K'))$;
	so it induces a regular birational contraction $c\colon K' \to \oK'$, given by the linear system $\abs{mb(p^*E)}$ for $m\gg 0$, where $E\coloneqq\ex(b_Y)$ and $b(p^*E)$ denotes the strict transform.
	By the description of $b$ in \cref{lem:SQM_of_K}, we can compute $b(p^*E)$ explicitly, which is disjoint from $Y'$, so $c$ contracts $Y'$ to a point.
	Furthermore, the flopped curves in $K'$ are the only proper curves except those contained in $Y'$, and each meets $b(p^*E)$ at exactly one point, so they are not contracted by $c$.
	It follows that $Y'$ is the exception locus of $c$.
\end{proof}

\begin{lemma} \label{lem:S_gamma}
	Let $\gamma\coloneqq \Nef(\oK')\in \MovFan(K)$.
	It is also a cone of $\MovSec(K)$.
	The corresponding closed stratum $S(\gamma)$ in $\TV(\MoriFan(K))$ (resp.\ in $\TV(\Sec(K))$) is canonically identified with $\TV(\MoriFan(K_{Y'}))$ (resp.\ $\TV(\Sec(K_{Y'})$). 
\end{lemma}
\begin{proof}
	Since the cone $\braket{\ex(b_Y)}$ is a face of $\Eff(Y)$,
		by \cref{lem:contract_Y1}, $\gamma$ is a face of $\Mov(K)$, hence it is also a cone of $\MovSec(K)$.
	The second statement follows from the explicit description in \cref{prop:secondary_fan_del_Pezzo}.
	\end{proof}

\begin{lemma} \label{lem:cone_induct}
	Let $\alpha$ be a cone of $\MoriFan(K)$ that does not contain any $\gamma$ as in \cref{lem:S_gamma} (for any $SQM$ $K \dasharrow K'$ as in \cref{lem:SQM_of_K}).
	Then $\alpha$ is either a cone in $\Nef(K)$ or a bogus cone adjacent to $\Nef(K)$.
\end{lemma} 
\begin{proof}
	Under the isomorphisms of cones in \cref{lem:Mori_fan}(3), a cone $\gamma$ as in Lemma \ref{lem:S_gamma} is either a face of $\Eff(Y)$ spanned by the exceptional divisors of a divisorial contraction, or a face of $\Nef(Y)$ corresponding to fibrations.
		So the result follows.
\end{proof}

Let $\Gamma$, $\uGamma$, $\Lambda$ and $\uLambda$ be as in \cref{rem:Gamma}.
The SQM $b\colon K\dasharrow K'$ over $\oK$ induces a new triangulation $\uLambda'$ of $\Lambda$ by flopping the edges of $\uLambda$ corresponding to the boundary components of $\ex(b)$.
By \cref{lem:SQM_of_K}(\ref{lem:SQM_of_K:J}), the dual cone complex $\Sigma_{(V\subset K_{Y'})}$ is a sub-complex of $\Sigma_{(V\subset K')}$.
Let $\Gamma_{Y'}$, $\uGamma_{Y'}$, $\Lambda_{Y'}$ and $\uLambda_{Y'}$ denote the counterparts of $\Gamma$, $\uGamma$, $\Lambda$ and $\uLambda$ for $Y'$.
Then $\uLambda_{Y'}$ is a sub-complex of $\uLambda'$.

Let $\uLambda_\gamma$ be the coarsening of $\uLambda'$ where we remove all the edges internal to $\Lambda_{Y'}$, thus $\uLambda_\gamma$ consists of the polytope $\Lambda_{Y'}$ together with one triangle for each boundary component of $\ex(b)$.
Let $\Sigma_\gamma\coloneqq C(\uLambda_\gamma)$, the cone complex over $\uLambda_\gamma$, which has support $\Gamma$, and is a coarsening of $\Sigma_{(V\subset K')}$.

\begin{proposition} \label{prop:restriction_SQM}
	Consider the restriction
	$$
	(\cX,\cE)|_{S(\gamma)} \to S(\gamma)
	$$
	together with the canonical theta functions.
	The following hold:
	\begin{enumerate}
	\item The product $\theta_P\cdot\theta_Q = 0$ unless $P, Q$ lie in a same maximal cone $\tau$ of $\Sigma_\gamma$.
	If $\tau$ is not the maximal cone $\Gamma_{Y'}\subset\Sigma_\gamma$, then $\theta_P\cdot\theta_Q=\theta_{P+Q}$.
	\item Let $\cX' \subset \cX$ be defined by the vanishing of $\theta_P$ for all $P\in\Sigma_\gamma\setminus\Gamma_{Y'}$.
	Let $\cE' \subset \cX'$ be defined by the vanishing of $\theta_P$ for all $P\in\Gamma^\circ_{Y'}\subset\Sigma_\gamma$.
	Then
	$$
	(\cX', \cE') \to S(\gamma)
	$$
	together with the restriction of the theta functions $\theta_P$, $P\in\Gamma_{Y'}$, is canonically identified with the mirror family $(\cX,\cE)_{(Y',D')}$ for the pair $(Y',D'\coloneqq Y'\setminus U)$ over $\TV(\MoriFan(K_{Y'}))$ (resp.\ $\TV(\Sec(K_{Y'}))$ as in \cref{lem:S_gamma}) with its theta functions.
	\end{enumerate}
\end{proposition}
\begin{proof}
	It suffices to study the restriction to the open stratum $S^\circ(\gamma)\subset S(\gamma)$,
	the extension to the closed stratum $S(\gamma)$ then follows by tracing through the extensions of Sections \ref{sec:extension_moving} and \ref{sec:extension_full}.
	Indeed, the extension over the moving cone only concerns curve classes, which are read off from $\varphi$, while the extension over the bogus cones is determined by the boundary torus action.
	
	Let $I_S \subset R_{K'}=\bbZ[\NE(K')]$ be the ideal generated by (monomials associated to) curves \emph{not} contained in $Y'\subset K'$, i.e.\ not contracted by $K' \to \oK'$.
	Let $f$ be a structure disk in a general position contributing to the multiplication rule of $A_{K'}$ modulo $I_S$.
	By \cref{lem:structure_disk_in_K}, we can assume that the cycle associated to $f$ is supported on $K'$, then it is just supported on $Y'$, so the spine $\Sp(f)$ (minus the marked points) lies in a same maximal cone of $\Sigma_\gamma$ (see \cite[Lemma 7.2]{Keel_Yu_The_Frobenius}).
	This implies that $\theta_P\cdot\theta_Q=0$ modulo $I_S$ unless $P,Q$ lie in a same maximal cone $\tau$ of $\Sigma_{\gamma}$.
	If $\tau$ is not $\Gamma_{Y'}\subset\Sigma_\gamma$, then the fact that the cycle associated to $f$ is supported on $Y'$ implies moreover that the whole tropical curve $\Trop(f)$ (minus the marked points) lies in the interior of $\tau$.
	So the structure disk lies in a toric locus, and we deduce from \cite[Lemma 9.4]{Keel_Yu_The_Frobenius} that $\theta_P\cdot\theta_Q=\theta_{P+Q}$ modulo $I_S$.
	This shows the first statement.
	
	It follows that the $R_{K'}$-submodule of $A_{K'}$ with basis $\theta_P$, $P\in\Gamma_{Y'}(\bbZ)$ gives a subalgebra modulo $I_S$.
	So we see that the coordinate rings for $\cX'|_{S^\circ(\gamma)}\to S^\circ(\gamma)$ and $\cX_{(Y',D')}|_{T_{\Pic(Y')}}\to T_{\Pic(Y')}\subset\TV(\Nef(Y'))$ have the same theta function basis, thus it remains to check that the multiplication rules are the same.
	We show that the structure disks responsible for the structure constants in the two cases are exactly the same.
	Let $J\subset K'$ be as in \cref{lem:SQM_of_K}(\ref{lem:SQM_of_K:J}), we have $K_{Y'}\simeq K'\setminus J$.
	Note $J$ is disjoint from $Y'$, and is a disjoint union of irreducible components, one for each exceptional curve of $Y \to Y'$.
	Write $J= J_\bdry \cup J_\interior$, according to whether the associated exceptional curve is boundary or interior (see \cref{prop:secondary_fan_del_Pezzo}).
	Note $J_\bdry$ is a union of components of $V \subset K'$, while $J_\interior$ contains no boundary strata.
	Let $f$ be the structure disk as above, with $\Sp(f)$ lying in the cone $\Gamma_{Y'}\subset\Sigma_\gamma$.
	Since $J$ is disjoint from $Y'$, $J_\interior$ has zero intersection with the class of $f$.
	It follows by \cite[Proposition 7.5]{Keel_Yu_The_Frobenius} that the image of $f$ is disjoint from $J_\interior$, so the structure disk (minus the marked points) lies in $K_{Y'}$, and is thus a structure disk for the mirror algebra $A_{K_{Y'}}$ of $K_{Y'}$.
	This identifies the structure disks in the two cases, completing the proof.
\end{proof}

\section{Singularities of the mirror family} \label{sec:singularities}

In this section, all the varieties are assumed to be over $\bbC$.
We will prove a general proposition concerning semi-log-canonical singularities in dimension two, see \cref{prop:slc_surface}, and deduce that in the 2-dimensional case, every fiber of the mirror family over the structure torus is a del Pezzo surface with at worst du Val singularities, see \cref{prop:du_Val}.

Let $A_K$ be the graded mirror algebra over $R_K$ as in \cref{prop:relation_AKAY}.
We have $\Spec(R_K\otimes\bbC)\simeq\TV(\Nef(Y))$, the toric variety (over $\bbC$) associated to the nef cone of $Y$.
Let $\cX\coloneqq\Proj(A_K\otimes\bbC)$, $\cE\subset\cX$ the closed subscheme given by the ideal $I\subset A_K\otimes\bbC$ which has basis (as submodule) the integer points $\Gamma^\circ(\bbZ)$ in the interior of $\Gamma$, and $\Theta\subset\cX$ the divisor as in \cref{def:Theta}.

\begin{proposition} \label{prop:mirror_is_Fano}
	Let $(X,E,\Theta)$ be a closed fiber of the mirror family $(\cX,\cE,\Theta) \to \TV(\Nef(Y))$.
	The following hold:
	\begin{enumerate}
	\item $K_X + E$ is semi-log-canonical and trivial (in particular Cartier).
	\item $E \subset X$ is reduced and ample, and it is the zero locus of the theta function $\theta_{[Y]}$, where $[Y]\in\Lambda$ is the unique interior lattice point.
		\item $X$ is Gorenstein, $K_X\simeq\cO(-1)$ is anti-ample.
		\item The self-intersection number $(-K_X)^{\dim X}$ is equal to the number of 0-strata of $D \subset Y$.
	\item For the generic fiber, $X$ is normal, $(X,E)$ is log-canonical, and $(X,E + \epsilon \Theta)$ is stable for sufficiently small $\epsilon > 0$.
\end{enumerate}
\end{proposition}
\begin{proof}
	Statements (1) and (5) follow from \cref{prop:relation_AKAY} and \cite[Proposition 19.1]{Keel_Yu_The_Frobenius}.
		For (2), under the isomorphism of \cref{prop:relation_AKAY}, $\theta_{[Y]}$ corresponds to $T$, and the ideal $I\subset A_K\otimes\bbC$ corresponds to the ideal generated by $T$.
	Thus $I$ is generated by $\theta_{[Y]}$, which is a section of $\cO(1)$, so (2) holds.
	Then (1) and (2) implies (3).
	By (1) and (3), the self-intersection number $(-K_X)^{\dim X}$ is the degree of $\cO(1)|_E$, which by \cref{prop:boundary_E} is the number of maximal cones in $\partial\uLambda$, which is the number of 0-strata of $D$, hence (4) holds.
\end{proof} 

We have a stronger result when $Y$ is two-dimensional.
We begin with the following general proposition concerning semi-log-canonical singularities in dimension two.

\begin{proposition} \label{prop:slc_surface}
	Let $X$ be a connected surface and $E\subset X$ an effective Cartier divisor which is a non-empty cycle of rational curves, such that $K_X+E$ is trivial and slc.
	Let $\tX \to X$ be the normalization, $\NN \subset X$ the non-normal locus, and $\tNN, \tE \subset \tX$ the reduced inverse images of $\NN,E$.
	Then $X$ is canonical off $\NN$, and the inverse image on $\tX$ of the set of minimal log-canonical centers for $(X,E)$ is exactly the singularities of $\tNN \cup \tE$.
\end{proposition}

For the proof we make use of the classification of log-canonical singularities in dimension two, taken from \cite[Theorem 4.15]{Kollar_Birational_geometry_of_algebraic_varieties} and \cite[Proposition 4.27]{Kollar_Threefolds_and_deformations}.

\begin{proposition} \label{prop:classfication_lc_surface}
	Let $(Y,D)$ be a 2-dimensional log-canonical pair, with $D$ reduced.
		Then $D$ is a curve with at worst ordinary nodal singularities, and locally	analytically at every point $p \in D \subset Y$ we have one of the following:
	\begin{enumerate}
		\item $(0 \in \bbA^2_{x,y}/\frac{1}{r}(1,a), (xy=0))$
		(see \cite[3.19]{Kollar_Singularities_of_the_minimal_model_program} for the notation for cyclic quotient singularity),
		\item $(0 \in \bbA^2_{x,y}/\frac{1}{r}(1,a), (x=0))$.
		\item  Quotient of (1) by a $\bbZ/2\bbZ$-action which is free on $Y\setminus p$ and interchanges the components of $D$.
	\end{enumerate}
	The divisor $K_Y + D$ is Cartier only in case (1) and in the smooth case $r=1$ of (2).
	Case (2) is purely log-terminal.
	Cases (1) and (3) are strictly log-canonical (i.e.\ $p$ is a log-canonical center).
\end{proposition}

\begin{proof}[Proof of \cref{prop:slc_surface}]
	Let $S \subset \tX$ be an irreducible component, $B\coloneqq\tNN|_S$ and $F\coloneqq\tE|_S$.
	Since $(X,E)$ is slc, $(S,B + F)$ is log-canonical.
	In particular, $B+F$ is a nodal curve by \cref{prop:classfication_lc_surface}.
	Since $K_S+B+F=0$, by adjunction each connected component of $B+F$ is a either a smooth elliptic curve or a cycle of rational curves.
	By \cref{prop:classfication_lc_surface}, the $0$-dimensional log-canonical centers of $(S,B+F)$ on $B+F$ are the nodes of $B+F$.
	
	We let $d\colon(Y,D) \to (S,B+F)$ be a dlt model,
		and use \cite[Prop.\ 5.1]{Kollar_Log_canonical_singularities}.
	We have $K_Y + D = d^*(K_S + B + F)$.
	Since $(S,B+F)$ is log-canonical, we have $Z_Y = d^{-1}(Z_S)$ where $Z$ in each case indicates the non klt locus of the pair.
	By \cite[Prop.\ 5.1]{Kollar_Log_canonical_singularities}, either $Z_Y$ and $Z_S$ are both connected (connectedness of one is equivalent to connectedness of the other since $d$ has connected fibers)
	or we are in case (2) of loc.\ cit..
	
	Let us first suppose that $Z_S$ is connected.
	The surface $X$ is connected by assumption, and has normal crossing singularities in codimension 1 and satisfies Serre's condition $S_2$ by the definition of slc singularities
	(see e.g.\ \cite[Definition-Lemma 5.10]{Kollar_Singularities_of_the_minimal_model_program}).
	The $S_2$ condition implies that the intersection of $S$ with the union of the other irreducible components of $X$ has pure codimension $1$ (see \cite[Theorem 5.10.7]{EGA4-2}).
	So if $X$ is reducible then $F \neq 0$ (and if $X$ is irreducible then $B \neq 0$ by assumption).
	So $B + F \neq 0$.
	Now since $Z_S$ contains $B + F$ and $S$ has isolated singularities, $Z_S=B + F$ (and so $B+F$ is connected) and $S$ is klt away from $B+F$.
	Also, $S$ is Gorenstein away from $B$, because $X$ is Gorenstein, so $S$ is canonical away from $B+F$.
	Now by adjunction $B + F$ is either a smooth (connected) elliptic curve, or a cycle of smooth rational curves.
	In the smooth elliptic curve case, if $B$ is not empty (equivalently, if $X$ is not normal), $F$ is empty and thus $S$ is disjoint from the preimage of $E$.
	
	Next we consider case (2) of \cite[Prop.\ 5.1]{Kollar_Log_canonical_singularities}: $(Y,D)$ is purely log-terminal, and $D$ has two connected components.
	Then by \cref{prop:classfication_lc_surface}, $D$ is smooth, $Y$ is smooth near $D$, $Z_Y = D$, and by adjunction $D$ is a disjoint union of two elliptic curves.
	Set-theoretically $D = d^{-1}(B + F)$.
	It follows that $B+F$ is a disjoint union of two irreducible curves each with elliptic normalization.
	But any irreducible component of $F$ is rational, so $F=0$ and $S$ is disjoint from the preimage of $E$.
	
	But now note that if we have two {\it adjacent} $S_1,S_2$ (i.e.\ the intersection of their images in $X$ contains a curve), then one satisfies case (2) of \cite[Prop.\ 5.1]{Kollar_Log_canonical_singularities} if and only if the other does (because they then share an elliptic double curve).
	Now by the connectedness of $X$, if one $S$ satisfies case (2) of \cite[Prop.\ 5.1]{Kollar_Log_canonical_singularities}, all components of $\tX$ do.
	But then by the above, $E$ is empty, a contradiction.
	
	We conclude that we are in the $B + F$ connected case above, with $B + F$ a cycle of
	rational curves, for each irreducible component $S \subset \tX$.
	By \cref{prop:classfication_lc_surface}, the minimal log-canonical centers of $(S,B+F)$ are exactly the nodes of $B + F$.
	But the inverse image on $\tX$ of the log-canonical centers of $(X,E)$ is just the union of the centers of $(S,B + F)$, thus the inverse image of the minimal log-canonical centers of $(X,E)$ is exactly the nodes of $\tNN\cup\tE$.
	We argued above that $S$ is canonical away from $B + F$, thus $X$ is canonical away from $\NN \cup E$.
	For an exceptional divisor $W$ with center in $E \setminus \NN$, the discrepancies satisfy $a(W,X) \geq a(W,X,E) + 1$ (notation as in \cite[Definition 2.4]{Kollar_Singularities_of_the_minimal_model_program}), since $E$ is Cartier.
	Therefore, since $K_X + E$ is log-canonical off $\NN$, $X$ is canonical along $E \setminus \NN$, and thus canonical off $\NN$.
	This completes the proof.
\end{proof}

\begin{proposition} \label{prop:du_Val}
	In the context of \cref{prop:mirror_is_Fano}, assume $Y$ is 2-dimensional, i.e.\ a del Pezzo surface.
	Then for every fiber $(X,E)$ of the mirror family over the structure torus $T_{\Pic(Y)}\subset\TV(\Nef(Y))$, $X$ is a del Pezzo surface with at worst du Val singularities, $E\subset X$ is an anti-canonical cycle of $K_X$-degree-$(-1)$ rational curves, and the self-intersection number of $-K_X$ is equal to the number of irreducible components of $D$.
\end{proposition}
\begin{proof}
	By \cref{prop:mirror_is_Fano}, the self-intersection number of $-K_X$ is equal to the number of irreducible components of $D$.
	Then it follows from \cref{prop:boundary_E} that $E\subset X$ is an anti-canonical cycle of $K_X$-degree-$(-1)$ rational curves.
	It remains to show that $X$ is du Val (which in dimension two is the same as canonical). 
	Consider the equivariant boundary torus $T_D$-action on the mirror family $\cX\to\TV(\Nef(Y))$ (see \cite[\S 16]{Keel_Yu_The_Frobenius}).
	Since $D$ is ample, by \cite[\S 2.3]{Fulton_Introduction_to_toric_varieties}, there is a one-parameter subgroup $\Gm\subset T_D$ that pushes any point of $\TV(\Nef(Y))$ to the unique toric 0-stratum $0 \in \TV(\Nef(Y))$.
	Given $t \in T_{\Pic(Y)}$, taking closure of the orbit $\Gm\cdot t$, we obtain $\Gm\subset\bbA^1\to\TV(\Nef(Y))$.
	Consider the pullback family $\pi\colon(\cX,\cE)|_{\bbA^1}\to\bbA^1$.
	By the equivariant torus action, $\pi$ is a trivial family over $\Gm$.
	By \cref{prop:slc_surface}, in order to prove that $\cX_t$ is canonical, it suffices to show that its singular locus has codimension at least two.
	Suppose to the contrary that its singular locus has codimension one, and consider the closure of the singular locus of $\pi|_{\Gm}$.
	Then its intersection with the central fiber has codimension at most 1.
	The central fiber (of $\cX \setminus \cE$) is isomorphic to the cone over an $n$-cycle of rational curves (see \cref{prop:umbrella}), which looks like an umbrella, with singular locus the ribs of the umbrella.
	Thus by generic smoothness, the total space $\cX|_{\bbA^1}$ is non-normal at the generic point of some rib.
	Using \cref{prop:alg_equal}, this contradicts the local equations for the mirror family along (the interior of) a rib, see \cite[Eq.\ (2.7)]{Gross_Mirror_symmetry_for_log_Calabi-Yau_surfaces_I_published}:
	In fact, if we base change to the completion of $R_Y\otimes\bbC$ at the maximal monomial ideal $\fm$, then because $f_{\rho_i} =1$ mod $\fm$ (notation from  \cite[Eq.\ (2.7)]{Gross_Mirror_symmetry_for_log_Calabi-Yau_surfaces_I_published}), a neighborhood in $\cX$ of (the interior of) the rib is isomorphic to the toric Mumford family, with explicit equations \cite[Eq.\ (0.5)]{Gross_Mirror_symmetry_for_log_Calabi-Yau_surfaces_I_published}, which is clearly normal.
\end{proof}

\section{Modularity of the mirror family} \label{sec:modularity}

Let $Y$ be a smooth complex del Pezzo surface, and $D\subset Y$ a normal crossing anti-canonical divisor containing a 0-stratum.
Let $(\cX,\cE,\Theta)$ be the mirror family over the toric variety $\TV(\Sec)$ associated to the full secondary fan which we constructed in \cref{sec:extension_full}.

\begin{notation} \label{nota:base_change}
	In this section, we will always consider mirror families after base change to $\bbC$, so we will omit the base change from the notation, in particular, $\TV(\cdot)$ will now denote toric varieties over $\bbC$ (contrary to \cref{nota:TV_Z}).
\end{notation}

The goal of this section is to prove the following:

\begin{theorem} \label{thm:stable}
	The following hold:
	\begin{enumerate}
	\item \label{thm:stable:family} The mirror family $(\cX,\cE) \to \TV(\Sec)$ is a flat family of semi-log-canonical pairs $(X,E)$ with $K_X+E$ trivial (in particular Cartier), and $H^i(X,\cO) =0$ for $i > 0$.
	\item \label{thm:stable:E} The boundary $\cE \to \TV(\Sec)$ is a trivial family, with fiber a cycle of rational curves.
	\item \label{thm:stable:restriction}
	For every fiber $(X,E)$ over the structure torus $T_{\Pic(Y)}\subset\TV(\Sec)$, $X$ is a del Pezzo surface with at worst du Val singularities, $E\subset X$ is an anti-canonical cycle of $K_X$-degree-$(-1)$ rational curves, and the self-intersection number of $K_X$ is equal to the number of irreducible components of $D$.
		\item \label{thm:stable:stable} For $0<\epsilon\ll 1$, $(\cX,\cE + \epsilon \Theta)\to\TV(\Sec)$ is a family of stable pairs.
	\item \label{thm:stable:finite} The induced map $\TV(\Sec)\to\SP$ to the moduli space of stable pairs is finite.
\end{enumerate}
\end{theorem}

From \cref{thm:stable} we deduce the next theorem, which is a detailed version of \cref{thm:del_Pezzo_intro}.

\begin{theorem} \label{thm:del_Pezzo}
	Let $Y$ be a smooth complex del Pezzo surface and $D\subset Y$ an anti-canonical cycle of $(-1)$-curves.
	Then the generic fiber of the mirror family over $\TV(\Sec)$ is smooth, and one fiber is the original $(Y,D)$.
	The image of the finite map $\TV(\Sec) \to \SP$ in \cref{thm:stable}(\ref{thm:stable:finite}) is equal to the closure $\oQ\subset \SP$, where $Q$ is the moduli space in \cref{conj:smooth} for the pair $(Y,D)$.
\end{theorem}

\begin{proof}
	Suppose $D$ has $n$ irreducible components, then $n$ is also the degree of $Y$, and thus $\Pic(Y)$ has rank $10-n$.
	Therefore by the dimension count in \cref{prop:dimcount}, the image of the finite map $\TV(\Sec)\to\SP$ is equal to the closure $\oQ\subset\SP$.
	In particular, the generic fiber of the mirror family must be smooth, and one fiber is the original $(Y,D)$.
\end{proof}

\begin{proposition} \label{prop:dimcount}
	For $n=1,\dots,9$, let $Q_n$ (resp.\ $Q_n^\dV$) be the moduli space of triples $(X,E,\Theta)$ where $X$ is a smooth (resp.\ du Val) del Pezzo surface of degree\footnote{By \emph{degree} we mean the self-intersection number of $K_X$.} $n$, $(X,E)$ is a log-canonical pair, $E\in\abs{-K_X}$ is a cycle of $K_X$-degree-$(-1)$ curves, and $\Theta\in\abs{-K_X}$ does not pass through any singular point of $E$.
	Then $Q_n$ is empty unless $n\le 6$;
	in this case it is irreducible of dimension at most $10-n$.
	On the other hand, $Q_n^\dV$ has dimension at most $9-n$, (and is possibly reducible).
\end{proposition}
\begin{proof}
	Let us first consider the smooth case $Q_n$:
	To see that $Q_n$ is empty for $n\ge 7$, suppose to the contrary that $D = D_1 + \dots D_n$ with $n \geq 7$.
	Let $a \coloneqq D_1 + D_2 + D_3$ and $b \coloneqq D_5 + D_6$.
	Then $a^2 >0$, $b^2 = 0$ and $a \cdot b =0$, which contradicts the Hodge index theorem.
	For $ n \leq 5$ the result follows from \cite[Theorem 1.1]{Looijenga_Rational_surfaces}, which gives a concrete description of all pairs $(X,E)$.
	If $n =6$, then $(X,E)$ is toric and unique (see \cite[Lemma 1.3]{Gross_Mirror_symmetry_for_log_Calabi-Yau_surfaces_I_v1}) and the result is obvious.
	Next we estimate the dimension of $Q_n^\dV$.
	Let $p\colon\tX \to X$ be a minimal resolution, which is crepant by the du Val assumption.
	Note that $-K_{\tX} = p^*(-K_X)$ is nef and big, but not ample.
	So $\tX$ is a blowup of $\bbP^2$ at $9-n$ (possibly infinitely near) points, which \emph{are not} in general position (otherwise $-K_{\tX}$ is ample), i.e.\ they satisfy the divisorial conditions as in the first paragraph of \cite[Appendix A]{Naruki_Cross_ratio_variety}.
	We can fix $4$ of the points by automorphisms of $\bbP^2$, thus the number of moduli for the points on $\bbP^2$ is at most $2 \cdot (5 - n) -1$.
	The linear system $\abs{-K_{\tX}} = \abs{-K_X}$ has dimension $n$.
	Each irreducible component of the strict transform $\tE$ of $E$ is one of the finitely many $(-1)$-curves on $\tX$, so this adds no moduli.
	Therefore, $Q_n^\dV$ has dimension at most $9-n$, completing the proof.
\end{proof}

\begin{remark} \label{rem:duval}
	Given $(Y,D= D_1 + \dots D_n)$ as in the context of \cref{thm:stable}, let $d_i \coloneqq -K_X \cdot D_i$.
	Modify slightly the above definition of $Q_n^{\dV}$ adding the conditions that $E= E_1 + \dots E_n$ (where we cyclically order the components $E_i$ and $D_i$) and that $X$ has an $A_{d_i -1}$ singularity at the node $E_i \cap E_{i+1}$.
	Then by a more involved dimension count and some finer analysis of the mirror family, we can prove that the image of the finite map $\TV(\Sec)\to\SP$ of \cref{thm:stable}(\ref{thm:stable:finite}) is the closure of $Q_n^{\dV}$ in $\SP$.
\end{remark}

\subsection{Proof of the main theorem} \label{sec:proof_of_stability}

Now we turn to the proof of \cref{thm:stable}.

Since the extension of the mirror family over the bogus cones are done via the equivariant boundary torus action (see \cref{sec:extension_full}), it suffices to prove (\ref{thm:stable:family}) for the restriction to the moving part of the secondary fan.
That reduces to the study of singularities of the mirror family associated to every SQM $K\dasharrow K_\alpha$, which follows from \cite[Proposition 19.1]{Keel_Yu_The_Frobenius}.

Statements (\ref{thm:stable:E}) and (\ref{thm:stable:restriction}) follow from Propositions \ref{prop:boundary_E} and \ref{prop:du_Val} respectively.

Next we prove (\ref{thm:stable:stable}).
By (\ref{thm:stable:family}) it is enough to check that for every fiber $(X,E,\Theta)$, $\Theta$ is disjoint from any minimal (under inclusion) log-canonical centers of $(X,E)$, which are described by \cref{prop:slc_surface}, in particular they are all 0-dimensional.
Moreover by \cref{lem:boundary_E_dim2}, we only need to consider 0-dimensional log-canonical centers of $X \setminus E$.

Since the secondary fan is a coarsening of the Mori fan, each cone $\gamma_0\in\Sec(K)$ contains a (not necessarily unique) cone $\gamma\in\MoriFan(K)$ of the same dimension.
Let $S^\circ(\gamma)\subset S(\gamma)\subset\TV(\MoriFan(K))$ denote the associated open and closed strata.
Then it suffices to prove (\ref{thm:stable:stable}) for the restriction of the mirror family over all such strata $S^\circ(\gamma)$.
If $\gamma$ is of the form in \cref{lem:S_gamma}, let $\cX'\subset\cX$ be as in \cref{prop:restriction_SQM} and $X'\coloneqq X\cap\cX'$.
For log-canonical centers contained in $\overline{X\setminus X'}$, the statement follows from \cref{prop:restriction_SQM}(1) because the irreducible components of this closure are toric projective varieties, and the theta functions restrict to toric monomials.
Then by \cref{prop:restriction_SQM}(2), we can proceed by induction on the Picard number of $Y$.
The cones of $\MoriFan(K)$ that do not contain any $\gamma$ as in \cref{lem:S_gamma} are described in \cref{lem:cone_induct}.
Therefore, it remains to consider the restriction of the mirror family over every open stratum $S^\circ(\gamma)$ where $\gamma$ is $c^*\Nef(K')$ for a regular contraction $c\colon K\to K'$, or the associated bogus cone $c^*\Nef(K') + \bbZ_{\geq 0} [Y]$.
We refer to these two cases as the nef cone case and the bogus cone case.
We will enumerate the possibilities for $c\colon K\to K'$, which will also be used later in the proof of (\ref{thm:stable:finite}).
Let $f\colon Y \to Y'$ denote the contraction induced by $c\colon K\to K'$.

Consider a fiber $(X,E,\Theta)$ over the open stratum $S^\circ(\gamma)$ as above, and a 0-dimensional log-canonical center $z$ of $(X,E)$ in $X\setminus E$.
In order to show that $\Theta$ does not contain $z$, it is enough to establish the following:

\begin{claim} \label{cl:center_claim}
	Exactly one of the theta functions $\theta_P$, $P\in \Lambda(\bbZ)$ is non-vanishing at $z$.
\end{claim}

Let us prove \cref{cl:center_claim}.
Let $\cV\to\TV(\Nef(Y))\simeq\Spec R_Y$ denote the restriction of $\cX\setminus\cE$ to $\TV(\Nef(Y))$.
Recall from \cref{prop:relation_AKAY} that $\cV\simeq\Spec(A_Y)$.

\smallskip
\textbf{Case I}: the nef cone case where $Y'$ is a point.
Then the open stratum $S^\circ(\gamma)$ is the structure torus $T_{\Pic(Y)}\subset\TV(\Nef(Y))$.
By \cref{prop:du_Val}, there are no log-canonical center $z$ of $(X,E)$ in $X\setminus E$.
So this case does not occur.

\smallskip
\textbf{Case II}: the nef cone case where $f\colon Y \to Y'$ is birational, i.e.\ $c\colon K\to K'$ is small.
By the choice of $\gamma$, it is contained in a cone $\gamma_0\in\Sec(K)$ of the same dimension.
We claim that the exceptional locus of $f$ is a disjoint union of boundary $(-1)$-curves.
Indeed, factor $f$ through $f_1\colon Y\to Y_1$ which contracts exactly the boundary exceptional divisors of $f$, and let $\gamma_1\coloneqq f_1^*(\Nef(Y_1))$.
Then by \cref{rem:Sec_on_PicY}, the set of cones of $\Sec(K)$ containing $\gamma$ is the same as the set of those containing $\gamma_1$.
The intersection of those cones is $\gamma_0$.
Since $\gamma$ and $\gamma_0$ have the same dimension, we deduce that $\gamma=\gamma_1$ and $f=f_1$.

The ideal $J\subset \bbZ[\NE(Y,\bbZ)]$ associated to the closed stratum $S(\gamma)$ is the monomial ideal generated by all curves other than the $f\colon Y \to Y'$ exceptional curves.
Since the exceptional locus is contained in $D$, and any effective cycle supported on the exceptional locus is rigid in the sense of \cref{def:rigid},
it follows from \cref{cor:count_rigid} that the restriction of $\cV$ to $S(\gamma)$ is the purely toric Mumford partial smoothing of $V_{\Sigma_{(Y,D)}}$ to $V_{\Sigma_{(Y',D')}}$ (see \cite[\S 1.2]{Gross_Mirror_symmetry_for_log_Calabi-Yau_surfaces_I_v1}).
Then \cref{cl:center_claim} holds because $z$ must be the center of the umbrella, and $\theta_0$ is the only theta function that does not vanish there.

\smallskip
\textbf{Case III}: the nef cone case where $f$ is a ruling $Y \to \bbP^1$.
We will show that over $S^\circ(\gamma)$ there are no log-canonical center $z$ of $(X,E)$ in $X\setminus E$.
So this case does not occur.

Again, the ideal $J \subset \bbZ[\NE(Y,\bbZ)]$ associated to the closed stratum $S(\gamma)$ is generated by all curves not contained in a fiber of $f$.
Decompose $D = D_H + D_F$, where $D_F$ are the irreducible components contracted by $f$, and $D_H$ are the horizontal components.
Let $\Sigma_H$ be the coarsening of $\Sigma_{(Y,D)}$ retaining only the rays corresponding to the horizontal components.

We factor $f$ through $f_1\colon Y \to Y_1$ contracting $(-1)$-curves in fibers of $f$ which meet $D_H\setminus D_F$.
Note that the tropicalization of any structure disk (that contributes modulo $J$) cannot cross rays of $\Sigma_H$, otherwise the disk class will contain a component of $D_H$ and so be trivial modulo $J$.
It follows that such structure disks are disjoint from $f_1$-exceptional divisors, so the mirror family over $S^\circ(\gamma)$ is obtained by base extension from the analogous family for $Y_1$.
Therefore, replacing $Y$ by $Y_1$, we can assume there are no such $(-1)$-curves.

Now we refer to \cref{sec:mirror_family_P1-fibration} to finish the proof of case III.
If $D_H$ is irreducible, we can invoke \cref{prop:P1_fibration} to deduce the absence of 0-dimensional log-canonical centers.
If $D_H$ is reducible, note that $\theta_P\cdot\theta_Q=0 \mod J$ unless $P,Q$ lie in a same cell of $\Sigma_H$, and the family $\cV$ restricted to $S^\circ(\gamma)$ is the union of two irreducible components, each given by the vanishing of all $\theta_P$ with $P$ lies in a single cone of $\Sigma_H$.
Write $X\eqqcolon X_1\cup X_2$ and $F\coloneqq X_1\cap X_2$.
Then each $(X_i, F)$ is of the form $(V',C')$ in \cref{prop:P1_fibration}.
By \cref{prop:slc_surface}, any 0-dimension log-canonical center of $X$ must be a log-canonical center of $(X_i,F)$ for $i=1$ or $2$.
So we can again conclude by \cref{prop:P1_fibration}.

\medskip
Next we consider the bogus cone case.
Then $c\colon K \to K'$ is a divisorial contraction, so the corresponding $f\colon Y\to Y'$ is a fibration.
We distinguish \textbf{case IV} where $Y'$ is a point, and \textbf{case V} where $Y'\simeq\bbP^1$.
Recall from \cref{sec:extension_full} that the fibers $(X,E)$ over $S^\circ(\gamma)$ are all isomorphic to fibers over $S^\circ(c^*\Nef(K'))$.
But by the analysis we just made in cases I and III, there are no 0-dimensional log-canonical centers off $E$ in fibers over $S^\circ(c^*\Nef(K'))$, so cases IV et V also do not occur.
This completes the proof of (\ref{thm:stable:stable}).

\medskip
Now we turn to the proof of (\ref{thm:stable:finite}), the finiteness of the map $\TV(\Sec)\to\SP$.

\begin{lemma} \label{lem:finite_map_from_toric_variety}
	Let $X$ be a complete toric variety.
	A map $f\colon X \to Y$ to any other variety is finite if and only if no toric 1-stratum of $X$ is contracted to a point.
\end{lemma}
\begin{proof}
	The ``only if'' direction is obvious.
	For the ``if'' direction, assume there is an irreducible curve $C \subset X$ contracted by $f$, and we want to find a contracted toric 1-stratum.
	Choose a toric resolution $b\colon\tX \to X$ with $\tX$ projective,
		and $\tC \subset \tX$ an irreducible curve mapping onto $C$.
	Consider the closure
	\[
	P\coloneqq\overline{T \cdot \{\tC\}} \subset \Hilb(\tX)
	\]
	in the Hilbert scheme of curves in $\tX$, where $T \subset \tX$ is the structure torus.
	This closure is proper so contains a torus fixed point $p$ (e.g.\ its normalization is a projective toric variety on which $T$ has a dense open orbit,
		just take the image of a torus fixed point on the normalization).
	Then the corresponding curve is supported on a 1-stratum.
	Choose an irreducible curve $S\subset P$ connecting $\{\tC\}$ and $p$, and let $\cC \to S$ be the corresponding family.
	Consider $g\coloneqq f \circ b\colon \cC \to Y$.
	By assumption the fiber $\tC \subset \cC$ is contracted by $g$, so by the rigidity lemma (\cite[Lemma 1.6]{Kollar_Birational_geometry_of_algebraic_varieties}), every fiber is contracted, in particular, $f(b(\cC_p))$ is a point.
	Again by the rigidity lemma, $b\colon\cC\to X$ does not contract any fibers, in particular $b(\cC_p)$ has support a toric 1-stratum, completing the proof.
\end{proof}

By \cref{lem:finite_map_from_toric_variety}, for the proof of \cref{thm:stable}(\ref{thm:stable:finite}), it suffices to show that no toric 1-stratum of $\TV(\Sec)$ is contracted (to a point).
We will show more:

\begin{claim} \label{cl:restriction_to_1-stratum}
	The restriction of the mirror family $(\cX,\cE)$ to every toric 1-stratum $S$ of $\TV(\Sec)$ is non-trivial.
	More precisely, we will exhibit a 0-stratum $s\in\partial S$ and a double curve in the fiber $\cV_s$ which disappears over $S^\circ$, except in case III below it will be a 0-dimensional log-canonical center which disappears.
\end{claim}

So we no longer care about the boundary $\cE$ and the divisor $\Theta$.
As in the proof of (\ref{thm:stable:stable}), using \cref{prop:restriction_SQM} and \cref{lem:cone_induct}, by induction on the Picard number of $Y$, it is enough to consider toric 1-strata of the form $S(\gamma)$ along the five cases as before.

\smallskip
\textbf{Case I}:
the nef cone case where $Y'$ is a point.
Then the open stratum $S^\circ(\gamma)$ is the structure torus (and since we are assuming this is a 1-stratum, $Y$ has Picard number one).
The generic fiber over $S^\circ(\gamma)$ is normal by \cref{prop:mirror_is_Fano}, while the fiber over the 0-stratum associated to $\Nef(K)$ is the non-normal umbrella by \cref{prop:umbrella}.
So \cref{cl:restriction_to_1-stratum} holds in this case.

\smallskip
\textbf{Case II}: the nef cone case where $f\colon Y\to Y'$ is birational.
Since $S(\gamma)$ is 1-dimensional, $\gamma\subset\Pic(K)_\bbR$ has codimension 1, so $f$ is the contraction of a single boundary $(-1)$-curve $E\subset D\subset Y$.
The fiber over the 0-stratum associated to $\Nef(K)$ is the umbrella, and the double curve corresponding to the ray of $\Sigma_{(Y,D)}$ given by $E$ smooths along $S(\gamma)$ by the toric Mumford smoothing.
So \cref{cl:restriction_to_1-stratum} holds in this case.

\smallskip
\textbf{Case III}: the nef cone case where $f$ is a ruling $Y\to\bbP^1$.
Since $\gamma\subset \Pic(K)_\bbR$ has codimension one, $f\colon Y\to\bbP^1$ is a smooth $\bbP^1$-bundle,
and $\gamma\subset \Nef(K)$
gives a 0-stratum in $S(\gamma)$, with fiber the umbrella $V_{\Sigma_{(Y,D)}}$, which has a 0-dimensional log-canonical center (i.e.\ $0 \in V_{\Sigma_{(Y,D)}})$.
By \cref{prop:P1_fibration}, the fibers over $S^{\circ}(\gamma)$ no longer have 0-dimensional log-canonical centers, completing the proof in this case.

\smallskip
\textbf{Case IV}: the bogus cone case where $Y'$ is a point.
We have $\gamma = 0 + \bbZ_{\geq 0} [Y]$.
Since $\gamma$ has codimension one, $Y$ has Picard number two.
There are two possibilities for the del Pezzo, in either case there is a ruling $r\colon Y \to \bbP^1$.
By \cref{sec:extension_full}, the fibers of $\cX$ over $S^\circ(\gamma)$ occur as fibers over $S^\circ(0) = T_{\Pic(K)}$, which are normal by \cref{prop:du_Val}.
But the ruling determines a 0-stratum in $\partial S(\gamma)$, given by the cone $\gamma_r\coloneqq c_r^*\Nef(K_r) + \bbZ_{\geq 0}[Y]$, where $c_r\colon K \to K_r$ is the divisorial contraction with exceptional locus $r\colon Y \to \bbP^1$.
Again by \cref{sec:extension_full}, the fibers of $\cX$ over $S^\circ(\gamma_r)$ occurs as fibers over $S^\circ(c^*\Nef(K_r))$.
Recall from the analysis of case III in the proof of (\ref{thm:stable:stable}) that all fibers over $S^\circ(c^*\Nef(K_r))$ are non-normal, completing the proof in this case.

\smallskip
\textbf{Case V}: the bogus cone case where $Y'\simeq\bbP^1$.
We have $\gamma = c^*\Nef(K') + \bbZ_{\geq 0}[Y]$.
Since $\gamma$ has codimension one and $\Pic(K'/\oK)\simeq\Pic(\bbP^1)$ has rank 1, $\Pic(Y)$ has rank 3; it follows that $f\colon Y\to\bbP^1$ has a unique singular fiber.
We consider the restriction to the closed stratum $(\cX,\cE) \to S(\gamma) = \bbP^1$, and show that it is not constant.
By \cref{sec:extension_full} and case III in the proof of (\ref{thm:stable:stable}), the fibers over the open stratum $S^\circ(\gamma)$ are unions of two normal surfaces, meeting along a smooth $\bbP^1$
(except when $D_H$ is irreducible, then the normalization is integral and the conductor is a smooth $\bbP^1$).
Let $F = E_1 + E_2 \subset Y$ be the singular fiber of $f$.
We can flop either $E_i$ and obtain $K \dasharrow K_i$, which induces $Y \to Y_i$ contracting $E_i$.
There is a regular divisorial contraction $c_i\colon K_i \to K'_i$ with exceptional locus $r_i\colon Y_i \to \bbP^1$, a smooth $\bbP^1$-bundle, and then a regular contraction $K'_i \to K'$, with exceptional locus $\bbP^1$, the flopped curve $E_i' \subset K_i$.
Then $\gamma \subset \gamma_{E_i} \coloneqq c_i^*\Nef(K'_i) + \bbZ_{\geq 0}[Y]$, where $\gamma_{E_i}$ are maximal cones and the corresponding 0-strata are the 0-strata in the closed $1$-stratum $S$.

We claim that at least one of the $E_i \subset Y$ is a boundary divisor (i.e.\ an irreducible component of $D \subset Y$).
Otherwise, by \cref{prop:secondary_fan_del_Pezzo} and \cref{rem:Sec_on_PicY}, the two cones $c_i^*\Nef(K'_i) \in \MovFan(K)$ lie in the same cone of $\MovSec(K)$, then $\gamma$ would not have the same dimension as the minimal cone of $\MovSec(K)$ that contains it, a contradiction.

Without loss of generality say $E_1 \subset D$, and we consider its flop.
Let $\uLambda_1$ be the dual complex of the central fiber of $K_1 \to \bbA^1$, obtained from $\uLambda$ by flopping the edge corresponding to $E_1$.
The fiber over the corresponding 0-stratum $S(\gamma_{E_1})$ (which we recall from \cref{sec:extension_full} is a fiber over a point of $S^\circ(c_1^*\Nef(K'_1))$ has a double curve corresponding to $E_1'$
(i.e.\ to the {\it new} edge of $\uLambda_1$).
From the description of the fibers over the open stratum $S^\circ(\gamma)$ above, this double curve is smoothed over $S^\circ(\gamma)$, since $E_1'$ is contracted under $K'_1 \to K_1$.

This completes the proof of \cref{cl:restriction_to_1-stratum} and thus of \cref{thm:stable}.

\subsection{The mirror family for \texorpdfstring{$\bbP^1$}{P1}-fibrations} \label{sec:mirror_family_P1-fibration}

Here is a detailed analysis for the case III in the proof of \cref{thm:stable}.
We use the comparison \cref{prop:alg_equal} to borrow results from \cite[\S 6]{Gross_Mirror_symmetry_for_log_Calabi-Yau_surfaces_I_v1}.

Let $(Y,D)$ be a Looijenga pair with a fibration $f\colon Y \to \bbP^1$ with generic fiber $\bbP^1$.
Assume that $f$ is smooth over $\bbA^1 \setminus 0$.

\begin{notation}
	For the rest of this subsection, we replace $(Y,D)$ and $f$ by their restrictions to $\bbA^1\subset\bbP^1$.
	Write $D = D_H + D_F$, where $D_F$ consists of irreducible components contained in fibers, and $D_H$ consists of horizontal components (which is either a union of two sections, or an irreducible component necessarily of degree 2 over the base).
	We assume that $D_F$ lies over $0\in\bbA^1$.
	Note that $D$ is a cycle of rational curves if $D_H$ is irreducible; but if $D_H$ is reducible, then $D$ is a chain of rational curves, with two non-proper ends, the components of $D_H$.
\end{notation} 

Since \cite[\S 6]{Gross_Mirror_symmetry_for_log_Calabi-Yau_surfaces_I_v1} does not treat the case of irreducible $D_H$, we have to do a bit of analysis in this case.

Let $B$ be the support of the dual cone complex $\Sigma_{(Y,D)}$, with a canonical \Zaffine structure determined by the self-intersection numbers of the irreducible components of $D_F$.
If $D_H$ is irreducible, let $B'$ be obtained from $B$ by cutting along the ray of $\Sigma_{(Y,D)}$ corresponding to $D_H$.
Let $q\colon B'\to B$ be the quotient map identifying the two boundary rays of $B'$.
When $D_H$ is reducible we take $B'\coloneqq B$.
Note that $B'$ is isomorphic (as \Zaffine manifold with boundary) to a convex cone in $\bbR^2$, which is a half space if and only if the map $f\colon Y\to\bbA^1$ is toric.

Let $R\coloneqq\bbC[\NE(Y/\bbA^1)]$, $S\coloneqq\Spec(R)$, 
$A$ the free $R$-module with basis $B(\bbZ)$, and $A'$ the free $R$-module with basis $B'(\bbZ)$.
The construction of \cite[\S 6]{Gross_Mirror_symmetry_for_log_Calabi-Yau_surfaces_I_v1} gives
an $R$-algebra structure on $A'$ using the formalism of scattering diagram and broken lines
(this construction depends only on the formal neighborhood of the fiber $D_F$, so it makes sense also in the case where $D_H$ is irreducible).
The scattering diagram is finite (there are only finitely many rays, and the
attached functions are polynomials, as opposed to formal power series), and does not contain either of the boundary rays of $B'$, so it can be equivalently viewed as a scattering diagram on $B$.
Then the exact same structure constants (using broken lines) gives an $R$-algebra structure on $A$.
Let $I \subset A$ be the free $R$-submodule with basis $\theta_b$ for $b\in B(\bbZ)$ not on a ray associated to any component of $D_H$.

\begin{proposition} \label{prop:normalize}
	Let $A \to A'$ be the map of $R$-modules sending \[\theta_b \mapsto\allowbreak \sum_{b' \in q^{-1}(b)} \theta_{b'}.\]
	This is an inclusion of $R$-algebras, and is an integral extension.
	The submodule $I$ is an ideal of both $A$ and $A'$.
	For any $\theta_b\in I$, we have an isomorphism of the localizations $A_{\theta_b} \simeq A'_{\theta_b}$.
\end{proposition}
\begin{proof}
	The proposition is easy to check if $Y\to\bbA^1$ is toric, so we may assume that $B'$ is strictly convex.
	The multiplication rules given by the scattering diagrams imply that $A\to A'$ is an inclusion of $R$-algebras.
	Moreover, the strict convexity of $B'$ implies:
	\begin{claim}
		For $a,b\in B'(\bbZ)$, if $r\in B'(\bbZ)\setminus 0$ lies on a boundary ray, and $\theta_r$ has nonzero coefficient in the product $\theta_a\cdot\theta_b$, then $a,b$ must lie on the same boundary ray as $r$.
		In particular, $\theta_a \cdot \theta_b \in A \subset A'$ unless $a,b$ lie on a same boundary ray.
	\end{claim}

Now take $b \in B(\bbZ)$ lying on (the image of) a boundary ray, and write $q^{-1}(b) = b_1 + b_2$.
Since $\theta_b=\theta_{b_1}+\theta_{b_2}$, we have
\[
\theta_{b_1} \cdot \theta_b = \theta_{b_1}^2 + \theta_{b_1} \cdot \theta_{b_2}.
\]
The claim implies that $\theta_{b_1}\cdot\theta_{b_2}\in A\subset A'$.
It follows that $\theta_{b_1} \in A'$ is integral over $A$, and thus $A \subset A'$ is an
integral extension.
The claim also implies that $I$ is an ideal of both $A$ and $A'$, and gives the localization statement.
This completes the proof.
\end{proof}

Let $\cV' \coloneqq \Spec(A')$ and $\cV \coloneqq \Spec(A)$.
Let $\cC' \subset \cV'$, $\cC \subset \cV$ be the subschemes given by the ideal $I$.
Note $\cC' = q^{-1}(\cC)$, and $q\colon\cV'\to\cV$ is an isomorphism outside $\cC$.

\begin{proposition} \label{prop:P1_fibration}
For any fibers $(V',C'),(V,C)$ over the open stratum (i.e.\ the structure torus) $S^\circ \subset S$, $V'$ is normal, $V' \setminus C' \simeq V \setminus C$ are canonical, and the pairs $(V,C)$, $(V',C')$ have no 0-dimensional log-canonical centers.
\end{proposition}
\begin{proof}
Consider first the statement for $(V',C')$.
For this we can assume $D_H$ is reducible, as $A'$ depends only on the formal neighborhood of the fiber $D_F$.
If $D_F$ is a full fiber of $f$ then $f\colon(Y,D) \to \bbA^1$ is toric,
$\cV' \to S$ is also toric, and the computation is easy.
If $D_F$ is not the full fiber, then $D_F$ has negative-definite intersection matrix,
and $f$ factors through $Y \to Y'$ contracting $D_F$ to a point.
The open stratum $S^\circ\simeq T_{\Pic(Y/\bbA^1)}$ is covered by translates of $T_{\Pic(Y'/\bbA^1)} \subset T_{\Pic(Y/\bbA^1)}$, and the restriction of $(\cV',\cC')$ to such cosets is described explicitly as a hypersurface, in the first displayed formula on \cite[p.\ 136]{Gross_Mirror_symmetry_for_log_Calabi-Yau_surfaces_I_v1}.
From this the proposition for $(V',C')$ follows, and implies the statement for $(V,C)$, since by \cref{prop:normalize}, $q\colon V' \to V$ is the normalization, and $V\simeq V'/\sim$, where $\sim$ is the involution of the smooth curve $C'$ interchanging the theta functions $\theta_{b_1},\theta_{b_2}$ as in the proof of \cref{prop:normalize}. 
\end{proof}

\todo{Paul, can you add some details to the above proof? Can we give a more precise description of exactly what $(V',C')$ is, and how $(V,C)$ is obtained from it?}

\subsection{Speculative strategy for the general conjecture} \label{sec:conjectural_construction}

Here we sketch briefly how we expect the general \cref{conj:main} can be proven, by running the mirror machine twice to obtain the desired universal family.

Let $(Y,D,H)$ be an element in the moduli space $Q$ of \cref{conj:smooth}.
Let $R_Y\coloneqq\bbC[\NE(Y,\bbZ)]$, $\fm\subset R_Y$ the maximal monomial ideal and $\hR_Y$ the completion along $\fm$.
Intersecting with $H$ gives
\[R_Y\twoheadrightarrow \bbC[\bbN]\eqqcolon R_H\quad\text{and}\quad \hR_Y\twoheadrightarrow\hR_H.\]
Let $A_Y$ be the free $\hR_Y$-module with basis $\Sk(U,\bbZ)$ where $U\coloneqq Y\setminus D$.

\begin{conjecture} \label{conj:construction}
	There is a canonical formal $\hR_Y$-algebra structure on $A_Y$.
	We denote the restriction of $\Spf A_Y$ to $\Spf\hR_H$ by $X$.
	It has trivial canonical bundle and has canonical singularities.
	Let $q\colon\hX\to X$ be a crepant terminal resolution, $\hX_0$ the central fiber over $\Spf\hR_H$, $R_q\coloneqq\bbC[\NE(\hX/X,\bbZ)]$ and $A_q$ the free $R_q$-module with basis the integer points in the cone $\Gamma$ over the dual complex $\Lambda$ of $\hX_0$.
	Then there is a canonical graded $R_q$-algebra structure on $A_q$, and we denote $\cY\coloneqq\Proj A_q$.
	Let $\cD\subset\cY$ be the zero locus of the ideal generated by the integer points in the interior of $\Gamma$, and $\Theta\subset\cY$ the zero locus of the sum of basis elements corresponding to the integer points in $\Lambda$.
	The family $(\cY,\cD,\Theta)$ over $\TV(\Nef(q))\simeq\Spec R_q$ extends canonically to a family over a complete toric variety $\TV(\Sec(q))$, where $\Sec(q)$ generalizes the construction of the secondary fan in \cref{sec:secondary_fan}.
	For $0<\epsilon\ll 1$, $(\cY,\cD+\epsilon\Theta)\to\TV(\Sec(q))$ is a family of stable pairs, and the induced map $\TV(\Sec(q))\to\SP$ is finite, with image the closure $\oQ\subset\SP$.
\end{conjecture}

\begin{remark} \label{rem:relation_with_conjectural_construction}
	Here is what we believe to be the connection between the conjectural construction above, and the mirror construction we use in this paper:
	Let $Y$ be Fano, and $(X,E)$ mirror to $(Y,D)$.
	For simplicity let us assume both are smooth (e.g.\ when $(Y,D)$ is a smooth del Pezzo with a cycle of $(-1)$-curves).
	Then (conjecturally) $V\coloneqq X \setminus E$ occurs as a fiber of the mirror family (from \cite{Keel_Yu_The_Frobenius}) over $\TV(\Nef(Y))$, say over $v \in T_{\Pic(Y)}$.
	Let $\Gm  \subset \bbA^1 \subset \TV(\Nef(Y))$ be the orbit closure of $v$ for the subgroup $\Gm$ of the boundary torus action (see \cite[\S 16]{Keel_Yu_The_Frobenius}) corresponding to $-K_Y$.
	We expect that the restriction of the mirror family $\cV \to \bbA^1$ is $\oK_X \to \bbA^1$ (as in \cref{nota:K}).
	Note $K \to \oK$ gives a crepant resolution.
	This restriction is deformation equivalent to the restriction to $\Spf\hR_H$ as in the first step of the conjectural construction, so we expect the mirror family we construct in this paper is a special case of the family from the second step of the conjectural construction above.
	\end{remark}

\bibliographystyle{plain}
\bibliography{dahema}

\def\cprime{$'$}
\begin{thebibliography}{10}

\bibitem{Alexeev_Moduli_spaces_MgnW}
Valery Alexeev.
\newblock Moduli spaces {$M_{g,n}(W)$} for surfaces.
\newblock In {\em Higher-dimensional complex varieties ({T}rento, 1994)}, pages
  1--22. de Gruyter, Berlin, 1996.

\bibitem{Alexeev_Complete_moduli}
Valery Alexeev.
\newblock Complete moduli in the presence of semiabelian group action.
\newblock {\em Ann. of Math. (2)}, 155(3):611--708, 2002.

\bibitem{Batyrev_Dual_polyhedra_and_mirror_symmetry}
Victor~V. Batyrev.
\newblock Dual polyhedra and mirror symmetry for {C}alabi-{Y}au hypersurfaces
  in toric varieties.
\newblock {\em J. Algebraic Geom.}, 3(3):493--535, 1994.

\bibitem{Berkovich_Spectral_theory}
Vladimir~G. Berkovich.
\newblock {\em Spectral theory and analytic geometry over non-{A}rchimedean
  fields}, volume~33 of {\em Mathematical Surveys and Monographs}.
\newblock American Mathematical Society, Providence, RI, 1990.

\bibitem{Berkovich_Vanishing_cycles_for_formal_schemes}
Vladimir~G. Berkovich.
\newblock Vanishing cycles for formal schemes.
\newblock {\em Invent. Math.}, 115(3):539--571, 1994.

\bibitem{Birkar_Existence_of_minimal_models_for_varieties_of_log_general_type}
Caucher Birkar, Paolo Cascini, Christopher~D. Hacon, and James McKernan.
\newblock Existence of minimal models for varieties of log general type.
\newblock {\em J. Amer. Math. Soc.}, 23(2):405--468, 2010.

\bibitem{Ducros_La_structure_des_courbes_analytiques}
Antoine Ducros.
\newblock La structure des courbes analytiques.
\newblock En cours de r\'edaction, date du 15/11/2012.

\bibitem{Fresnel_Rigid_analytic_geometry_and_its_applications}
Jean Fresnel and Marius van~der Put.
\newblock {\em Rigid analytic geometry and its applications}, volume 218 of
  {\em Progress in Mathematics}.
\newblock Birkh\"auser Boston, Inc., Boston, MA, 2004.

\bibitem{Fulton_Introduction_to_toric_varieties}
William Fulton.
\newblock {\em Introduction to toric varieties}, volume 131 of {\em Annals of
  Mathematics Studies}.
\newblock Princeton University Press, Princeton, NJ, 1993.
\newblock The William H. Roever Lectures in Geometry.

\bibitem{Fulton_Intersection_theory}
William Fulton.
\newblock {\em Intersection theory}, volume~2 of {\em Ergebnisse der Mathematik
  und ihrer Grenzgebiete. 3. Folge. A Series of Modern Surveys in Mathematics}.
\newblock Springer-Verlag, Berlin, second edition, 1998.

\bibitem{Gelfand_Discriminants}
I.~M. Gelfand, M.~M. Kapranov, and A.~V. Zelevinsky.
\newblock {\em Discriminants, resultants, and multidimensional determinants}.
\newblock Mathematics: Theory \& Applications. Birkh\"{a}user Boston, Inc.,
  Boston, MA, 1994.

\bibitem{Gross_Mirror_symmetry_for_log_Calabi-Yau_surfaces_I_v1}
Mark Gross, Paul Hacking, and Sean Keel.
\newblock Mirror symmetry for log {C}alabi-{Y}au surfaces {I}.
\newblock {\em arXiv preprint arXiv:1106.4977v1}, 2011.

\bibitem{Gross_Mirror_symmetry_for_log_Calabi-Yau_surfaces_I_published}
Mark Gross, Paul Hacking, and Sean Keel.
\newblock Mirror symmetry for log {C}alabi-{Y}au surfaces {I}.
\newblock {\em Publ. Math. Inst. Hautes \'{E}tudes Sci.}, 122:65--168, 2015.

\bibitem{Gross_Moduli_of_surfaces}
Mark Gross, Paul Hacking, and Sean Keel.
\newblock Moduli of surfaces with an anti-canonical cycle.
\newblock {\em Compos. Math.}, 151(2):265--291, 2015.

\bibitem{Gross_K3}
Mark Gross, Paul Hacking, Sean Keel, and Bernd Siebert.
\newblock Theta functions for {K}3 surfaces.
\newblock {\em preprint}, 2018.

\bibitem{Gross_Mirror_symmetry_via_logarithmic_degeneration_data_I}
Mark Gross and Bernd Siebert.
\newblock Mirror symmetry via logarithmic degeneration data. {I}.
\newblock {\em J. Differential Geom.}, 72(2):169--338, 2006.

\bibitem{Gross_Intrinsic_mirror_symmetry_announcement}
Mark Gross and Bernd Siebert.
\newblock Intrinsic mirror symmetry and punctured {G}romov-{W}itten invariants.
\newblock In {\em Algebraic geometry: {S}alt {L}ake {C}ity 2015}, volume~97 of
  {\em Proc. Sympos. Pure Math.}, pages 199--230. Amer. Math. Soc., Providence,
  RI, 2018.

\bibitem{Gross_Intrinsic_mirror_symmetry}
Mark Gross and Bernd Siebert.
\newblock Intrinsic mirror symmetry.
\newblock {\em arXiv preprint arXiv:1909.07649}, 2019.

\bibitem{EGA4-2}
A.~Grothendieck.
\newblock \'{E}l\'{e}ments de g\'{e}om\'{e}trie alg\'{e}brique. {IV}. \'{E}tude
  locale des sch\'{e}mas et des morphismes de sch\'{e}mas. {II}.
\newblock {\em Inst. Hautes \'{E}tudes Sci. Publ. Math.}, (24):231, 1965.

\bibitem{Hatcher_Algebraic_topology}
Allen Hatcher.
\newblock {\em Algebraic topology}.
\newblock Cambridge University Press, Cambridge, 2002.

\bibitem{Hu_Mori_dream_spaces}
Yi~Hu and Sean Keel.
\newblock Mori dream spaces and {GIT}.
\newblock volume~48, pages 331--348. 2000.
\newblock Dedicated to William Fulton on the occasion of his 60th birthday.

\bibitem{Kapranov_Quotients_of_toric_vareities}
M.~M. Kapranov, B.~Sturmfels, and A.~V. Zelevinsky.
\newblock Quotients of toric varieties.
\newblock {\em Math. Ann.}, 290(4):643--655, 1991.

\bibitem{Keel_Geometry_of_Chow_quotients}
Sean Keel and Jenia Tevelev.
\newblock Geometry of {C}how quotients of {G}rassmannians.
\newblock {\em Duke Math. J.}, 134(2):259--311, 2006.

\bibitem{Keel_Yu_The_Frobenius}
Sean Keel and Tony~Yue Yu.
\newblock The {F}robenius structure theorem for affine log {C}alabi-{Y}au
  varieties containing a torus.
\newblock {\em arXiv preprint arXiv:1908.09861}, 2019.

\bibitem{Kollar_Threefolds_and_deformations}
J.~Koll\'{a}r and N.~I. Shepherd-Barron.
\newblock Threefolds and deformations of surface singularities.
\newblock {\em Invent. Math.}, 91(2):299--338, 1988.

\bibitem{Kollar_Moduli_of_varieties_of_general_type}
J\'{a}nos Koll\'{a}r.
\newblock Moduli of varieties of general type.
\newblock In {\em Handbook of moduli. {V}ol. {II}}, volume~25 of {\em Adv.
  Lect. Math. (ALM)}, pages 131--157. Int. Press, Somerville, MA, 2013.

\bibitem{Kollar_Singularities_of_the_minimal_model_program}
J\'{a}nos Koll\'{a}r.
\newblock {\em Singularities of the minimal model program}, volume 200 of {\em
  Cambridge Tracts in Mathematics}.
\newblock Cambridge University Press, Cambridge, 2013.
\newblock With a collaboration of S\'{a}ndor Kov\'{a}cs.

\bibitem{Kollar_Log_canonical_singularities}
J\'{a}nos Koll\'{a}r and S\'{a}ndor~J. Kov\'{a}cs.
\newblock Log canonical singularities are {D}u {B}ois.
\newblock {\em J. Amer. Math. Soc.}, 23(3):791--813, 2010.

\bibitem{Kollar_Birational_geometry_of_algebraic_varieties}
J\'anos Koll\'ar and Shigefumi Mori.
\newblock {\em Birational geometry of algebraic varieties}, volume 134 of {\em
  Cambridge Tracts in Mathematics}.
\newblock Cambridge University Press, Cambridge, 1998.
\newblock With the collaboration of C. H. Clemens and A. Corti, Translated from
  the 1998 Japanese original.

\bibitem{Looijenga_Rational_surfaces}
Eduard Looijenga.
\newblock Rational surfaces with an anticanonical cycle.
\newblock {\em Ann. of Math. (2)}, 114(2):267--322, 1981.

\bibitem{Mandel_Theta_bases}
Travis Mandel.
\newblock Theta bases and log {G}romov-{W}itten invariants of cluster
  varieties.
\newblock {\em arXiv preprint arXiv:1903.03042}, 2019.

\bibitem{Miller_Combinatorial_commutative_algebra}
Ezra Miller and Bernd Sturmfels.
\newblock {\em Combinatorial commutative algebra}, volume 227 of {\em Graduate
  Texts in Mathematics}.
\newblock Springer-Verlag, New York, 2005.

\bibitem{Naruki_Cross_ratio_variety}
Isao Naruki.
\newblock Cross ratio variety as a moduli space of cubic surfaces.
\newblock {\em Proc. London Math. Soc. (3)}, 45(1):1--30, 1982.
\newblock With an appendix by Eduard Looijenga.

\bibitem{Nicaise_Xu_Yu_The_non-archimedean_SYZ_fibration}
Johannes Nicaise, Chenyang Xu, and Tony~Yue Yu.
\newblock The non-archimedean {SYZ} fibration.
\newblock {\em Compos. Math.}, 155(5):953--972, 2019.

\bibitem{Serhiyenko_Cluster_structures}
M.~Serhiyenko, Sherman-Bennett and L.~Williams.
\newblock Cluster structures in schubert varieties in the grassmannian.
\newblock {\em arXiv preprint arXiv:1902.00807}, 2019.

\bibitem{Stacks_project}
The {Stacks Project Authors}.
\newblock {Stacks Project}.
\newblock \url{http://stacks.math.columbia.edu}, 2013.

\bibitem{Thuillier_Geometrie_toroidale}
Amaury Thuillier.
\newblock G\'eom\'etrie toro\"\i dale et g\'eom\'etrie analytique non
  archim\'edienne. {A}pplication au type d'homotopie de certains sch\'emas
  formels.
\newblock {\em Manuscripta Math.}, 123(4):381--451, 2007.

\bibitem{Yu_Enumeration_of_holomorphic_cylinders_I}
Tony~Yue Yu.
\newblock Enumeration of holomorphic cylinders in log {C}alabi--{Y}au surfaces.
  {I}.
\newblock {\em Math. Ann.}, 366(3-4):1649--1675, 2016.

\bibitem{Yu_Enumeration_of_holomorphic_cylinders_II}
Tony~Yue Yu.
\newblock Enumeration of holomorphic cylinders in log {C}alabi--{Y}au surfaces,
  {II} : {P}ositivity, integrality and the gluing formula.
\newblock {\em Geom. Topol.}, 25(1):1--46, 2021.

\end{thebibliography}

\end{document}